\documentclass[11pt,reqno]{article}
\usepackage{amsmath}
\usepackage{mathrsfs}
\usepackage{amsfonts}
\usepackage{amssymb}

\usepackage{amssymb}
\usepackage{amsthm}
\usepackage{graphicx}              
\usepackage{amsmath}               
\usepackage{amsfonts}              
\usepackage{amsthm}                
\usepackage{setspace}
\usepackage{epstopdf}
\usepackage{epsfig}

\newtheorem{theorem}{Theorem}[section]
\newtheorem{proposition}{Proposition}[section]
\newtheorem{lemma}{Lemma}[section]
\newtheorem{corollary}{Corollary}[section]
\newtheorem{remark}{Remark}[section]
\newtheorem{definition}{Definition}[section]

\renewcommand{\thefootnote}{\fnsymbol{footnote}}
\makeatletter

\newcommand{\Rmnum}[1]{\expandafter\@slowromancap\romannumeral #1@}
\makeatother

\allowdisplaybreaks
\numberwithin{equation}{section}

\title{\bf{Asymptotic stability of large energy harmonic maps under the wave map from 2D hyperbolic spaces to 2D hyperbolic spaces}}
\author{
{Ze Li}
}
\date{}

\begin{document}

\maketitle

\renewcommand{\thefootnote}{\fnsymbol{footnote}}


\noindent{\bf Abstract}
In this paper, we prove that the large energy harmonic maps from $\Bbb H^2$ to $\Bbb H^2$ are asymptotically stable under the wave map equation.

\bigskip

\section{Introduction}
Let $(\Bbb R \times M,{\mathbf{\Lambda}})$ be a given Lorentz manifold with metric $ \mathbf{\Lambda}$, where $(M,\mathbf{h})$ is a m-dimensional Riemannian manifold without boundary.
Let $(N,\mathbf{g})$ be a n-dimensional complete Riemannian manifold without boundary. A wave map $u:\Bbb R\times M\to N$ is a formal critical point of the Lagrangian functional $\mathcal{L}(u)$ defined by
\begin{align}\label{uyghv}
\mathcal{L}(u):=\int_{\Bbb R}\int_{M}({\bf D}{u},{\bf D}{u}){\rm{dvol_M}}dt.
\end{align}
where the integrand of (\ref{uyghv}) is the norm of ${\bf{D} }u$ viewed as a section of the vector bundle $T^*(\Bbb R\times M)\otimes u^*TN$ with metric ${\bf\Lambda}^{-1}\otimes u^*\mathbf{g}$. Under the local coordinate system  $\{x^i\}^m_{i=1}$ for $M$ and $\{y^k\}^{n}_{k=1}$ for $N$ respectively, $({\bf{D} }u,{\bf{D} }u)$ is given by
\begin{align}
h^{ij}(\partial_iu,\partial_ju)_{u^*\mathbf{g}}-(\partial_tu,\partial_tu)_{u^*\mathbf{g}},
\end{align}
where $\mathbf{h}=h_{ij}dx^idx^j$, $\mathbf{g}=g_{ij}dy^idy^j$, ${\bf \Lambda}=-dtdt+h_{ij}dx^idx^j$ are the metric tensors for $M,N$ and $\Bbb R\times M$.

In the above local coordinate, the Euler-Lagrange equation for  (\ref{uyghv}) is
\begin{align}\label{z32ccc344fyeo}
\Box u^l +{\Lambda^{\alpha \beta }}\overline{\Gamma}_{kp}^l(u){\partial _\alpha }{u^k}{\partial _\beta }{u^p}= 0,
\end{align}
where $\alpha,\beta$ run over $0,1,...,m$. Moreover, $\Box=-\partial_t^2+\Delta_M$ is the D' Alembertian on $\Bbb R\times M$, $\overline{\Gamma}^k_{ij}(u)$ are the
Christoffel symbols at the point $u(t,x)\in N$. In this paper, we consider the case $M=\Bbb H^2$, $N=\Bbb H^2$.

The wave map equation on flat spacetimes known as the nonlinear $\sigma$-model, arises as a model problem in particle physics and is
related to the general relativity, see for instance \cite{MS,IK,L}. Moreover, the background of hyperbolic spaces is of particular interest since the anti-de Sitter space model is asymptotically hyperbolic.

In this paper, we aim to study the stability of harmonic maps under the wave map equation (\ref{z32ccc344fyeo}) without size restriction of the harmonic map or equivariant assumptions. This work is on the way to a more vast project called soliton resolution conjecture (SRC) for dispersive PDEs which claims that every solution with bounded trajectory in energy space either splits into the superposition of divergent solitons with a radiation part plus an asymptotically vanishing remainder as $t\to\infty$ or converges to divergent solitons with a regular weak limit as $t$ approaches the blow-up time. The SRC reduces the dynamic behaviors of arbitrary data  to dynamics near bubbles or multi-solitons.  In this paper, we focus on dynamics near stationary solutions of wave maps, i.e. stability/instabity for harmonic maps. The instability of ground state of equivariant energy critical wave maps was shown by Cote \cite{Cote}. Later, a codimension-2 stability of 1-equivariant energy critical wave maps was proved by Bejenaru-Krieger-Tataru \cite{BKT}. And a series of work done by Lawrie-Oh-Shahshahani \cite{FSWE3,FSWE376,6H8UR,EFSWE376} studied the stability and SRC for equivariant wave maps on hyperbolic planes. Moreover,
\cite{EFSWE376} raised the following conjecture for wave maps from $\Bbb R\times \Bbb H^2$ to $\Bbb H^2$, \\
{\bf Conjecture 1.1}  Suppose that $\{Q_{\nu}\}$ is the 1-equivariant harmonic map class parameterized by $\nu\in(0,1)$ from $\Bbb H^2$ to $\Bbb H^2$. Let $(u_0,u_1)$ be the finite energy initial data to the wave map equation
for $u:\Bbb R\times \Bbb H^2\to \Bbb H^2$, and let $u_0(x)=Q_{\nu}(x)$ outside of some compact subset of $\Bbb H^2$.
Then the unique solution $(u(t),\partial_tu(t))$ to the wave map equation scatters to $(Q(x),0)$ as $t\to\infty$.

In this paper, we consider the case when the initial data are perturbations of the large energy harmonic maps.  Before stating
our main result, we recall the notion of admissible harmonic maps used in our previous works \cite{EFE376} and \cite{LZ} where small harmonic maps are considered.
\begin{definition}\label{dywetu67tu68}
Denote the Poincare disk by  $\Bbb D$.
We say the harmonic map $Q:\Bbb D\to\Bbb D$ is admissible if $\overline{Q(\Bbb D)}$ is a compact subset of $\Bbb D$ covered by a
geodesic ball centered at the origin of radius $R_0$, $\|\nabla^kdQ\|_{L^2}<\infty$ for $k=0,1,2,3$,
and there exists some $\varrho>0$ such that $e^{\varrho r}|dQ|^2\in L^{\infty}$, where $r$ is the distance between $x\in \Bbb D$ and
the origin.
\end{definition}
\noindent{\bf{Remark 1.1}}(Examples for the admissible harmonic maps) Any analytic function $f:\Bbb C \to \Bbb C$ with
$\overline{f(\Bbb D)}\Subset\Bbb D$ is an admissible harmonic map. See [Appendix,\cite{LZ}] for the proof.

For any given admissible harmonic map $Q$, the work space $\mathcal{H}^k\times\mathcal{H}^{k-1}$ is defined by (\ref{aseh8sdfed97}). Our main theorem is as
follows.
\begin{theorem}\label{a1}
Let $Q$ be an admissible harmonic map in Definition \ref{dywetu67tu68}. Assume that the initial data $(u_0,u_1)\in {\mathcal{H}^3\times\mathcal{H}^2}$ to
(\ref{z32ccc344fyeo}) with $u_0:\Bbb H^2\to\Bbb H^2$, $u_1(x)\in T_{u_0(x)}N$ for each $x\in \Bbb H^2$ satisfy
\begin{align}\label{as3}
\|(u_0,u_1)-(Q,0)\|_{{\mathcal{H}}^2\times {\mathcal{H}^1}}<\mu_1,
\end{align}
Then if $\mu_1>0$ is sufficiently small, (\ref{z32ccc344fyeo}) has a global solution $(u(t),\partial_tu(t))$ and as $t\to\infty$ we have
$$
\mathop {\lim }\limits_{t\to\infty }\mathop {\sup }\limits_{x \in {\mathbb{H}^2}} {d_{{\mathbb{H}^2}}}\left( {u(t,x),Q(x)} \right) = 0.
$$
\end{theorem}

\footnotetext{${ }^1$ For targets with non-positive sectional curvature, the uniqueness is unconditional. For positive sectional curvature targets, one has to assume the image $Q(M)$ is contained in a ball of radius less than a constant depending on the curvature.}
\noindent{\bf{Remark 1.2}}
\noindent We remark that the perturbation norms in Theorem 1.1 assume the initial data tends to $Q$ at infinity. By the conditional uniqueness${ }^1$ of
harmonic maps with prescribed boundary map, one can expect the final asymptotic harmonic map of the solution $(u,\partial_tu)$ to
(\ref{z32ccc344fyeo}) is exactly $Q$. This is one key reason for why the asymptotic harmonic map coincide with the unperturbed one, which
is different from wave maps from $\Bbb R^{1+2}$ to $S^n$ where we have moving and modulated solitons after
the perturbation.  The other key reason is that the bubble tree convergence seems to imply that the solution converges to the superposition of one harmonic map $Q_{\infty}:\Bbb H^2\to\Bbb H^2$ and finite numbers of scaled and translated harmonic maps from $\Bbb R^2$ to $\Bbb H^2$. Meanwhile since finite energy harmonic maps from $\Bbb R^2$ to $\Bbb H^2$ are trivial, one can expect the solution to (\ref{z32ccc344fyeo}) converges to only one bubble.

\noindent{\bf{Remark 1.3}}(Examples for the perturbations of admissible harmonic maps) Since one has the global coordinates (\ref{tyw55gvu}) for $\Bbb H^2$, it is trivial to give an example of the perturbation in the sense of (\ref{as3}).

\noindent{\bf{Remark 1.4}}
The initial data considered in this paper are perturbations of harmonic maps in the $\bf H^2$ norm. One shall build the $S_k$ v.s. $N_k$ norm constructed by Tataru \cite{6HZD} and Tao \cite{6HD} in the hyperbolic setting while considering perturbations in the energy critical norm ${\bf H}^1$.

\subsection{Outline of the proof and main ideas.}

We first describe the outline of the proof. By constructing Tao's caloric gauge in our setting, one obtains the nonlinear wave equation
for the heat tension field. Separating the ``effective" linear part from the nonlinear terms yields a magnetic wave equation.
By establishing the Kato smoothing effect for the master linear equation, one obtains the corresponding
non-endpoint Strichartz estimates. Applying an abstract theorem built in our
work \cite{6HDFRG} gives us the endpoint Strichartz estimates and a key weighted Strichartz estimate. Meanwhile, we prove the smoothing effect for the linear heat equation with large magnetic potential.
By bootstrap, the endpoint and weighted Strichartz estimates, one can prove the heat tension filed enjoys a global space-time norm. Transforming the bounds of the heat tension field back to the differential
fields closes the bootstrap and thus finishing the whole proof. The caloric gauge used here was previously built in \cite{LZ,EFE376} where we used
caloric gauge as a geometric linearization.

One of the main contribution of this paper is that we use the freedom of the gauge fixed on the harmonic map and the geometric meaning of the master linear equation to rule out the bottom resonance and the possibility of eigenvalues in the gap $[0,1/4]$. The first observation is the two freedoms of the Schr\"odinger operator studied here: The Schr\"odinger operator varies as the gauge fixed on the harmonic map changes; the Schr\"odinger operator is invariant under the coordinates transformation of $M=\Bbb H^2$. The other observation is that the Schr\"odinger operator is indeed well-defined on the pullback bundle $Q^*(TN)$, where $Q$ is the harmonic map, which enables us to work in a purely geometric setting. In fact, given any frame $\{\Xi_1,\Xi_2\}$ on $Q^*(TN)$, suppose that $A=A_idx^i$ is the corresponding connection one form. Then any $\Bbb C^2$ valued function $f:=(f_1,f_2)^t$ defined on $M$ induces a complex vector field $f\Xi$ on $N$ by
\begin{align}
f\mapsto f\Xi=f_1\Xi_1+f_2\Xi_2.
\end{align}
Then the potential part of Schr\"odinger operator $H:=-\Delta+W$ can be written as
\begin{align}
Wf=-2(A,df)+(d^*A)f-(A,A)f+Sf,
\end{align}
where $(\cdot,\cdot)$ denotes the metric tensor for one forms on $M$, and $S$ is a symmetric linear mapping in $\Bbb C^2$ defined on $M$ related to the sectional curvature. Since the connection coefficients $A_i$ are antisymmetric and real, $-(A,A)$ defines a non-negative symmetric operator in $L^2(M,\Bbb C^2)$.  And due to the non-positive sectional curvature of the target $N=\Bbb H^2$, $S$ is  non-negative as well. Meanwhile, integration by parts implies $H$ is symmetric. Thus, the somewhat bad term for determining the spectrum especially whether there exists bottom resonance  is $-2(A,df)+(d^*A)f$. But since $A$ depends on the frame fixed on $Q^*(TN)$, one may take the Coulomb gauge to simplify the determination. Fortunately, this idea works well in our setting. And besides fixing Coulomb gauge, it is important to
do calculations by using the covariant derivatives on $Q^*TN$, which matches the geometric structure of $H$ well, rather than just viewing $f$ as $\Bbb C^2$ valued functions.

The main difficulty for the large energy harmonic map case is to derive the Kato smoothing effect of a wave equation with large magnetic potentials, which can be further divided into the small frequency, mediate frequency and high frequency part.
The enemy for the small frequency part is the possibility of bottom resonance. We use the Coulomb gauge on the harmonic map to obtain a nice spectrum distribution of the operator $(V+X)(-\Delta-\frac{1}{4}\pm i\epsilon)^{-1}$, where the matrix valued function $V$ denotes the electric potential part and the vector field $X$ denotes the magnetic field respectively. In fact, by choosing the Coulomb gauge we have the spectrum of  $(V+X) (-\Delta-\frac{1}{4}\pm i\epsilon)^{-1}$ lies on the right of the imaginary axis, then the resonance can be ruled out by a perturbation argument using the Riesz projection operators.  Moreover, We exclude the possible existence of eigenvalues in $(-\infty,1/4)$ by calculating the numerical range of the magnetic Schr\"odinger operator by using covariant derivatives on $Q^*(TN)$.

The high frequency part is always difficult in the large magnetic potential case, even in the Euclidean case, see for instance \cite{EGS}.  In our argument, we split the magnetic potential into a large long range part supported outside some geodesic ball and a remainder part supported near the original point. For the long rang part, we can put the magnetic Schr\"odinger operator uniformly bounded in the weighted space $w(x)L^2$ for all high frequencies by a similar positive commutator method of \cite{CCV}. The important gain of this energy argument is the weight $w^{-1}$ can be chosen to vanish near the origin point.  Due to the extra smallness gain from the vanishing of $w^{-1}$ and the closeness to the origin of the support of the remainder potential, we can view the Schr\"odinger operator with the whole magnetic
potential as the perturbation of the long range Schr\"odinger operator.

In the large energy case, the smoothing effect for magnetic heat equations is also needed. The $L^p-L^q$ estimates of $e^{-tH}$ is relatively easy by noticing that the geometric structure of $H$ shows $|e^{-tH}f|\le e^{t\Delta}|f|$ holds point-wisely. Then the $L^p-L^q$ estimates of $e^{-tH}$ follow directly from the known results for $e^{t\Delta}$.   The main difficulty here is to derive the smoothing estimates, which cannot be transferred to the corresponding ones for $e^{t\Delta}$ as $L^p-L^q$ estimates.  We use ideas from semigroups of linear operators to deduce the smoothing estimates. In fact, by the Laplacian transform formula connecting the resolvent with the heat semigroup, the resolvent estimates of $H$ for part regime of the resolvent set follow by that of $e^{-tH}$. Using the almost equivalence technique used in our previous paper \cite{6HDFRG} and frequency decomposition we get the smoothing effect for $e^{-tH}$ by interpolation.

\subsection{History}

In the following, we recall the non-exhaustive lists of results on the Cauchy problem, the long dynamics and blow up for wave maps on $\Bbb R^{1+m}$. The sharp subcritical well-posedness theory was developed by Klainerman-Machedon \cite{RTEW1,RTEW2} and Klainerman-Selberg
\cite{TEW2}. The critical well-posedness theory in equivariant case was considered by Christodoulou, Tahvildar-Zadeh \cite{CT},
Shatah, Tahvildar-Zadeh \cite{STZ2} and improved by Chiodaroli-Krieger-Luhrmann \cite{CaaKaaLaa} in the radial case. The critical small data global well-posedness theory was started by the breakthrough work of Tataru \cite{6HZD}
and Tao \cite{8HG,6HD}, see also \cite{Ja1a,EaWa3a,KaRa3a, NaSaUa,Tataru3} for generalizations of the targets.  The below threshold critical global well-posedness theory was
obtained by  Krieger-Schlag \cite{6HZDR}, Sterbenz-Tataru \cite{ST1,ST2}, Tao \cite{O6hour}. The bubbling theorem
in the equivariant case was obtained by Struwe \cite{S}. The type II blow up solutions for equivariant energy critical wave maps were constructed by
Krieger-Schlag-Tataru \cite{E3PYT}, Raphael-Rodnianski \cite{O6H8UR}, and Rodnianski-Sterbenz \cite{6HR}. For the SRC on energy critical wave maps/hyperbolic Yang-Mills in the equivariant case, the pioneering works of Cote \cite{FE376COTE}, Jia-Kenig \cite{JK} obtained results along some time sequence and recently Jendrej-Lawrie \cite{JL} constructed the two bubble solution by studying corresponding threshold solutions.  For the SRC on energy critical wave maps to spheres in the non-equivariant case, see the works of Grinis \cite{G} and  Duyckaerts-Jia-Kenig-Merle \cite{6HDFGVFESDEREC}. We also mention the works \cite{JYGT1,JYGT2} for outer-ball wave maps and \cite{Gy} for wave maps on wormholes in the equivaraint case.

Wave map equations on curved spacetime were relatively less understood. D'Ancona-Zheng \cite{6HZDRR} studied critical small data global well-posedness of wave maps on rotationally symmetric manifolds in the equivariant case. The critical small data global well-posedness theory for wave maps on small asymptotically flat perturbations of $\Bbb R^4$ was studied by Lawrie \cite{O6HI8UR}. The long time dynamics for wave maps on $\Bbb R\times \Bbb H^2$ in the equivariant class were studied by sequel works of Lawrie, Oh, Shahshahani \cite{FSWE3,FSWE376,EFSWE376}.
And Lawrie, Oh, Shahshahani \cite{6H8UR} obtained the critical small data global well-posedness theory for
wave maps from $\Bbb R\times \Bbb H^d$ to compact Riemann manifolds with $d\ge4$.

This paper is organized as follows. In Section 2, we recall some results obtained in our previous works. Particularly, we recall the
work space and the existence of the caloric gauge. In addition, we
prove the limit harmonic map for the heat flow is exactly the unperturbed one. In Section 3 to Section 5, we recall master equation and prove the corresponding Kato smoothing effects.  In Section 6, we prove the smoothing estimates for the magnetic heat equation and recall  Strichartz estimates for magnetic wave equations.  In Section 7,  by bootstrap we deduce the global
spacetime bounds for the heat tension field and finish the proof of Theorem 1.1.

\section{\bf Notations and Preliminaries}

\subsection{Hyperbolic Planes}
In this paper, we consider the simplest class of Riemanniann symmetric spaces of noncompact type, i.e., the hyperbolic plane $\Bbb H^2$. Recall that $\Bbb H^2$ can be realized as the hyperboloid in $\Bbb R^{1+2}$:
\begin{align}\label{fdc}
\left\{
  \begin{array}{ll}
   -|x_0|^2+|x_1|^2+|x_2|^2=1  & \hbox{ } \\
    (x_0,x_1,x_2)\in \Bbb R^{1+2} , x_0\ge 1& \hbox{ }
  \end{array}
\right.
\end{align}
with metric being the pullback of Minkowski metric $(-1,1,1)$ in $\Bbb R^{1+2}$. In the geodesic coordinates of $\Bbb H^2$, the Riemannian metric is written as
\begin{align*}
dr^2+(\sinh r)^2d\theta^2,
\end{align*}
and the Laplace-Beltrami operator is given by
\begin{align*}
\Delta=\partial^2_r+\coth r\partial_r+(\sinh r)^{-2}\partial^2_{\theta}.
\end{align*}
For $\lambda\in \Bbb C$, the spherical function $\varphi_{\lambda}$ is the radial normalized eigenfunction of $\Delta$:
\begin{align*}
\Delta \varphi_{\lambda}=-(\lambda^2+\frac{1}{4})\varphi_{\lambda},  \mbox{  }\varphi_{\lambda}(0)=1.
\end{align*}
In the case $ \lambda\in \Bbb R$, $r\ge 0$,  $\varphi_{\lambda}$ satisfies the bound:
\begin{align}\label{tby}
\left|{\varphi_{\lambda}(r)}\right|\le {\varphi_{0}(r)}\lesssim (1+r)e^{-\frac{r}{2}}.
\end{align}

The hyperbolic plane can also be realized as the Poincare disk:
\begin{align*}
   {\Bbb D}=\{ (x_1,x_2)\in \Bbb R^{2} : |x_1|^2+|x_2|^2< 1\}
\end{align*}
with the metric
\begin{align*}
  \frac{4dx^2_1+4dx_2^2}{1-|x_1|^2-|x_2|^2}.
\end{align*}
The volume form is given by
\begin{align*}
 dz= 4(1-|z|^2)^{-1}dz,
\end{align*}
where $z=x_1+ix_2$ is the complex coordinate.

In addition, as a homogeneous space $\Bbb H^2$ can be viewed as  $G/K$, where $K=SO(2)$ is the rotation group in $\Bbb R^2$ and  $G={\bf SU}(1,1)$ is  defined by
\begin{align*}
{\bf SU}(1,1) = \left\{ {\left( {\begin{array}{*{20}{c}}
   a & c  \\
   {\bar{c}} & {\bar{a}}  \\
 \end{array} } \right):{{\left| a \right|}^2} - {{\left| c \right|}^2} = 1} \right\}
\end{align*}
$G$ acts on $\Bbb D$ in the following way:
\begin{align*}
g: z\in \Bbb D\longmapsto \frac{az+c}{\bar{c}z+\bar{a}}\in \Bbb D.
\end{align*}
Let $dg$ be the normalized Haar measure of the group $G=SU(1,1)$ such that
\begin{align*}
\int_{G}f(g.o)dg=\int_{\Bbb D} f(z)dz,
\end{align*}
where $o$ denotes the original point of $\Bbb D$. Recall that  the convolution of functions $f_1,f_2$ on $\Bbb D$ are defined by
\begin{align}\label{u6trf}
f_1*f_2(z)=\int_{G}f_1(g.o)f_2(g^{-1}z)dg.
\end{align}
The convolution is symmetric, i.e., $f_1*f_2=f_2*f_1$. And if $f_1$ or $f_2$ is a bi-$K$ invariant function (radial function in our case) i.e.,
\begin{align*}
f(z)={\bf f}({\rm distance}(z,o))
\end{align*}
for some function ${\bf f}$ defined on $\Bbb R^+$, then  (e.g. $f_2$ is radial )
\begin{align}\label{6trf}
f_1*f_2(z)=\int_{\Bbb D}f_1(\widetilde{z}){{\bf f}_2(d(z,\widetilde{z}))}d\widetilde{z}.
\end{align}
Recall that $SL_2(R)=\mathbf{N}\mathbf{A}\mathbf{K}$ where $\mathbf{A}$ is the group of diagonal matrices with determine 1:
\begin{align*}
\mathbf{A }= \left\{ {{a_c} =\left( {\begin{array}{*{20}{c}}
   c & 0  \\
   0 & {{c^{ - 1}}}  \\
 \end{array} } \right):c \in \Bbb R*} \right\},
\end{align*}
$\mathbf{N}$ is the unipotent group of matrices:
\begin{align*}
{\bf N} = \left\{ {{n_b} = \left( {\begin{array}{*{20}{c}}
   1 & b  \\
   0 & 1  \\
 \end{array} } \right):b \in \Bbb R} \right\},
\end{align*}
$\mathbf{K}$ denotes rotation group $SO(2)$ as before.
Identifying ${\bf SU}(1,1)=ZSL_2(R)Z^{-1}$ with
\begin{align*}
Z = \frac{1}
{{\sqrt 2 }}\left\{ {\left( {\begin{array}{*{20}{c}}
   1 & { - {\mathbf{i}}}  \\
   0 & {\mathbf{i}}  \\
 \end{array} } \right)} \right\},
\end{align*}
the decomposition $SL_2(R)=\mathbf{N}\mathbf{A}\mathbf{K}$ induces the Iwasawa decomposition of $G={\bf {SU}}(1,1)=Z\mathbf{NAK}Z^{-1}$.

Back to the hyperboloid model (\ref{fdc}), Iwasawa decomposition can be written as ${\Bbb G}={\Bbb N}{\Bbb A}{\Bbb K}$ where ${\Bbb G}=SO(d,1)$ is the connected Lie subgroup of $GL(3)$ that keeps Minkowsi metric and ${\Bbb N},{\Bbb A}$ are given by
\begin{align}\label{tvu}
{\Bbb A} = \left\{ {{a_{{v_1}}} = \left( {\begin{array}{*{20}{c}}
   {\cosh {v_1}} & {\sinh {v_1}} & 0  \\
   {\sinh {v_1}} & {\cosh {v_1}} & 0  \\
   0 & 0 & 1  \\
 \end{array} } \right):{v_1} \in \Bbb R} \right\},
\end{align}
and
\begin{align}\label{t5vu}
\mathbb{N} = \left\{ {{n_{{v_2}}} = \left( {\begin{array}{*{20}{c}}
   {1 + \frac{1}
{2}|{v_2}{|^2}} & { - \frac{1}
{2}|{v_2}{|^2}} & {{v_2}}  \\
   {\frac{1}
{2}|{v_2}{|^2}} & {1 - \frac{1}
{2}|{v_2}{|^2}} & {{v_2}}  \\
   {{v_2}} & { - {v_2}} & 1  \\
 \end{array} } \right):{v_2} \in \Bbb R} \right\}.
\end{align}
This induces a global coordinate system by
the diffeomorphism $\Phi:\Bbb R\times \Bbb R\to \mathbb{H}^2$ given by
\begin{align*}
\Phi: (v_1,v_2)\longmapsto n_{v_1}a_{v_2}.o
\end{align*}
or explicitly written as
\begin{align}\label{tyw55gvu}
\Phi: (v_1,v_2)\longmapsto({\cosh} v_2+\frac{1}{2}e^{-v_2}|v_1|^2, {\sinh} v_2+\frac{1}{2}e^{-v_2}|v_1|^2, e^{-v_2}v_1).
\end{align}
Then the Riemannian metric of $\Bbb H^2$ now is
\begin{align}\label{t1}
{\bf h}=e^{-2v_2}(dv_1)^2+(dv_2)^2.
\end{align}
The corresponding Christoffel symbols are
\begin{align}\label{rex3456gvh}
\Gamma^1_{2,1}=-1, \mbox{  }\Gamma^2_{1,1}=e^{-2v_2},\mbox{ }\Gamma^2_{2,2}=\Gamma^1_{2,2}=\Gamma^1_{1,1}=\Gamma^2_{2,1}=0.
\end{align}

In addition, the analogy of (\ref{u6trf}) and (\ref{6trf}) is
\begin{align}\label{u6trf}
f_1*f_2({  x})=\int_{\Bbb G}f_1(g.o)f_2(g^{-1}{  x})dg,
\end{align}
and
\begin{align}\label{67trf}
f_1*f(x)=\int_{\Bbb H^2}f_1(\widetilde{x}){{\bf f}(d(x,\widetilde{x}))}{\rm dvol}_{\widetilde{x}},
\end{align}
provided that $f$ is a bi-$\Bbb K$ invariant function (radial function in our case).

In the following, we denote $(x_1,x_2)$ instead of $(v_1,v_2)$ for the coordinate given by (\ref{tyw55gvu}) for $M=\Bbb H^2$. And the coordinates for the target manifold $N=\Bbb H^2$  induced by (\ref{tyw55gvu}) are  denoted by $(y_1,y_2)$.

There exists a natural orthonormal frame at given point $y\in N$ given by
\begin{align}\label{ex3456gvh}
\Omega_1(y)=e^{y_2}\frac{\partial}{\partial y_1}; \mbox{  }\mbox{  }\Omega_2(y)=\frac{\partial}{\partial y_2}.
\end{align}

\subsection{Fourier transform and Sobolev embedding}

In this subsection, we recall the Fourier analysis on $\mathbb{H}^2$ (see Helgason \cite{6HDFGV}). The following is a sketch rather than a complete introduction, one may see Section 2 of our previous work \cite{6HDFRG} for a more detailed introduction.
For $x,y\in \Bbb R^{1+2}$, denote the Minkowski metric by $[x,y]=-x_1y_1+x_2y_2+x_3y_3$. Given any $b\in \Bbb S^1$ and $\tau\in \Bbb C$, define $\mathbf{k}(b)=(1,b)\in \Bbb R^{1+2}$. Let
\begin{align*}
h_{\tau,b}:\mathbb{H}^2\to \Bbb C, \mbox{  }h_{\tau,b}=[x,\mathbf{k}(b)]^{i\tau-\frac{1}{2}}.
\end{align*}
Given any $g\in C_0(\mathbb{H}^2)$, the Fourier transform is defined by
\begin{align}\label{t54re}
\mathcal{F}{g}\left( {\tau,b } \right)= \int_{\Bbb H^2} g(x)[ x,\mathbf{k}(b )]^{i\tau - \frac{1}{2}}dx.
\end{align}
Denote $c(\lambda)$ the Harish-Chandra c-function on $\mathbb{H}^2$.  For some constant $C$ it is defined by
$c(\tau)=C\frac{\Gamma(i\tau)}{\Gamma(\frac{1}{2}+i\tau)}.$
Then the corresponding Fourier inversion formula is
\begin{align*}
g(x) = \int_0^\infty  \int_{\Bbb S^1}\mathcal{F}g( \tau ,b)[ x,\mathbf{k}(b ) ]^{ - i\tau - \frac{1}{2}}|c(\tau )|^{- 2} dbd\tau.
\end{align*}
The Plancherel theorem is given by
\begin{align*}
\int_{{\Bbb{H}^2}} {f(x)\overline {g(x)}{\rm{dvol_h}} = \frac{1}{2}} \int_{\mathbb{R} \times {\Bbb S^1}} {\mathcal{F}f( {\tau ,b })\overline {\mathcal{F}g( {\tau ,b})} {| {c(\tau )} |^{ - 2}}d\tau db}.
\end{align*}
Any function $m:\Bbb R\to \Bbb C$ defines a Fourier multiplier operator $m(-\Delta)$ by
\begin{align}\label{vcdf}
\mathcal{F}\left(m(-\Delta)g\right)(\tau, b)=m(\frac{1}{4}+\tau^2)\mathcal{F}{g}(\tau,b).
\end{align}
The fractional derivatives $(-\Delta)^{\frac{s}{2}}$ are defined by the Fourier multiplier $\lambda\to (\frac{1}{4}+\lambda^2)^{\frac{s}{2}}$.

The Sobolev spaces ${\mathrm{H}}^{s,p}$ are defined by
\begin{align*}
{\mathrm{H}}^{s,p}(\Bbb H^2)=(-\Delta )^{-\frac{s}{2}}L^p(\Bbb H^2), \mbox{  }(1<p<\infty, \mbox{  }s\in \Bbb R).
\end{align*}
Moreover, for $s=n\in \Bbb N,$ ${\mathrm{H}}^{s,p}(\Bbb H^2)$ coincides with
\begin{align*}
W^{n,p}(\Bbb H^2)=\{f\in L^p(\Bbb H^2): |\nabla^k f|\in L^p(\Bbb H^2), \mbox{ }\forall\mbox{ } 0\le k\le n\}.
\end{align*}
where $\nabla $ is the covariant derivative on $\Bbb H^2$. The Sobolev inequalities of functions in $\mathrm{H}^{s,p}$ are recalled in Appendix A.
And it is known that $C^{\infty}_c(\Bbb {H}^2)$ is dense in $W^{n,p}(\mathbb{H}^2)$.

For radial functions $k$ the Fourier transform (\ref{t54re}) also has the identity
\begin{align*}
\mathcal{F}(f*k)=(\mathcal{F}f)\cdot (\mathcal{F}k).
\end{align*}
Hence, given Fourier multiplier $T$ of symbol $m(\lambda)$ defined via (\ref{vcdf}), by (\ref{67trf}) we can write $T$ by its Schwartz kernel as
\begin{align*}
Tf (x) =\int_{\Bbb H^2}M({\rm{distance}}(x,\widetilde{x}))f(\widetilde{x}){\rm dvol_h}.
\end{align*}
We call the function  $M(x,\widetilde{x}):=M({\rm{distance}}(x,\widetilde{x}))$ the kernel of $T$ or Green functions for $T$ sometimes. For instance, the kernel of the free resolvent $(-\Delta-s(1-s))^{-1}$ in $\Bbb H^{d}$ is given by (see (\ref{7mvbtr}))
\begin{align*}
[{ }^d\widetilde{R}]_0(s;x,y)=(2\pi)^{-\frac{d}{2}}e^{-i\pi\mu}(\sinh r)^{-\mu}Q^{\mu}_{\nu}(\cosh r),
\end{align*}
and one may see Lemma \ref{Tomnagdfe4} for estimates of kernels of resolvent.

\subsection{Function Spaces for Maps between $\Bbb H^2$ }

Denote $\widetilde{\nabla}$ the covariant derivative in $TN$, and  $\nabla$ the covariant derivative induced by $u$ in $u^*(TN)$. Denote $\mathbf{R}$  the Riemann curvature tensor of $N$.  $\Gamma^{k}_{ij}$ and  $(\overline{\Gamma}^{k}_{ij})$ denote the Christoffel symbols on $M$ and $N$ respectively.

For $X,Y,Z\in TN$, we adopt the following notation for simplicity
\begin{align}\label{ex346gvh}
\left( {X \wedge Y} \right)Z = \left\langle {X,Z} \right\rangle Y - \left\langle {Y,Z} \right\rangle X.
\end{align}
Then the Riemannian curvature on $N=\Bbb H^2$ can be written as
\begin{align*}
{\bf R}(X,Y)Z={\widetilde{\nabla} _X}{\widetilde{\nabla} _Y}Z - {\widetilde{\nabla }_Y}{\widetilde{\nabla} _X}Z - {\widetilde{\nabla} _{[X,Y]}}Z = \left( {X \wedge Y} \right)Z.
\end{align*}

For maps $u:\mathbb{H}^2\to \mathbb{H}^2$, we define the intrinsic Sobolev semi-norm $\mathfrak{H}^n$ by
\begin{align*}
\|u\|^2_{{\mathfrak{H}}^n}=\sum^n_{k=1}\int_{\mathbb{H}^2} |\nabla^{k-1} du|^2 {\rm {dvol_h}}.
\end{align*}
Given map $u:\mathbb{H}^2\to\mathbb{H}^2$, (\ref{tyw55gvu}) equips it with a vector-valued function $x\longmapsto (u^1(x),u^2(x))$ defined via
\begin{align*}
\Phi(u^1(x),u^2(x))=u(x)
\end{align*}
for any $x\in\Bbb H^2$. Let $Q:\Bbb H^2\to\Bbb H^2$ be an admissible harmonic map in Definition \ref{dywetu67tu68}. Then the extrinsic Sobolev space is defined by
\begin{align}\label{sderrtg897}
{\bf{H}^n_{Q}}=\{u: u^1-Q^1(x), u^2-Q^2(x)\in \mathrm{H}^n(\Bbb {H}^2;\Bbb C)\},
\end{align}
where $\Phi(Q^1(x),Q^2(x))=Q(x)$. We equip ${\bf{H}^n_{Q}}$ with the following distance
\begin{align}\label{sderrh97}
{\rm{dist}}_{n,Q}(v,w)=\sum^2_{p=1}\|v^p-w^p\|_{{W}^{n,2}(\Bbb{H}^2;\Bbb C)}, \mbox{  }\forall v,w\in \bf{H}^n_{Q}.
\end{align}
Denote
$$\mathcal{D}=\{u:M\to N \mbox{  }{\rm is}\mbox{  }{\rm smooth}\mbox{ }\big| \mbox{ }\exists K\Subset M\mbox{  }{\rm such } \mbox{  }{\rm that} \mbox{  }u=Q, \forall x\in M\backslash K \}.
$$
Let $\mathcal{H}^n_{Q}$ be the completion of $\mathcal{D}$ under the metric given by (\ref{sderrh97}).
Then $\bf H^n_Q$ coincides with $\mathcal{H}_Q^n$ by the density of $C^{\infty}_c$ in $W^{n,2}$. And we write $\mathcal{H}^n$ for simplicity.
For maps from $\Bbb R\times\Bbb H^2$ to $\Bbb H^2$, we define the space $\mathcal{H}^n\times\mathcal{H}^{n-1}$ by
\begin{align}\label{aseh8sdfed97}
{\mathcal{H}^n\times\mathcal{H}^{n-1}}=\left\{u \big| \sum^2_{k=1}\|u^k-Q^k\|_{\mathrm{H}^n(\Bbb {H}^2;\Bbb R)}+\|\partial_tu^k\|_{\mathrm{H}^{n-1}(\Bbb {H}^2;\Bbb R)}<\infty\right\},
\end{align}
with the distance given by
\begin{align}\label{46hour}
{\rm{dist}}_{\mathcal{H}^n\times\mathcal{H}^{n-1}}(u,v)=\sum^2_{k=1}\|\partial_tu^k-\partial_t v^k\|_{\mathrm{H}^{n-1}}+\|u^k-v^k\|_{\mathrm{H}^n}.
\end{align}

The local and conditional global well-posedness of (\ref{z32ccc344fyeo}) in $\mathcal H^3\times\mathcal H^2$ are recalled in Proposition \ref{vskfheccc344fyeo} in Appendix A, see \cite{EFE376} for its proof.

\subsection{Gauges}

Let the latin letters denote the index of $x\in \Bbb H^2$, and let the Greek letter $\alpha=0,1,2$ denote the index of $t,x$. Throughout the paper we use the index $0$ to stand for $t$. Let $\{e_1(t,x),e_2(t,x)\}$ be an orthonormal frame for $u^*(TN)$. The corresponding connection coefficients are antisymmetric matrices:
\begin{align*}
[{A_\alpha}]^k_j = \left\langle {{\nabla_\alpha}{e_j},{e_k}} \right\rangle.
\end{align*}
Denote $\phi_\alpha=(\phi^1_\alpha,\phi^2_\alpha)$ the components of $\partial_{t,x}u$ under the frame $\{e_1,e_2\}$:
\begin{align*}
\phi_\alpha^j =\langle {\partial _\alpha}u,{e_j} \rangle.
\end{align*}
For any given $\Bbb R^2$-valued function $\phi$ defined on $[0,T]\times\Bbb H^2$, associate $\phi$ with a vector filed $ \phi e $
\begin{align*}
\phi e:=\sum^2_{j=1}\phi^je_j.
\end{align*}
Then the covariant derivative induced by $u$ on the trivial complex vector bundle over $[0,T]\times\Bbb H^2$ with fiber $\Bbb C^2$ is defined by
$$D_\alpha\phi=\partial_\alpha \phi+[A_\alpha]\phi,
$$
which in the form of components reads as,
\begin{align*}
\big(D_\alpha\phi\big)^k=\partial_{\alpha}\phi^k+\sum^2_{j=1}[A_{\alpha}]^k_j\phi^j.
\end{align*}
It is easy to check the torsion free identity,
\begin{align}\label{pknb}
D_\alpha\phi_\beta=D_\beta\phi_\alpha.
\end{align}
and the commutator identity (in the two dimensional case ( 2-d is Abel))
\begin{align}\label{x346gvh}
\big([D_\alpha,D_\beta]\phi\big)e=\big((\partial_\alpha A_\beta-\partial_\beta A_\alpha)\phi\big) e=\mathbf{R}(u)(\partial_\alpha u, \partial_\beta u)( \phi e ).
\end{align}
\noindent{\bf{Remark 2.1}}
With a little abuse of notation, for $\bf{a},\bf{b},\bf{c}\in \Bbb R^2$, we define a matrix valued function $\bf{a}\wedge\bf{b}$ by
\begin{align}\label{detu7tu68}
(\bf{a}\wedge\bf{b})\bf{c}=\langle {\bf{a},\bf{c}} \rangle\bf{b} - \langle {\bf{b},\bf{c}}\rangle \bf{a}.
\end{align}
By letting $X=a^ke_k$, $Y=b^ke_k$, $Z=c^ke_k$, it is easy to see (\ref{detu7tu68}) coincide with (\ref{ex346gvh}).  Hence, (\ref{x346gvh}) can be written as
\begin{align}\label{detu67tu68}
[D_\alpha,D_\beta]\phi=(\phi_\alpha\wedge\phi_{\beta})\phi
\end{align}
\begin{lemma}
Using the above notations, (\ref{z32ccc344fyeo}) can be written as
\begin{align*}
D_t\phi_t-h^{jk}D_j\phi_k+h^{jk}\Gamma^l_{jk}\phi_l=0¡£
\end{align*}
\end{lemma}

\subsection{The existence of Caloric Gauge}
Let $\mathcal{X}_T$ denote the space of maps from $[0,T]\times \Bbb H^2$  to $\Bbb H^2$ which satisfies
\begin{align*}
(u(t),\partial_tu)\in C([0,T];\mathcal{H}^{3} \times \mathcal{H}^{2}).
\end{align*}

When non-trivial harmonic maps occur, Tao's caloric gauge can be defined as follows.
\begin{definition}\label{7894obhxdg}
Assume $u(t,x)\in\mathcal{X}_T$ is a solution to (\ref{z32ccc344fyeo}). Suppose that the heat flow with initial data $u_0$ converges to some harmonic map $Q:\mathbb{H}^2\to \mathbb{H}^2$. Then for a given orthonormal frame ${\bf e}(x)\triangleq\{{\bf e}_j(Q(x))\}^2_{j=1}$ on $Q^*(TN)$, a caloric gauge is the couple of a map $\widetilde{u}:\Bbb R^+\times [0,T]\times\mathbb{H}^2\to\Bbb H^2$ and an orthonormal frame $\Theta\triangleq\{\Theta_j(\widetilde{u}(s,t,x))\}^2_{j=1}$ which satisfies
\begin{align*}
(i)&\partial_s\widetilde u= \tau (\widetilde u)  \\
(ii)&\nabla_s\Theta_j = 0 \\
(iii)&\mathop {\lim }\limits_{s \to \infty } \Theta_j = {\bf e}_j,
\end{align*}
where the convergence of the frame is defined by
\begin{align}
 \mathop {\lim }\limits_{s \to \infty } \widetilde{u}(s,t,x) &= Q(x)\nonumber \\
 \mathop {\lim }\limits_{s \to \infty } \langle {\Theta_i},{\Omega _j}\rangle \upharpoonright_{\widetilde{u}(s,t,x)} &= \langle {{\bf e}_i}(Q(x)),\Omega _j(Q(x))\rangle.\label{kanghytrfv}
\end{align}
\end{definition}

The equation $(i)$ in Definition \ref{7894obhxdg} is called the heat flow equation. Roughly speaking, caloric gauge means transposing the frame fixed on the bundle $Q^*TN$ parallel along the heat flow to the original map $u(t,x)$.

Recall the dynamic heat flow from $\mathbb{H}^2$ to $\mathbb{H}^2$ with a parameter $t\in [0,T)$
\begin{align}\label{dt67vftu68}
\left\{ \begin{array}{l}
 {\partial _s}\widetilde{u} = \tau (\widetilde{u}) \\
 \widetilde{u}(s,t,x) \upharpoonright_{s=0}= u(t,x) \\
 \end{array} \right.
\end{align}

The long time existence of the heat flow from $\Bbb H^2\to\Bbb H^2$ in $\mathcal H^3$ is known, see our previous work \cite{EFE376}.
We summarize the long time and short time behaviors obtained in \cite{EFE376} as a proposition.
\begin{proposition}\label{fcuyder47te}
Suppose that  $u:[0,T]\times\Bbb H^2\to \Bbb H^2$ is a solution to (\ref{z32ccc344fyeo}) satisfying
\begin{align}\label{KPIL}
\|(\nabla du,\nabla\partial_t u)\|_{L^2\times L^2}+\|(du,\partial_t u)\|_{L^2\times L^2}\le C(M_1),
\end{align}
Denote $\widetilde{u}:\Bbb R^+\times[0,T]\times\Bbb H^2\to \Bbb H^2$ the solution to (\ref{dt67vftu68}) with initial data $u(t,x)$. Then there exists some universal constant $\delta>0$ such that for $t\in[0,T]$, it uniformly holds that
\begin{align*}
&\| s^{\frac{1}{2}} \nabla d\widetilde{u} \|_{L_s^\infty[0,1] L_x^\infty } + \| s^{\frac{1}{2}}e^{\delta s}\nabla \partial _t\widetilde{u} \|_{L_s^\infty L_x^\infty} + \| se^{\delta s}\nabla {\partial _s}\widetilde{u}  \|_{L_s^\infty L_x^\infty } \\
&+ \| s^{\frac{1}{2}}e^{\delta s}\partial _s\widetilde{u}\|_{L_s^\infty L_x^\infty }
+\| d\widetilde{u}\|_{L_s^\infty[1,\infty) L_x^{\infty}}+\| \nabla d\widetilde{u}\|_{L_s^\infty[1,\infty) L_x^{\infty}}+\| \nabla d\widetilde{u}\|_{L_s^\infty L_x^2}\\
&+ \| s^{\frac{1}{2}}e^{\delta s}\nabla \partial _s\widetilde{u}\|_{L_s^\infty L_x^2} \le C(M_1).
\end{align*}
\end{proposition}

It has been prove in \cite{EFE376} that the heat flow initiated from all $u(t,x)$ for different $t$ converges to the same harmonic map say $\widetilde{Q}$.
We aim to prove $\widetilde{Q}$ is exactly $Q$ in the definition of our working space ${\mathcal{ H}}^k_Q$. But in our previous work \cite{EFE376}, this was only verified for small energy harmonic maps. Thus we give a new proof to involve the large energy case. The key ingredient is the differential inequality concerning the distance between two harmonic maps proved by \cite{JW}. We remark that the main theorem 1.1 of \cite{JW} cannot be directly applied to our case since it seems not easy to relate their boundary map setting to our working space in a reasonable and effective way.

\begin{lemma}\label{DERfder}
If $(u,\partial_tu)$ is a solution to (\ref{z32ccc344fyeo}) in $\mathcal{X}_T$, then as  $s\to\infty$,
$$
\mathop {\lim }\limits_{s \to \infty } \mathop {\sup }\limits_{(x,t) \in {\mathbb{H}^2} \times [0,T]}
dist_{\Bbb H^2}(\widetilde{u}(s,x,t),Q(x))=0.
$$
\end{lemma}
\begin{proof}
As remarked above we only need to verify $Q=\widetilde{Q}$. First we note that due to Corollary \ref{8908} and $Q,\widetilde{Q}\in {\mathcal H}^3_{Q}$, one has $\widetilde{Q}(M)$ and $Q(M)$ are contained in a geodesic ball of $N=\Bbb H^2$ with radius $R_1$. Hence the distance between $Q(x)$ and $\widetilde{Q}$ is equivalent to $|Q^1-\widetilde{Q}^1|+|Q^2-\widetilde{Q}^2|$ up to some large constant $C(R_1)$ depending on $R_1$.  [\cite{JW},Page 286] (the $\kappa=0$ case) has obtained the following inequality
\begin{align}
\Delta\big(\frac{1}{2}[{\rm {dist}}(Q(x),\widetilde{Q}(x))]^2\big)\ge 0.
\end{align}
Then the mean value inequality for nonnegative subharmonic functions yields
\begin{align}\label{p78v}
\frac{1}{2\pi}\int^{2\pi}_{0}[{\rm{dist}}(Q(r,\theta),\widetilde{Q}(r,\theta))]^2d\theta
\end{align}
is an nondecreasing function to $r\in(0,\infty)$.
If there exists some $r_0$ such that  (\ref{p78v}) is strictly positive when $r=r_0$, then integrating (\ref{p78v}) with respect to $r$ in $[r_0,\infty)$ gives
\begin{align}\label{bbbbdgd758}
\int^{\infty}_{r_0}\sinh r[{\rm{dist}}(Q(r,\theta),\widetilde{Q}(r,\theta))]^2d\theta dr=\infty.
\end{align}
But the left hand side of (\ref{bbbbdgd758}) is bounded by
$$
C(R_1)\int^{\infty}_0\sinh r|(Q^1,Q^2)-(\widetilde{Q}^1,\widetilde{Q}^2)|^2drd\theta=C(R_1)\|\widetilde{Q}\|^2_{{\bf H}^0_Q}<\infty,
$$
which contradicts with (\ref{bbbbdgd758}). Hence ${\rm{dist}}(\widetilde{Q,}Q)=0$.
\end{proof}

The existence of the caloric gauge defined in Definition \ref{7894obhxdg} is given below.
\begin{proposition}\label{94obhxdg}
Let $(u,\partial_tu)\in \mathcal{X}_T$ solve (\ref{z32ccc344fyeo}). Fixing any frame ${\bf e}\triangleq\{{\bf e}_1(Q(x)),{\bf e}_2(Q(x))\}$, there exists a unique caloric gauge defined in Definition \ref{7894obhxdg}.
\end{proposition}

The matrix valued connection coefficient $A_{x,t}$ can be expressed by the differential fields and the heat tension field.
\begin{lemma}
Let $\Theta(s,t,x)$ be the caloric gauge built in Proposition \ref{94obhxdg}, then for $i=1,2$, $s>0$ we have
\begin{align}
&A_i(s,t,x)\sqrt{h^{ii}(x)} = \int_s^\infty \sqrt{ h^{ii}(x)}\left( \phi _s\wedge{\phi _i}\right) ds' +\sqrt{h^{ii}(x)}A^{\infty}_i.\label{ober45xdg}\\
&A_t(s,t,x)=\int^{\infty}_s\left(\phi_s\wedge{\phi_t}\right) ds'.\label{obhxdg}
\end{align}
where $[A^{\infty}_i]^k_j= \sqrt{h^{ii}(x)}\langle \nabla _i{\bf e}_k(x),{\bf e}_j(x)\rangle$.
\end{lemma}

\begin{remark}\label{zdsru314rucuvc344fyeo}
Rewrite (\ref{ober45xdg}) as
$A_i(s,t,x)=A^{\infty}_i(s,t,x)+A^{qua}_i(s,t,x),$
where $A^{\infty}_i$ denotes the limit part, and $A^{qua}_i$ denotes the  quadratic part, i.e.,
\begin{align*}
A^{qua}_i=\int^{\infty}_s\phi_s\wedge\phi_ids'.
\end{align*}
Similarly, split $\phi_i$ into $\phi_i=\phi^{\infty}_i+\phi^{qua}_i$, where $
\phi^{qua}_i=\int^{\infty}_s\partial_s\phi_ids',$
and
\begin{align*}
\phi^{\infty}_i(x)={\left( {\left\langle {{\partial _i}Q ,{{\bf e} _1} } \right\rangle ,\left\langle {{\partial _i}Q ,{{\bf e}_2} } \right\rangle } \right)^t}\upharpoonright_{Q(x)}.
\end{align*}
\end{remark}

\subsection{Master Equation for Heat Tension Field}
Writing the heat flow equation under the gauge shows the heat tension filed $\phi_s$ satisfies
\begin{align}\label{zz2ccgsrucuvc344fyeo}
\phi_s=h^{jk}D_j\phi_k-h^{jk}\Gamma^l_{jk}\phi_l.
\end{align}
And define the wave tension filed as Tao by
\begin{align}\label{zzz32314rucuvc344fyeo}
Z = {D_t}\phi _t- {h^{jk}}D_j\phi _k + h^{jk}\Gamma _{jk}^l\phi _l.
\end{align}
which is indeed the gauged equation for the wave map (\ref{z32ccc344fyeo}).

\begin{lemma}[\cite{EFE376} ]
The wave tension field $Z$ defined by (\ref{zzz32314rucuvc344fyeo}) satisfies
\begin{align}
&{\partial _s}Z = \Delta Z + 2h^{ij}{A_i}{\partial _j}Z +h^{ij} {A_i}{A_j}Z + h^{ii}{\partial _i}{A_i}Z- {h^{ij}}\Gamma _{ij}^k{A_k}Z + {h^{ij}}\left( {Z \wedge {\phi _i}} \right){\phi _j} \nonumber\\
&\mbox{  }\mbox{  }\mbox{  }\mbox{  }\mbox{  }+ 3{h^{ij}}({\partial _t}\widetilde{u} \wedge {\partial _i}\widetilde{u}){\nabla _t}{\partial _j}\widetilde{u},\label{ab1}
\end{align}
where with abuse of notations, since the last term is an intrinsic quantity, the identity holds in the sense of equivalence:
\begin{align}
Z\leftrightarrow Z_1\Theta_1+Z_2\Theta_2;\mbox{ }(\Delta Z)\leftrightarrow (\Delta  Z_1)\Theta_1+(\Delta  Z_2)\Theta_2,\mbox{  }{\rm{and}}\mbox{ }{\rm{ so }}\mbox{ }{\rm{ on}}.
\end{align}
And we will adopt this convention in all the following sections without emphasizing the implicit frame $\{\Theta_i\}^{2}_{i=1}$ in the identities if both intrinsic and frame dependent quantities appear at the same time.
\end{lemma}

\begin{lemma}[\cite{EFE376} ]\label{refor45xdg}
Let $Q$ be an admissible harmonic map.
Fix the frame ${\bf e}$ in Remark \ref{zdsru314rucuvc344fyeo} by taking ${\bf e}(Q(x))=\Omega(Q(x))$ (see (\ref{tyw55gvu})). Then
\begin{align}
&|\sqrt{h^{ii}}A_i^{\infty}|+ |\sqrt{h^{ii}}\phi_i^{\infty}|\lesssim |dQ|\label{rer45xdg}\\
&|{h^{ii}}\left( {{\partial _i}{A^{\infty}_i} - \Gamma _{ii}^k{A^{\infty}_k}} \right)|\lesssim  |dQ|^2.\label{resder445xdg}
\end{align}
\end{lemma}

\begin{lemma}\label{REWSDERF}
Given any fixed frame ${\bf e}$ in Proposition \ref{94obhxdg}, we have
the heat tension filed $\phi_s$ satisfies
\begin{align*}
&(\partial^2_t-\Delta){\phi _s} + W{\phi _s} =  - 2A_t{\partial _t}\phi _s - A_tA_t\phi _s - \partial _tA_t\phi _s+ {\partial _s}Z + (\phi _t\wedge\phi _s)\phi _t \\
&- h^{jk}( \phi^\infty_j \wedge\phi _s  )\phi^{qua} _k - h^{jk}(\phi^{qua} _j\wedge\phi _s  )\phi^\infty _k- h^{jk}( \phi _j^{qua}\wedge \phi _s )\phi^{qua} _k\\
&+ 2h^{jk}A_j^{qua}\partial _k\phi _s
+ {h^{jk}}A_j^{qua}A^\infty_k {\phi _s} + {h^{jk}}A_j^\infty A^{qua}_k{\phi _s} + {h^{jk}}A^{qua}_jA^{qua}_k{\phi _s} \\
&+ {h^{jk}}( \partial _jA^{qua}_k- \Gamma _{jk}^lA^{qua}_l){\phi _s},
\end{align*}
where $A_{x}^{\infty}$, $A_{x}^{qua}$ are defined in Remark \ref{zdsru314rucuvc344fyeo}, and $W$ is given by
\begin{align}
\varphi\longmapsto W\varphi & :=  -2 h^{jk}A_j^\infty \partial_k\varphi  -{h^{jk}}A^\infty _jA^\infty _k\varphi  +{h^{jk}}\left( { \phi^\infty _j\wedge\varphi   } \right)\phi^\infty _k\nonumber\\
&-h^{jk}(\partial_jA^{\infty}_{k}-\Gamma^l_{jk}A^{\infty}_l)\varphi.\label{quxgdhfghka}
\end{align}
\end{lemma}

\section{Geometric setting of $H$ and Limiting Absorption Principle for Free Resolvent }

\subsection{Geometric Setting of Magnetic Schr\"odinger Operator $H$}

In Section 3 to Section 5 , we will prove the Kato smoothing effect for $H=-\Delta+W$ given in Lemma \ref{REWSDERF}.
First of all, we point out the operator $H$ is independent of the coordinates chosen for $M$.
\begin{proposition}
The operator $H$ define in Lemma \ref{REWSDERF} is independent of the coordinates chosen for $M$. Furthermore, $H$ is a symmetric operator in $L^2(\Bbb H^2;\Bbb C^2)$.
\end{proposition}
\begin{proof}
In fact, it is easy to see $\widetilde{A}\triangleq A^{\infty}_idx^i$ is the connection one form on $Q^{*}TN$, thus it is independent of the coordinates chosen for $M$. And for any given $\Bbb C^2$-valued function $\varphi$, we have
\begin{align}\label{geortyu3}
{h^{jk}}A_j^\infty {\partial _k}\varphi=( \widetilde{A},d\varphi),
\end{align}
where $(\cdot,\cdot)$ denotes the metric tensor for one forms on $M=\Bbb H^2$.
Moreover, it is easy to check
\begin{align*}
h^{jk}A_j^\infty A_k^\infty=(\widetilde{A},\widetilde{A}).
\end{align*}
Thus they are invariant under the transform of coordinates as well.
And for any given coordinate system, one has
\begin{align}\label{geortyu}
h^{jk}(\partial_jA^{\infty}_{k}-\Gamma^l_{kj}A^{\infty}_l)=-d^*\widetilde{A}.
\end{align}
Combining these three facts shows $H$ is independent of the coordinates chosen for $M$. Let the inner product in the definition of the wedge operator $\wedge$ given by  (\ref{detu7tu68}) be complex inner product, i.e.,
\begin{align}\label{gevcdo}
(\mathbf{a}\wedge \mathbf{b})\mathbf{c}=\langle\mathbf{a},\mathbf{c}\rangle_{\Bbb C^2}\mathbf{b}-\langle \mathbf{b},\mathbf{c} \rangle_{\Bbb C^2}\mathbf{a}.
\end{align}
Denote
\begin{align}\label{geo}
B \varphi=-h^{jk}A_j^\infty A_k^\infty\varphi+h^{jk}( \phi^\infty_j\wedge \varphi  )\phi _k^\infty.
\end{align}
Since $A^{\infty}_i$ is a real antisymmetric $2\times2$ matrix, $B$ is a real symmetric non-negative matrix valued function defined on $M=\Bbb H^2$.
Now $W\varphi$ is of the form $-2(\widetilde{A},d\varphi)+(d^*\widetilde{A})\varphi+B\varphi$.
Hence, by integration by parts (see e.g. Lemma 3.1 in \cite{6HDFRG} which is a companion paper of this work), $H=-\Delta+W$ is symmetric in $L^2(\Bbb H^2;\Bbb C^2)$.
\end{proof}

\subsection{Free Resolvent Estimates}

In this subsection, we prove estimates for free resolvent especially the limiting absorbing principle.

\subsection{Notations}
Denote $D=\sqrt{-\Delta}$.
Let $\mathfrak{{D}}=(-\Delta-\frac{1}{4})^{\frac{1}{2}}$  be the shifted differential operator.
We remark that the shifted operator $\mathfrak{D}$ is not equivalent to $D$ in $L^p$.

For two Banach spaces $X$ and $Y$, the space of bounded linear operators from $X$ to $Y$ is denoted by $\mathcal{L}(X,Y)$, and the operator norm
is denoted by $\|\cdot\|_{\mathcal{L}(X,Y)}$.  For a closed densely defined operator $T:X\to Y$, the resolvent of $T$ is denoted by $R_{T}(z):=(T-z)^{-1}$,  for instance
\begin{align}
R_{D}(z)=(D-z)^{-1}, R_{H}(z)=(H-z)^{-1}, R_{\sqrt{H}}(z)=(\sqrt{H}-z)^{-1}.
\end{align}
For simplicity denote $R_0(z)$ the free resolvent $(-\Delta-z)^{-1}$.

Given a variable, e.g.  $r\in \Bbb R^+$, $t\in [0,\infty)$, $j\in\Bbb Z$, we write $r^{-\infty}$  ( $t^{-\infty}$, $j^{-\infty}$)  to represent a function $f(r)$ (respectively $f(t),f(j)$) satisfying:\\
\noindent For any integer $n\ge 1$ there exists a constant $C(n)$ such that
\begin{align}\label{Ytfedrx}
\left\{
  \begin{array}{ll}
    |f(r)|\le C(n)r^{-n}\mbox{  } {\rm as }\mbox{ }{r\to\infty}& \hbox{ } \\
   {\rm or}\mbox{  }|f(t)|\le C(n)t^{-n}\mbox{  } {\rm as }\mbox{ }{t\to\infty} & \hbox{ } \\
   {\rm or}\mbox{  }|f(j)|\le C(n)|j|^{-n}\mbox{  } {\rm as }\mbox{ }{j\to\infty} & \hbox{ }
  \end{array}
\right.
\end{align}

A generalization of the Kunze-Stein phenomenon obtained by  [Lemma 5.1, \cite{45XDAP}] will also be widely used then.
We recall it below for reader's convenience.
\begin{lemma}[\cite{45XDAP}]\label{tywe55gvu}
Given any $2\le p,r<\infty$, for any function $h\in L^{p'}(\Bbb H^2)$ and radial function $g$ on $\Bbb H^2$, there holds that
$$
\left\| {g * h} \right\|_{{L^r}} \le C\left\|h \right\|_{L^{p'}}\left(\int_0^\infty  |\varphi_0|^{s_1}|g|^{s_2}\sinh rdr \right)^{1/P},
$$
where $\varphi_0(r)$ is the spherical function and satisfies the point-wise bound (\ref{tby}). $(s_1,s_2)$ are given by
\begin{align*}
s_1 = \frac{{2\min \{ r,p\} }}{{r+ p}}, s_2 = \frac{{rp}}{{r + p}},
\end{align*}
and the constant $C$ is independent of $g,h$,
\end{lemma}

The pointwise estimates for the free resolvent are given in Appendix A. We recall the kernel estimates of the shifted wave operator proved by  [Theorem 4.1, Theorem 4.4 \cite{45XDADER45P}] for reader's convenience.

Let $\chi_{\infty}(\lambda)$ be a cutoff function which equals one when $\lambda\ge \frac{3}{2}$ and vanishes near zero. Denote $\chi_0=1-\chi_{\infty}$.  Recall that $\mathfrak{{D}}=(-\Delta-\frac{1}{4})^{\frac{1}{2}}$ and  $D=\sqrt{-\Delta}$.
For $\sigma\in\Bbb R$, $\tau\in[0,\frac{3}{2}),$ define the low frequency cutoff shifted wave operator
$$
\widetilde{W}^{\sigma,\tau}_{t,0}=\chi_{0}(\mathfrak{D}){\mathfrak{D}}^{-\tau}{D}^{\tau-\sigma}e^{it {\mathfrak{D}}},
$$
and denote its kernel as $w^{\sigma,\tau}_{t,0}(r)$.
The modified high frequency wave operator is defined by an analytic family of operators
$$
\widetilde{W}^{\sigma,\tau}_{t,\infty}
=\frac{e^{\sigma^2}}{\Gamma(3/2-\sigma)}\chi_{\infty}({\mathfrak{D}}){\mathfrak{D}}^{-\tau}{D}^{\tau-\sigma}e^{it {\mathfrak{D}}}
$$
in the vertical strip $0\le \Re\sigma\le \frac{3}{2}$. Denote its kernel by $\widetilde{w}^{\sigma,\tau}_{t,\infty}(r)$.

\begin{lemma}[\cite{45XDADER45P}]\label{wjiu87}
The two kernels $w^{\sigma,\tau}_{t,0}(r)$ and $\widetilde{w}^{\sigma,\tau}_{t,\infty}(r)$ satisfy
\begin{itemize}
  \item Assume $|t|\le2$. Then for any $r\ge0$
  $$|w^{\sigma,\tau}_{t,0}(r)|\lesssim \varphi_0(r)
  $$
  \item Assume $|t|\ge2$. \begin{itemize}
           \item (a) If $0\le r\le \frac{|t|}{2}$, then
           $$|w^{\sigma,\tau}_{t,0}(r)|\lesssim |t|^{\tau-3}\varphi_0(r)
           $$
           \item (b) If $r\ge \frac{|t|}{2}$, then
           $$|w^{\sigma,\tau}_{t,0}(r)|\lesssim (1+|t-r|)^{\tau-2}e^{-\frac{1}{2} r}.
           $$
           \end{itemize}
\end{itemize}
For any fixed $\tau\in \Bbb R$ and $\sigma\in \Bbb C$ with $\Re \sigma=\frac{3}{2}$, we have the following:
\begin{itemize}
\item Assume $0<|t|\le2$. \begin{itemize}
           \item (a) If $0\le r\le 3$, then
           $$|\widetilde{w}^{\sigma,\tau}_{t,\infty}(r)|\lesssim |t|^{-\frac{1}{2}}(1-\log|t|)
           $$
           \item (b) If $r\ge 3$, then $|\widetilde{w}^{\sigma,\tau}_{t,\infty}(r)|=O(r^{-\infty} e^{-\frac{1}{2} r})$.
           \end{itemize}
\item Assume $|t|\ge2$. Then for any $r\ge0$
$$|\widetilde{w}^{\sigma,\tau}_{t,\infty}(r)|\lesssim (1+|r-|t||)^{-\infty}e^{-\frac{1}{2} r}.
$$
 \end{itemize}
where we used the notation (\ref{Ytfedrx}).
\end{lemma}

Now we use the kernel estimates for the shifted wave operator to deduce the resolvent estimates.
\begin{lemma}\label{4der45P}
Let $\rho(x)=e^{-d(x,0)}$ for $x\in\Bbb H^2$.
Let $\alpha>0$. Then for all $z\in \Bbb C\setminus [0,\infty)$:
\begin{itemize}
  \item For $\frac{6}{5}<r<2$ and $2<p<6$,
  \begin{align}
  \|(-\Delta-\frac{1}{4}-z)^{-1}f\|_{L^p_x}\lesssim \|f\|_{L^r_x}\label{4der45P}
  \end{align}
  \item For $\frac{6}{5}<r\le 2$ and $2\le p<6$,
  \begin{align}
  \|\rho^{\alpha}(-\Delta-\frac{1}{4}-z)^{-1}\rho^{\alpha}f\|_{L^p_x}\lesssim \|f\|_{L^r_x}\label{pq3}
  \end{align}
\end{itemize}
\end{lemma}
\begin{proof}
We shall use the formula
\begin{align}\label{z7vdr}
(-\Delta-\frac{1}{4}-(\lambda+i\mu)^2)^{-1}=C(\lambda,\mu)\int^{\infty}_0e^{i({\rm {sgn}} \mu)\lambda t}e^{-|\mu|t}{\mathfrak{D}}^{-1}(\sin t{\mathfrak{D}})dt,
\end{align}
where $C(\lambda,\mu)={{\rm{sgn}}\mu}\frac{-|\mu|+i({\rm {sgn}} \mu)\lambda}{i(\lambda+i\mu)}$.
Consider the analytic family of operators
\begin{align}
&R^{\sigma,\tau}_0=C(\lambda,\mu)\int^{\infty}_0e^{i({\rm {sgn}} \mu)\lambda t}e^{-|\mu|t}{D}^{\tau-\sigma}{\mathfrak{D}}^{-\tau}(\sin t{\mathfrak{D}})\chi_{0}({\mathfrak{D}})dt.\\
&R^{\sigma,\tau}_{\infty}=\frac{C(\lambda,\mu)e^{\sigma^2}}{\Gamma(\frac{3}{2}-\sigma)}\int^{\infty}_0e^{i({\rm {sgn}} \mu)\lambda t}e^{-|\mu|t}{D}^{\tau-\sigma}{\mathfrak{D}}^{-\tau}(\sin t{\mathfrak{D}})\chi_{\infty}({\mathfrak{D}})dt.\label{2xtf}
\end{align}
Thus by Lemma \ref{wjiu87} and Lemma \ref{tywe55gvu}, for $1<r<2< p<\infty$, $\lambda,\mu\in \Bbb R$, $\tau<2$, $\sigma\in\Bbb R$,
\begin{align}\label{poijn}
\|R^{\sigma,\tau}_0f\|_{L^p_x}\le C\|f\|_{L^r_x},
\end{align}
where $C$ is independent of $\lambda, \mu$.
In particular, the low frequency part of (\ref{4der45P}) is done.

When $\Re \sigma=\frac{3}{2}$, $\tau<2$, for any $1<r<2<p<\infty$, the kernel of ${D}^{\tau-\sigma}{\mathfrak{D}}^{-\tau}(\sin t{\mathfrak{D}})\chi_{\infty}({\mathfrak{D}})$ denoted by $w^{\sigma,\tau}_{\infty,t}$ satisfies for any $\omega>0$
\begin{align}\label{wp321}
\|w^{\sigma,\tau}_{\infty,t}*f\|_{L^p_x}\lesssim  \min(t^{-\infty},t^{-\frac{1}{2}-\omega})\|f\|_{L^r_x},
\end{align}
by applying Lemma \ref{wjiu87} and Lemma \ref{tywe55gvu} again.

Meanwhile, when $\Re\sigma=0, \tau<2$, we have the trivial $L^2\to L^2$ bound.
Hence, by complex interpolation one obtains that $w^{\sigma,\tau}_{\infty,t}$
satisfies (\ref{wp321}) as well in the regime
\begin{align}\label{gt56drtf}
\frac{1}{2}-\frac{\sigma}{3}<\frac{1}{p}<\frac{1}{2}, \mbox{ }\frac{1}{2}<\frac{1}{r}<\frac{1}{2}+\frac{\sigma}{3}.
\end{align}

Thus, take $\tau=1$ and $\sigma=1$ in (\ref{gt56drtf}), then by inserting the bound (\ref{wp321}) to (\ref{2xtf}) we obtain the high frequency part of (\ref{4der45P}). Combining with the bound for low frequency part of (\ref{4der45P}) contained in (\ref{poijn}), we arrive at  (\ref{4der45P}).

(\ref{pq3}) can be similarly proved by noticing the additional $\rho^{\alpha}$ weight helps us to use the $p=q=2$ case in Lemma \ref{tywe55gvu}.
\end{proof}

\begin{remark}
The results as (\ref{4der45P}) are usually called uniform resolvent estimates. For high dimensional hyperbolic spaces, one needs a scaling balance condition for $p,r$, see \cite{HS} for $\Bbb H^n$, $n\ge3$. We remark that the proof here is only available for $n=2$ due to the $t^{-1}$ singularity at $t=0$ in (\ref{z7vdr}) when one tries to apply dispersive estimates of wave operators in higher dimensions.
\end{remark}

The convolution kernel of the free resolvent $(-\Delta-s(1-s))^{-1}$ in $\Bbb H^{n}$ is given by
\begin{align}\label{7mvbtr}
[{ }^n\widetilde{R}]_0(s;x,y)=(2\pi)^{-\frac{n}{2}}e^{-i\pi\mu}(\sinh r)^{-\mu}Q^{\mu}_{\nu}(\cosh r),
\end{align}
where $Q^{\mu}_{\nu}$ is the Legendre function with $\mu=\frac{n-2}{2}$, $\nu=s-\frac{n}{2}$. The point-wise estimates for $[{ }^n\widetilde{R}]_0(s;x,y)$ are given in Lemma \ref{Tomnagdfe4} in Appendix A.

The spectrum of $-\Delta$ overlap  with $[\frac{1}{4},\infty)$. In order to study resolvent near the spectrum half-line $[\frac{1}{4},\infty)$,
we need to define $(-\Delta-\frac{1}{4}-\lambda^2\pm i0)^{-1}$ as what was done in the Euclidean case. The following lemma known as the limiting absorption principle is totally analogous to the $\Bbb R^n$ case in  Agmon \cite{Agmon} and can be proved by using (\ref{v2jsfcdf}) to (\ref{bbxdgd758}) below.
\begin{lemma}\label{fswe34dre5d}
For any $\lambda>0$, the limit $\mathop {\lim }\limits_{z \to 0, \Im z>0} {\left( { - \Delta  - {\frac{1}{4}} - {\lambda }\pm z} \right)^{ - 1}}$  exists in the space $\mathcal{L}(\rho^{-\alpha}L^2,\rho^{-\alpha}L^2)$. And we denote
\begin{align}
\mathop {\lim }\limits_{z \to 0, \Im z>0} {\left( { - \Delta  - {\frac{1}{4}} - {\lambda } \pm z} \right)^{ - 1}} = {\mathfrak{R}}_0\left(\lambda  {  \pm i0} \right).
\end{align}
Moreover, $g:={\mathfrak{R}_0}\left( {\lambda  \pm i0} \right) f$ satisfies
\begin{align}\label{equ}
(-\Delta-\frac{1}{4}-\lambda)g=0,
\end{align}
and for all $f\in \rho^{\alpha}L^2$ it holds
\begin{align}\label{Agdfe4}
\Im\langle \mathfrak{R}_0(\lambda\pm i0)f,f\rangle(\tau)=\pm\frac{\pi}{2\sqrt{\lambda}}\int_{|\xi|=\lambda}|\mathcal{F}f(\tau,b)|^2|c(\tau)|^{-2}db.
\end{align}
Assume $f_n\rightharpoonup f_*$ weakly in $\rho^{\alpha}L^2$, $z_n\to z_*$ with $\Im z_n>0$,
\begin{itemize}
  \item if $\Im z_*=0$, $z_*>0$, then it converges strongly in $L^2$ that
 \begin{align}
 \rho^{\alpha}R_0(\frac{1}{4}+z_n)\rho^{\alpha}f_n&\to \rho^{\alpha}\mathfrak{R}_0(z_*- i0)\rho^{\alpha}f_*\label{flg1}\\
 \rho^{\alpha}\nabla R_0(\frac{1}{4}+z_n)\rho^{\alpha}f_n&\to \rho^{\alpha}\nabla\mathfrak{R}_0(z_*-i0)\rho^{\alpha}f_*;\label{flg2}
 \end{align}
 \item if $\Im z_*>0$ or $\Re z_*<0$ then it converges strongly in $L^2$ that
 \begin{align}
 \rho^{\alpha}R_0(\frac{1}{4}+z_n)\rho^{\alpha}f_n&\to \rho^{\alpha}{R}_0(\frac{1}{4}+z_*)\rho^{\alpha}f_*\label{flg3}\\
 \rho^{\alpha}\nabla R_0(\frac{1}{4}+z_n)\rho^{\alpha}f_n&\to \rho^{\alpha}\nabla{R}_0(\frac{1}{4}+z_*)\rho^{\alpha}f_*.\label{flg4}
 \end{align}
\end{itemize}
\end{lemma}

Denote the convolution operator with kernel $[{ }^2\widetilde{R}]_0(\frac{1}{2},x,y)$ in (\ref{7mvbtr}) by $G(0)$.
\begin{lemma}\label{asw234dfe4}
For $\epsilon\in\{\epsilon:\Im \epsilon^2 >0, |\epsilon|\ll1 \}$, there exists some universal constant $C$ such that
\begin{align}
\left\| \left( { - \Delta  - {\frac{1}{4}}\pm{\epsilon^2  }} \right)^{ - 1}-{G}(0)\right\|_{\mathcal{L}(\rho^{-\alpha}L^2,\rho^{-\alpha}L^2)}&\lesssim\epsilon^{1/4}\label{4dfe476}\\
\left\|\nabla \left( { - \Delta  - {\frac{1}{4}}\pm{\epsilon^2  }} \right)^{ -  1}-\nabla{G}(0)\right\|_{\mathcal{L}(\rho^{-\alpha}L^2,\rho^{-\alpha}L^2)}&\lesssim \epsilon^{1/4}.\label{0cdfjpjuk}
\end{align}
\end{lemma}
\begin{proof}
We only prove the case when $\Im \epsilon >0$.
The proof of Lemma \ref{asw234dfe4} is based on the corresponding estimates for $\partial_{s}[{ }^n\widetilde{R}]_0(s,x,y)$. We will frequently use the identity
\begin{align}\label{tiand}
(z^2-1)^{m/2}\frac{d^m}{dz^m}Q^{0}_{\eta}=Q^{m}_{\eta}(z).
\end{align}
Let $r=d(x,y)$, $s=\frac{1}{2}+e^{i\frac{\pi}{4}}\epsilon$, then Lemma \ref{Tomnagdfe4} implies for any $\delta>0$
\begin{align}\label{cbndg45dr}
\left| \partial _\epsilon[{ }^2{\widetilde{R}}]_0(s,x,y) \right| \lesssim \left\{ \begin{array}{l}
 \log \left| r \right|,\left| r \right| \le 1 \\
 {\left| \epsilon \right|^{ - \frac{1}{2}}}{C_\delta }{e^{ - (\frac{1}{2} - \delta )r}},\left| r \right| \ge 1 \\
 \end{array} \right.
\end{align}
By Lemma \ref{tywe55gvu}, one easily obtains (\ref{4dfe476}) from (\ref{cbndg45dr}) and Newton-Leibniz formula.

Since $|\nabla_x d(x,y)|=1$ for any $y\in\Bbb H^2$, the key ingredient to prove (\ref{0cdfjpjuk}) is the estimate for
$\partial_{\epsilon}\partial_r [{ }^2{\widetilde{R}}]_0(s,x,y)$. By using (\ref{tiand}),  (\ref{7mvbtr}) and Lemma \ref{Tomnagdfe4}, we have for any $\delta>0$
\begin{align}
\left| {{\partial _\epsilon}{\partial _r}{{{[^2}{\widetilde{R}}]}_0}(s,x,y)} \right| \lesssim \left\{ \begin{array}{l}
 {({\cosh ^2}r - 1)^{ - \frac{1}{2}}}({\sinh ^2}r){r^{ - 2}},\mbox{  }\left| r \right| \le 1 \\
 {C_\delta }{({\cosh ^2}r - 1)^{ - \frac{1}{2}}}({\sinh ^2}r){\left| \epsilon \right|^{\frac{1}{2}}}{e^{ - (3/2 - \delta )r}},\mbox{  }\left| r \right| \ge 1 \\
 \end{array} \right.
\end{align}
Thus (\ref{0cdfjpjuk}) follows by Lemma \ref{tywe55gvu} and the Newton-Leibniz formula.
\end{proof}

The following lemma was proved in \cite{6HDFRG}.
\begin{lemma}[\cite{6HDFRG}]\label{somerthing}
For $\sigma\in\Bbb C$ with $\Re \sigma\ge 0$, we have for $\alpha>0$ sufficiently small
\begin{align}
\|\rho^{\alpha}R_{0}(\frac{1}{4}-\sigma^2)f\|_{L^2}&\lesssim|\sigma|^{-1}\|f\rho^{-\alpha}\|_{L^2}\label{v2jsfcdf}\\
\|\rho^{\alpha}R_{0}(\frac{1}{4}-\sigma^2)f\|_{L^2}&\lesssim\|f\rho^{-\alpha}\|_{L^2}\label{cbdg45dr}\\
\|\rho^{\alpha}\nabla R_{0}(\frac{1}{4}-\sigma^2)f\|_{L^2}&\lesssim\|f\rho^{-\alpha}\|_{L^2}.\label{bbxdgd758}
\end{align}
\end{lemma}

\begin{lemma}\label{something}
Let $0<\alpha<\frac{1}{2}$. Then for all $\sigma\ge 0$, $p\in (1,\infty)$, we have
\begin{align}
\|(-\Delta+\sigma^2-\frac{1}{4})^{-1}\|_{L^p\to L^p}&\lesssim \min(1,\sigma^{-2})\label{onumber}\\
\|\nabla (-\Delta+\sigma^2-\frac{1}{4})^{-1}\|_{L^p\to L^p}&\lesssim \min(1,\sigma^{-1})\label{1onumber}\\
\|(-\Delta+\sigma^2-\frac{1}{4})^{-1}\|_{\rho^{-\alpha}L^2\to \rho^{-\alpha}L^2}&\lesssim \min(1,\sigma^{-2})\label{2onumber}\\
\|\nabla (-\Delta+\sigma^2-\frac{1}{4})^{-1}\|_{\rho^{-\alpha}L^2\to\rho^{-\alpha} L^2}&\lesssim \min(1,\sigma^{-1}).\label{3onumber}
\end{align}
\end{lemma}
\begin{proof}
(\ref{2onumber}) and (\ref{3onumber}) have been proved in our companion paper [Lemma 4.5,\cite{6HDFRG}]. The proof of (\ref{onumber}) and (\ref{1onumber}) are postponed to Section 9 of this paper.
\end{proof}

\subsection{Some Auxiliary Estimates}
\begin{lemma}\label{hessian}
For $r\in(1,\infty)$, $1\le p<q\le \infty$ and $1+\frac{1}{q}-\frac{1}{p}> \frac{3}{4}$, we have
\begin{align}
\|\nabla^2f\|_{L^q_x}&\lesssim \|\Delta f\|_{L^p_x}\label{poking1}\\
\|\nabla^2f\|_{L^r_x}&\lesssim \|\Delta f\|_{L^r_x}+\|\nabla f\|_{L^r_x}+\|f\|_{L^r_x}.\label{poking2}
\end{align}
\end{lemma}
\begin{proof}
(\ref{poking2}) is obtained by \cite{Aer45P}.
Fix $y\in\Bbb H^2$. Let $(r,\theta)$ be the polar coordinates with $y$ being the origin. Then (\ref{tiand}) and (\ref{7mvbtr}) show
\begin{align*}
\partial^2_r[{ }^2\widetilde{R}]_0(\frac{1}{2} + \sigma,d(x,y))&=c_1(\cosh^2r-1)^{-\frac{3}{2}}\sinh^3 r\cosh r[{ }^4\widetilde{R}]_0(\frac{3}{2} + \sigma ,r)\\
&+c_2(\cosh^2r-1)^{-\frac{1}{2}}\sinh r\cosh r [{ }^4\widetilde{R}]_0(\frac{3}{2} + \sigma ,r)\\
&+c_3(\cosh^2r-1)^{-1}\sinh^3 r\cosh r [{ }^4\widetilde{R}]_0(\frac{3}{2} + \sigma ,r)\\
&+c_4(\cosh^2r-1)^{-1}\sinh^4 r [{ }^6\widetilde{R}]_0(\frac{5}{2} + \sigma ,r)\\
\end{align*}
Thus when $\sigma=\frac{1}{2}$, Lemma \ref{Tomnagdfe4} gives
\begin{align}\label{stein5}
\left| {\partial _r^2{{{[^2}\widetilde{R}]}_0}(1,d(x,y))} \right| \le \left\{ \begin{array}{l}
 {Ce^{ - r}},\mbox{  }\left| r \right| \ge 2 \\
 {Cr^{ - \frac{3}{2}}},\mbox{  }\left| r \right| \le 2 .\\
 \end{array} \right.
\end{align}
Meanwhile, Lemma \ref{xuannv} yields
\begin{align}\label{stein6}
\left|\coth r\partial_r[{ }^2\widetilde{R}]_0(1,d(x,y)) \right|  \le \left\{ \begin{array}{l}
 {Ce^{ - r}},\mbox{  }\left| r \right| \ge 1 \\
 {Cr^{ - \frac{3}{2}}},\mbox{  }\left| r \right| \le 1. \\
 \end{array} \right.
\end{align}
Therefore, in the polar coordinates with $y$ being the origin, by (\ref{stein6}) and (\ref{stein5}), for $x\in \Bbb H^2$  we have
\begin{align}\label{stein7}
\left| {{\nabla_x ^2}{{[{{ }^2}\widetilde{R}]}_0}(1,d(x,y))} \right| \le \left\{ {\begin{array}{*{20}{c}}
   {C{r^{ - \frac{3}{2}}}{\rm{, }}\left| r \right| \le 1}  \\
   {C{e^{ - r}},\left| r \right| \ge 1}.  \\
\end{array}} \right.
\end{align}
Since the left and right hand side of (\ref{stein7}) are free of coordinates charts, we obtain that (\ref{stein7}) holds for all $(x,y)\in\Bbb H^2\times\Bbb H^2$.
Thus (\ref{poking1}) follows by Young's convolution inequality.
\end{proof}

\section{ Spectrum Distribution and  Coulomb Gauge on the pullback bundle $Q^*TN$}

First we show the spectrum of $H$ is contained in $[\frac{1}{4},\infty)$ by doing calculation in the pullback bundle $Q^*TN$.

\begin{proposition}\label{CXXXXE}
Let $\bf e$ be any fixed frame on $Q^*(TN)$. Then the spectrum of the operator $-\Delta+W$ defined in Lemma \ref{REWSDERF} is contained in $[1/4,\infty)$.
\end{proposition}
\begin{proof}
Since $H$ is independent of coordinates of $M=\Bbb H^2$, without loss of generality we assume $(x_1,x_2)$ is an orthogonal coordinate system.
Let ${\bf e}(x)=\{e_i(Q(x))\}^{2}_{i=1}$ be any given orthonormal frame on $Q^*TN$ in Proposition \ref{94obhxdg}.
For $f,g\in L^2(\Bbb H^2,\Bbb C^2)$, we associate them with the vector fields $f{\bf e}\triangleq f_1 e_1+f_2 e_2$ and $g{\bf e}\triangleq g_1 e_1+g_2 e_2$ respectively. The corresponding induced covariant derivative on the complex vector bundle over $\Bbb H^2$ with fibre $\Bbb C^2$ is given by
\begin{align*}
\mathbb{D}_i=\partial_i+A^{\infty}_i, \mbox{  }{\rm{with}}\mbox{  }[A^{\infty}_i]^k_p=\langle \nabla_i e_p,e_k\rangle.
\end{align*}
Then we have
\begin{align*}
\nabla_i (f{\bf e})=(\partial_i f_k+\sum^2_{p=1}[A^{\infty}_i]^k_pf_p)e_k=(\mathbb{D}_i f){\bf e}.
\end{align*}
Furthermore, one has
$$
h^{ii}\nabla_i\nabla_i (f{\bf e})=h^{ii}(\mathbb{D}_i\mathbb{D}_if){\bf e}.
$$
Meanwhile we see
\begin{align}\label{tu87}
\big(h^{ii}( \phi_i^{\infty}\wedge f)\phi_i^{\infty}\big){\bf e}=h^{ii}\mathbf{R}(Q(x))(\phi^{\infty}_i{\bf e},f{\bf e})\phi_i^{\infty}{\bf e},
\end{align}
where we remark that the curvature tensor $\mathbf{R}$ here has been extended to complex vector fields (see (\ref{gevcdo}) for the corresponding scalar case).
Therefore, we conclude
\begin{align}
&h^{ii}\nabla_i\nabla_i(f{\bf e})-h^{ii}\Gamma^{k}_{ii}\nabla_k (f{\bf e})=h^{ii}(\mathbb{D}_i\mathbb{D}_if){\bf e}-h^{ii}\Gamma^k_{ii}(\mathbb{D}_kf){\bf e})\nonumber\\
&=(\Delta f-Wf){\bf e}-h^{ii}\mathbf{R}(Q(x))(\phi^{\infty}_i{\bf e},f{\bf e})\phi^{\infty}_i{\bf e}.\label{tu89}
\end{align}
Hence (\ref{tu87}) and (\ref{tu89}) give
\begin{align*}
\Delta |f{\bf e}|^2&=h^{ii}\langle \nabla_i\nabla_i (f{\bf e}),f{\bf e}\rangle+h^{ii}\langle  f{\bf e},\nabla_i\nabla_i(f{\bf e})\rangle
+2h^{ii}\langle \nabla_i (f{\bf e}),\nabla_i (f{\bf e})\rangle\\
&-h^{ii}\Gamma^k_{ii}\langle \nabla_i (f{\bf e}),f{\bf e}\rangle-h^{ii}\Gamma^k_{ii}
\langle  f{\bf e},\nabla_i (f{\bf e})\rangle\\
&=\langle \Delta f-Wf,f\rangle+\langle f,\Delta f-Wf\rangle+2h^{ii}\langle \nabla_i (f{\bf e}),\nabla_i (f{\bf e})\rangle\\
&-h^{ii}\langle\mathbf{R}(\phi^{\infty}_i{\bf e},f{\bf e})\phi^{\infty}_i{\bf e},f{\bf e}\rangle-h^{ii}\langle f{\bf e},\mathbf{R}(\phi^{\infty}_i{\bf e},f{\bf e})\phi^{\infty}_i{\bf e}\rangle.
\end{align*}
Then by integration by parts, the self-adjointness of $\Delta-W$ and the non-positiveness of the sectional curvature, we obtain
\begin{align}
2\langle (-\Delta+W)f,f\rangle_{L^2}\ge2\int_{\Bbb H^2}h^{ii}\langle \nabla_i (f{\bf e}),\nabla_i (f{\bf e})\rangle {\rm{dvol_h}}.
\end{align}
By Kato's inequality $|\nabla|X||\le |\nabla X|$ and the Sobolev inequality $\|\nabla g\|^2_{L^2}\ge \frac{1}{4}\|g\|^2_{L^2}$, one deduces
\begin{align*}
\int_{\Bbb H^2}\langle (-\Delta+W)f,f\rangle {\rm{dvol_h}}&\ge \frac{1}{4}\int_{\Bbb H^2}\langle (f{\bf e}),(f{\bf e})\rangle {\rm{dvol_h}}\\
&=\frac{1}{4}\int_{\Bbb H^2}|f|^2{\rm{dvol_h}}.
\end{align*}
Since the spectrum is contained in the numerical range, we obtain our lemma.
\end{proof}

\subsection{Weighted Elliptic Estimates for Coulomb Gauge}

We will use the Coulomb gauge on $Q^*(TN)$ to kill possible bottom resonances. Thus we first need to prove the point-wise  estimates for the new connection coefficients induced by the Coulomb gauge. The existence of the Coulomb gauge in two dimensions is well-known, see for instance  \cite{UhK}. We give the detailed proof since it tells us the explicit form of the Coulomb gauge.
\begin{lemma}\label{huax}
There exists an orthonormal frame $\{{\bf e}_1,{\bf e}_2\}$ which spans $T_{Q(x)}N$ for any $x\in M=\Bbb H^2$ such that the corresponding connection  1-form $\widetilde{{A}}\in \Lambda^1({\rm {Ad}} \mbox{ }Q^*TN)$ satisfies the Coulomb condition, i.e.,
\begin{align}
d^* \widetilde{A}=0.
\end{align}
\end{lemma}
\begin{proof}
In the two dimensional case, the connection coefficient matrix $A_i$ is of the form
\begin{align}\label{pegxs5tf}
\left( {\begin{array}{*{20}{c}}
   0 & {{a_i}}  \\
   { - {a_i}} & 0  \\
\end{array}} \right)
\end{align}
for some real-valued function $a_i$.  Denote
\begin{align*}
J=\left(\begin{array}{cc}
                                        0& 1 \\
                                        -1 & 0 \\
                                      \end{array}
                                    \right).
\end{align*}
Suppose that $\Omega(Q(x))$ is the frame on $Q(x)$ given by (\ref{ex3456gvh}). Then the corresponding connection 1-form is
\begin{align}
\widetilde{{A}}=A^{\infty}_idx^i, \mbox{  }[A^{\infty}_i]^k_l=\langle \nabla_i \Omega_l,\Omega_k\rangle .
\end{align}
Given any real valued function $\xi\in H^1(\Bbb H^2;\Bbb R)$, we associate it with a matrix $U\in SO(2)$ defined by
\begin{align*}
U(x)=\left(
    \begin{array}{cc}
      \sin \xi(x) & \cos \xi(x) \\
      -\cos \xi(x) & \sin \xi(x) \\
    \end{array}
  \right)
\end{align*}
Define the new frame ${\bf e}^*(x)=U(x)\Omega(Q(x))$.
Then the new connection 1-from $\mathcal{A}$ is given by
\begin{align}
\mathcal{A}=\widetilde{{A}}+d\xi J.
\end{align}
Thus the Coulomb condition reduces to
\begin{align}
d^{*}\mathcal{A}=d^{*}\widetilde{{A}}+\Delta \xi J=0.
\end{align}
Hence it suffices to set
\begin{align}\label{op09}
\xi J=(-\Delta)^{-1}d^{*}\widetilde{{A}}.
\end{align}
\end{proof}

\begin{proposition}\label{hutyhvf}
Assume that the given frame $\bf e$ in Proposition \ref{94obhxdg} is the Coulomb gauge constructed in Lemma \ref{huax}. Then the associated  Schr\"odinger operator $H=-\Delta+W$ reads as
\begin{align}\label{uy78}
H\varphi=-\Delta\varphi-2h^{ii}\mathcal{A}_i\partial_i\varphi-h^{ii}\mathcal{A}_i\mathcal{A}_i\varphi
+h^{ii}(\widehat{\phi}_i\wedge\varphi)\widehat{\phi}_i.
\end{align}
where $\mathcal{A}_i=A^{\infty}_i+\partial_i\xi J$.
And $\widehat{\phi}_i$ is $\Bbb R^2$ valued function defined on $M=\Bbb H^2$:
\begin{align*}
\widehat{\phi}_i=\big(\langle \partial_iQ(x),{\bf e}_1(x)\rangle, \langle \partial_iQ(x),{\bf e}_2(x)\rangle\big)^t.
\end{align*}
Moreover, let $X=-2h^{ii}\mathcal{A}_i\frac{ \partial}{\partial x_i}$  denote the magnetic field part of $H$ and  $V$ denote the left electric potential part, i.e. $H=-\Delta+X+V$, then for any $0<\beta<\varrho$ there holds
\begin{itemize}
  \item (a) $V$ is nonnegative on $L^2(\Bbb H^2;\Bbb C^2)$
  \item (b) $\|\rho^{-\beta}\mathcal{A}\|_{L^{\infty}_x}\le C$
  \item (c) $\|\rho^{-\beta}V\|_{L^{\infty}_x}\le C$
\end{itemize}
\end{proposition}
\begin{proof}
Fix any orthogonal coordinate system $\{x_1,x_2\}$ for $M=\Bbb H^2$. Denote $\mathcal{A}=\mathcal{A}_idx^i$. Then the Coulomb condition is written as
\begin{align}\label{hums4}
h^{ii}\big(\partial_i\mathcal{A}_i-\Gamma^k_{ii}\mathcal{A}_k\big)=0,
\end{align}
thus the (\ref{hums4}) term in $H=-\Delta+W$ vanishes and  (\ref{uy78}) follows. It remains to prove the three claims (a),(b),(c).
Since the $V$ part reads as
\begin{align}\label{pov3p}
V\varphi=-h^{ii}\mathcal{A}_i\mathcal{A}_i\varphi+h^{ii}(\widehat{\phi }_i\wedge\varphi)\widehat{\phi}_i,
\end{align}
(a) is easy to verify by the negative sectional curvature of $N$ and the fact $\mathcal{A}_i$ is antisymmetry and real.
By (\ref{rer45xdg}) and $\mathcal{A}=\widetilde{{A}}+d\xi J$,  for (b) it suffices to prove
\begin{align}\label{y7cvx}
|d\xi|\lesssim \rho^{\beta}.
\end{align}
With (\ref{op09}), one notices that (\ref{y7cvx}) reduces to prove
\begin{align}\label{iu67b}
\|\rho^{-\beta}\nabla(-\Delta)^{-1}d^{*}\widetilde{{A}}\|_{L^{\infty}_x}\le C.
\end{align}
By the identity $d^{*}\widetilde{{A}}=-h^{ii}\big(\partial_i{A}^{\infty}_i-\Gamma^k_{ii}{A}^{\infty}_k\big)$ and (\ref{resder445xdg}), we see $\rho^{-\varrho}d^{*}{\widetilde{A}}\in L^{2}_x$. Thus
by (\ref{xuannv}) and Young's convolution inequality, one has for any $\beta\in(0,\varrho)$
\begin{align}
\|\rho^{-\beta}\nabla(-\Delta)^{-1}d^{*}\widetilde{{A}}\|_{L^{\infty}_x}\le C.\label{5.19}
\end{align}
Until now we have obtained
\begin{align}\label{po9mvb}
|\mathcal{A}|\lesssim \rho^{\beta}.
\end{align}
With (\ref{pov3p}), (c) follows from  (\ref{rer45xdg}) and (\ref{po9mvb}).
\end{proof}

\begin{lemma}\label{polarconnection}
Fix the frame $\bf e$ in Proposition \ref{94obhxdg} to be the Coulomb gauge built in Lemma \ref{huax}. In the polar coordinates for $M=\Bbb H^2$,  $H$ can be written as
\begin{align}
H\varphi=-\Delta\varphi-2\mathcal{A}_{r}\partial_r\varphi-2\mathcal{A}_{\theta}\sinh^{-2}r\partial_{\theta}\varphi-\mathcal{U}_r\varphi
-\mathcal{U}_{\theta}\varphi
\end{align}
where we denote
\begin{align}
&\mathcal{U}_r\varphi=\mathcal{A}_r\mathcal{A}_r\varphi-(\widehat{\phi}_r\wedge\varphi)\widehat{\phi}_r\\
&\mathcal{U}_{\theta}\varphi=\sinh^{-2}r\mathcal{A}_{\theta}\mathcal{A}_{\theta}\varphi
-\sinh^{-2}r(\widehat{\phi}_{\theta}\wedge\varphi)\widehat{\phi}_{\theta}
\end{align}
Then one has when $r\ge \delta$,
\begin{align}
&|\mathcal{U}_{\theta}|+|\mathcal{U}_r|\le C\rho^{\beta}\label{p08cbgtw}\\
&|\partial_r\mathcal{A}_{r}|+|(\sinh^{-1}r)\partial_\theta\mathcal{A}_{\theta}|+|(\sinh^{-2}r)\partial_{\theta}\mathcal{A}_{\theta}|\nonumber\\
&+|(\sinh^{-1}r)\partial_r\mathcal{A}_{\theta}|\le C(\delta)\label{8cbgtw}
\end{align}
\end{lemma}
\begin{proof}
(\ref{p08cbgtw}) is a trivial corollary of Proposition \ref{hutyhvf}. It suffices to verify (\ref{8cbgtw}).
By viewing $A^{\infty}_{\theta,r}$ as a real-valued function $[A^{\infty}_{\theta,r}]^{1}_2$ we have
\begin{align}\label{p78jiu}
\partial_i{A}^{\infty}_{j}=\langle \nabla_i\nabla_j\Omega_1,\Omega_2\rangle+\langle \nabla_j\Omega_1,\nabla_i\Omega_2\rangle,
\end{align}
where we associate $r$ with $i=1$ and $\theta$ with $i=2$ respectively.
Then inserting the explicit formula (\ref{ex3456gvh}) for $\Omega_{1,2}$  into (\ref{p78jiu}), we get
\begin{align*}
&{\partial _i}{A^{\infty}_j} = \left\langle {{\nabla _i}\left( {{e^{{Q^2}}}{\partial _j}{Q^2}\frac{\partial }{{\partial {y_1}}}} \right),{\Omega_2}} \right\rangle  + \left\langle {{\nabla _i}\left( {{e^{{Q^2}}}\overline \Gamma  _{j1}^p\frac{\partial }{{\partial {y_p}}}} \right),{\Omega_2}} \right\rangle
\\
&+ \left\langle {{e^{{Q^2}}}{\partial _j}{Q^2}\frac{\partial }{{\partial {y_1}}}
+ {e^{{Q^2}}}\overline \Gamma  _{j1}^p\frac{\partial }{{\partial {y_p}}},\overline \Gamma  _{i2}^q\frac{\partial }{{\partial {y_q}}}} \right\rangle  \\
&= \left\langle {{\partial _i}{Q^2}{\partial _j}{Q^2}{e^{{Q^2}}}\frac{\partial }{{\partial {y_1}}},{\Omega_2}} \right\rangle  + \left\langle {{e^{{Q^2}}}{\partial _{ji}}{Q^2}\frac{\partial }{{\partial {y_1}}},{\Omega_2}} \right\rangle  + \left\langle {{e^{{Q^2}}}{\partial _j}{Q^2}\overline \Gamma  _{j1}^p\frac{\partial }{{\partial {y_p}}},{\Omega_2}} \right\rangle  \\
&+ \left\langle {{\partial _i}{Q^2}{e^{{Q^2}}}\overline \Gamma  _{j1}^p\frac{\partial }{{\partial {y_p}}},{\Omega_2}} \right\rangle  + \left\langle {{e^{{Q^2}}}{\partial _i}\overline \Gamma  _{j1}^p\frac{\partial }{{\partial {y_p}}},{\Omega_2}} \right\rangle   + \left\langle {{e^{{Q^2}}}\overline \Gamma  _{j1}^p\overline \Gamma_{ip}^k\frac{\partial }{{\partial {y_k}}},{\Omega _2}} \right\rangle  \\
&+ \left\langle {{e^{{Q^2}}}{\partial _j}{Q^2}\frac{\partial }{{\partial {y_1}}} + {e^{{Q^2}}}\overline \Gamma  _{j1}^p\frac{\partial }{{\partial {y_p}}},\overline \Gamma  _{i2}^q\frac{\partial }{{\partial {y_q}}}} \right\rangle.
\end{align*}
Recalling the explicit formula for $|\nabla^kQ^{i}|$ for $i=1,2,k=1,2$, one obtains
\begin{align}
&|\partial_rA^{\infty}_{r}|+|(\sinh^{-1}r)\partial_\theta A^{\infty}_{\theta}|+|(\sinh^{-2}r)\partial_{\theta}A^{\infty}_{\theta}|
+|(\sinh^{-1}r)\partial_r A^{\infty}_{\theta}|\nonumber\\
&\lesssim \sum^{2}_{k=1,i=1}|\nabla^kQ^{i}|+|\coth r\partial_rQ^{i}|.\label{po6vcde}
\end{align}
When $r\ge\delta$, by Sobolev embedding and Kato's inequality,  (\ref{po6vcde}) further gives
\begin{align}
&|\partial_rA^{\infty}_{r}|+|(\sinh^{-1}r)\partial_\theta A^{\infty}_{\theta}|+|(\sinh^{-2}r)\partial_{\theta}A^{\infty}_{\theta}|
+|(\sinh^{-1}r)\partial_r A^{\infty}_{\theta}|\nonumber\\
&\lesssim C(\delta)\sum^{2}_{k=1,i=1}\|\nabla^kQ^{i}\|_{L^{\infty}}\nonumber\\
&\lesssim C(\delta)\sum^2_{i=1}\|Q^i\|_{\mathrm{H}^4}.
\end{align}
Hence, since $\mathcal{A}_i=A^{\infty}_i+\partial_i\xi J$, for (\ref{8cbgtw}) it suffices to bound $|h^{ij}\partial^2_{ij}\xi|$. Recalling $\widetilde{A}=A^{\infty}_idx^i$ and $\xi J=(-\Delta)^{-1}d^*\widetilde{{A}}$, we see
\begin{align}
h^{ij}|\partial^2_{ij}u|&\le |\nabla^2\big((-\Delta)^{-1}d^*\widetilde{A}\big)|+\coth r|\partial_r u|\nonumber\\
&\le |\nabla^2\big((-\Delta)^{-1}d^*\widetilde{A}\big)|+\coth r|du|.
\end{align}
Applying Lemma \ref{hessian} and \underline{Lemma 2.5} give for $r\ge \delta$ and some $p_0>4$
\begin{align}
h^{ij}|\partial^2_{ij}u|&\le \|\nabla^2\big((-\Delta)^{-1}d^*\widetilde{{A}}\big)\|_{L^{\infty}}+C(\delta)\|du\|_{L^{\infty}}\nonumber\\
&\le \|d^*\widetilde{{A}}\|_{L^{p_0}}+C(\delta)\|du\|_{L^{\infty}}
\end{align}
Thus (\ref{8cbgtw}) follows.
\end{proof}

\section{Resolvent Estimates for Magnetic Schr\"odinger Operators }

We recall the Kato smoothing theorem.
\begin{theorem}[\cite{TREWSD}]\label{TREWSD}
Let $Y_1,Y_2$ be two Hilbert spaces and $U:Y_1\to Y_2$ be a self-adjoint operator with resolvent $(U-\lambda)^{-1}$. Let $T:Y_1\to Y_2$ be a closed densely defined operator. Assume that for any $\lambda\in\Bbb C\setminus\Bbb R$, $g\in D(U^*)$, there holds
\begin{align*}
\|T(U-\lambda)^{-1}T^*g\|_{Y_2}\le C\|g\|_{Y_2}.
\end{align*}
Then $e^{\pm itT}f\in D(U)$ for all $f\in Y_1$ and a.e. $t$, and
\begin{align*}
\int^{\infty}_{-\infty}\|Te^{\pm itU}f\|^2_{Y_2}dt\le \frac{2}{\pi}C^2\|f\|^2_{Y_1}.
\end{align*}
\end{theorem}

Denote the Schr\"odinger operator $-\Delta+W$ as $H=-\Delta+X+V$. (see Proposition \ref{hutyhvf})
The resolvent of $H$ is denoted by $R_{H}(z)=(H-z)^{-1}$.

\subsection{Nonexistence of Resonance and Small Frequency Estimates}

\begin{lemma}\label{huapx}
Let $\alpha>0$, and $\bf e$ be the Coulomb gauge on $Q^*TN$ in Proposition \ref{94obhxdg}. And $H=-\Delta+W$ is the corresponding Schr\"odinger operator. Then we have
\begin{align}\label{0cdek}
\langle Hg,g\rangle_{L^2_x}\ge \langle \nabla g,\nabla g\rangle_{L^2_x}.
\end{align}
\end{lemma}
\begin{proof}
From Proposition \ref{hutyhvf}, $V$ is nonnegative and $H$ is self-adjoint. Thus one has
\begin{align}\label{8cdek}
\langle -\Delta g+Wg,g\rangle=\Re\langle -\Delta g+Wg,g\rangle \ge \Re\langle \nabla g ,\nabla g\rangle+\Re 2\langle h^{ii}\mathcal{A}_i\partial_i g,g\rangle.
\end{align}
Meanwhile by viewing $\mathcal{A}=\mathcal{A}_idx^i$ and the fact $d^*$ is the dual operator of $d$, we conclude
\begin{align}\label{92vgki}
2\Re \langle h^{ii}\mathcal{A}_i\partial_i g,g\rangle= \langle \mathcal{A},d(|g|^2)\rangle=\langle d^*\mathcal{A},|g|^2\rangle=0,
\end{align}
where the last equality is due to the Coulomb condition.
Therefore, (\ref{92vgki}) and (\ref{8cdek}) yield
\begin{align*}
\langle -\Delta g+Wg,g\rangle\ge \langle \nabla g ,\nabla g\rangle.
\end{align*}
\end{proof}

\begin{proposition}\label{resonance}
Let $\alpha>0$, and $\bf e$ be the Coulomb gauge on $Q^*TN$. Then $\rho^{-\alpha}(I+WG(0))^{-1}\rho^{\alpha}$ is invertible in $L^2$.
\end{proposition}
\begin{proof}
Let $T(\epsilon)=\rho^{-\alpha}W(-\Delta-\frac{1}{4}+i\epsilon^2)^{-1}\rho^{\alpha}$, for $\epsilon\in(0,1)$ and let $T_0=\rho^{-\alpha}WG(0)\rho^{\alpha}$. We first calculate the spectrum distribution of $T(\epsilon)$. By Fredholm's alternative, $T(\epsilon)$ only has pure point spectrum. Assume $\lambda$ is an eigenvalue of $T(\epsilon)$, then for some $u\in L^2$ it holds $T(\epsilon)u=\lambda u$.
Let $g=(-\Delta-\frac{1}{4}+i\epsilon^2)^{-1}\rho^{\alpha}u$, then we have
\begin{align}
\lambda(-\Delta-\frac{1}{4}+i\epsilon^2)g=Wg.
\end{align}
Thus we obtain
\begin{align}
(\lambda+1)(-\Delta-\frac{1}{4}+i\epsilon^2)g=(-\Delta-\frac{1}{4}+i\epsilon^2+W)g.
\end{align}
Since $\frac{1}{4}-i\epsilon^2$ belongs to the resolvent set of $-\Delta$, $\rho^{\alpha}u\in L^2$, we see $g\in H^2$. And thus integration by parts yields
\begin{align}\label{huax1}
(\lambda+1)\left(\langle\nabla g,\nabla g\rangle-(\frac{1}{4}-i\epsilon^2)\langle g,g\rangle\right)=\langle(-\Delta-\frac{1}{4}+i\epsilon^2+W)g,g\rangle.
\end{align}
Without loss of generality, we assume $\|g\|_{L^2}=1$, and then it holds
\begin{align}\label{huax20}
\lambda=\frac{(\Gamma_1-\Gamma_2)(\Gamma_2-i\epsilon^2)}{\Gamma^2_2+\epsilon^4},
\end{align}
where we denote $\Gamma_1=\langle(-\Delta-\frac{1}{4}+W)g,g\rangle$ and $\Gamma_2=\langle\nabla g,\nabla g\rangle-\frac{1}{4}\langle g,g\rangle$.
In the Coulomb case, by Lemma \ref{huapx},  (\ref{huax20}) gives
\begin{align}
\Re\lambda\ge 0.
\end{align}
Suppose that $1+T_0$ is not invertible in $L^2$, by Fredholm's alternative, $-1$ is an eigenvalue of $T_0$. Since the only possible accumulated point of $\sigma(T_0)$ is 0 due to the compactness, we see $-1$ is an isolated spectrum of $T_0$. Let $\partial B(-1,\delta)$ be a small circle centered at $-1$ with radius $\delta>0$. Define the projection operator $P_0$ by
\begin{align*}
P_0=\int_{\partial B(-1,\delta)}(T_0-z)^{-1}dz.
\end{align*}
Since $-1$ is an isolated spectrum, we have a uniform bound for all $z\in \partial B(-1,\delta)$
\begin{align}\label{uniform}
\|(T_0-z)^{-1}\|_{L^2\to L^2}\le C(\delta).
\end{align}
Then by the resolvent identity, Lemma \ref{asw234dfe4}, and Neumann's series argument, we have for all $z\in \partial B(-1,\delta)$ and $\epsilon\in(0,\epsilon(\delta))$ with $0<\epsilon(\delta)\ll1$,
\begin{align}\label{uniform2}
\|(T(\epsilon)-z)^{-1}\|\le \|(T_0-z)^{-1}(I+(T(\epsilon)-T_0)(T_0-z)^{-1})^{-1}\|_{L^2\to L^2}\le C_1(\delta).
\end{align}
Similarly we define the projection operator $P_\epsilon$ by
\begin{align*}
P_\epsilon=\int_{\partial B(-1,\delta)}(T(\epsilon)-z)^{-1}dz.
\end{align*}
Then by the resolvent identity, (\ref{uniform}) and (\ref{uniform2}), we obtain for $\epsilon$ sufficiently small,
\begin{align*}
&\|P_\epsilon-P_0\|_{L^2\to L^2}\\
&=\int_{\partial B(-1,\delta)}\|(T(\epsilon)-z)^{-1}(T(\epsilon)-T_0)(T_0-z)^{-1}\|_{L^2\to L^2}dz\\
&\le C_1(\delta)C(\delta)\delta\epsilon.
\end{align*}
Let $\epsilon\ll 1$, one has $\|P_\epsilon-P_0\|_{L^2\to L^2}\le \frac{1}{2}$. But we have shown $B(-1,\delta)$ is away from the spectrum of $T(\epsilon)$ for any $\epsilon>0$, thus $P_\epsilon=0$. Hence we arrive at $\|P_0\|_{L^2\to L^2}\le \frac{1}{2}$, which contradicts with the assumption $-1\in \sigma(T_0)$. Thus $I+T_0$ is invertible and the lemma follows.
\end{proof}

Now we deal with the small frequency part.
\begin{lemma}\label{resolvent1}
Let $\bf e$ be the Coulomb gauge on $Q^*TN$, $\alpha>0$. There exist $\delta>0$ and  $C>0$ such that it holds uniformly for $0<|\sigma|<\delta$, $\Im \sigma>0$ that
\begin{align}
\left\|\rho^{\alpha}(H-\frac{1}{4}\pm \sigma)^{-1}\rho^{\alpha}\right\|_{L^2\to L^2}\le C.
\end{align}
\end{lemma}
\begin{proof}
Denote $(-\Delta-\frac{1}{4}-\sigma)^{-1}=Z_{\sigma}$. By the formal identity
\begin{align*}
(H-\frac{1}{4}-\sigma)^{-1}=Z_{\sigma}(I+WZ_{\sigma})^{-1},
\end{align*}
it suffices to prove for some $C$ independent of $0<|\sigma|<\delta$, $\Im \sigma>0$
\begin{align}
\|\rho^{\alpha}Z_{\sigma}\rho^{\alpha}\|_{L^2\to L^2}&\le C\label{noh6}\\
\|(I+\rho^{-\alpha}WZ_{\sigma}\rho^{\alpha})^{-1}\|_{L^2\to L^2}&\le C.\label{noh7}
\end{align}
(\ref{noh6}) has been verified in Lemma \ref{somerthing}. By resolvent identity, we have formally that
\begin{align}\label{noh8}
(I+\rho^{-\alpha}WZ_{\sigma}\rho^{\alpha})^{-1}=(I+\rho^{-\alpha}WG(0)\rho^{\alpha})^{-1}(I+\widetilde{Z}(I+\rho^{-\alpha}WG(0)\rho^{\alpha})^{-1})^{-1},
\end{align}
where $\widetilde{Z}$ denotes $\rho^{-\alpha}W(Z_{\sigma}-G(0))\rho^{\alpha}$.
Then by Lemma \ref{asw234dfe4}, Proposition \ref{resonance}, and Neumann series argument we obtain for $0<|\sigma|<\delta$, $\Im \sigma>0$
\begin{align*}
\|(I+\widetilde{Z}(I+\rho^{-\alpha}WG(0)\rho^{\alpha})^{-1})^{-1}\|_{L^2\to L^2}\le 2.
\end{align*}
Hence (\ref{noh7}) follows by Proposition \ref{resonance}, (\ref{noh8}). The $(-\Delta-\frac{1}{4}+\sigma)^{-1}$ case is the same and thus our lemma follows.
\end{proof}

\subsection{Mediate Frequency Resolvent Estimates}

The mediate frequency resolvent estimate is standard in our case by applying the original idea of Agmon \cite{Agmon} and the Fourier restriction estimates obtained by Kaizuka \cite{Kaizuka}.
\begin{lemma}\label{re3}
If $I+\rho^{-\alpha}W\mathfrak{R}_0(\lambda+i0)\rho^{\alpha}$ is not invertible in $L^2$ for some $\lambda>0$, then $\frac{1}{4}+\lambda $ is an eigenvalue of $-\Delta+W$ in $L^2$.
\end{lemma}
\begin{proof}
By Fredholm's alternative, we can assume there exists $\widetilde{f}\in L^2$ such that
\begin{align}\label{jiuh8}
\widetilde{f}+\rho^{-\alpha}W\mathfrak{R}_0(\lambda\pm i0)\rho^{\alpha}\widetilde{f}=0.
\end{align}
Let $g=\mathfrak{R}_0(\lambda\pm i0)\rho^{\alpha}\widetilde{f}$, then Lemma \ref{somerthing} with (\ref{4der45P}) shows $g\in L^{r}$ for $r\in(2,6)$, $\nabla g\in \rho^{-\alpha}L^{2}_x$ and
\begin{align}\label{jiu8b}
(-\Delta-\frac{1}{4})g=\lambda g+Wg
\end{align}
in the distribution sense.
By density arguments and (\ref{jiuh8}) (the potential part decays exponentially),
one can verify
\begin{align}
\langle \rho^{\alpha}\widetilde{f},\mathfrak{R}_0(\lambda\pm i0)\rho^{\alpha}\widetilde{f}\rangle+\langle Wg,g\rangle=0
\end{align}
By the self-adjointness of $W$, we deduce
\begin{align}
\Im\langle \rho^{\alpha}\widetilde{f},\mathfrak{R}_0(\lambda\pm i0)\rho^{\alpha}\widetilde{f}\rangle=0
\end{align}
Let $f=\rho^{\alpha}\widetilde{f}$.
Hence (\ref{Agdfe4}) implies $|\mathcal{F}f(\tau,b)c^{-1}(\tau)|^{-2}$ vanishes when $\tau^2=\lambda$. Then the Fourier restriction estimate in [\cite{Kaizuka}, Equ. (4.4)] gives for any $\theta\in(0,1)$
\begin{align}
&\big(\int_{\Bbb S^1}\big|c^{-1}(\tau)\mathcal{F}f(\tau,b)-c^{-1}(\lambda^{\frac{1}{2}})\mathcal{F}f(\lambda^{\frac{1}{2}},b)\big|^2db\big)^{\frac{1}{2}}\nonumber\\
&\le C|\tau-\lambda^{\frac{1}{2}}|^{\theta}\|\langle x\rangle^{\frac{1}{2}+\theta} f\|_{L^2}.\label{huasi}
\end{align}
By the vanishing of  $|\mathcal{F}f(\lambda^{\frac{1}{2}},b)c^{-1}(\lambda^{\frac{1}{2}})|^{2}$, (\ref{huasi}) further yields
\begin{align*}
\int_{\Bbb S^1}\big|c^{-1}(\tau)\mathcal{F}f(\tau,b)\big|^2db\le |\lambda^{\frac{1}{2}}-\tau|^{2\theta}\|\langle x\rangle^{\frac{1}{2}+\theta} f\|^2_{L^2},
\end{align*}
Thus by Plancherel identity, one has for $0<\theta_1\ll1$ and $\frac{1}{2}<\theta_2<1$,
\begin{align*}
\|g\|^2_{L^2}&\le \int^{\infty}_0\int_{\Bbb S^1}(\tau^2-\lambda)^{-2}\big|c^{-1}(\tau)\mathcal{F}f(\tau,b)\big|^2dbd\tau\\
&\le C(\lambda) \|\langle x\rangle^{\frac{1}{2}+\theta_1} f\|^2_{L^2}+\|\langle x \rangle^{\frac{1}{2}+\theta_2} f\|^2_{L^2}\\
&\le C(\lambda) \|\widetilde{f}\|_{L^2}.
\end{align*}
This implies $g\in L^2$, thus by (\ref{jiu8b}) and $Wg\in L^2$, $\Delta g\in L^2$. Hence $g\in D(H)$ and $\lambda+\frac{1}{4}$ is an eigenvalue of $H$.
\end{proof}

The proof of the following lemma is quite standard, see [\cite{IS}, Lemma 4.6]. For completeness, we give the detailed proof below.
\begin{lemma}\label{resolvent2}
For all $\lambda>0$ and $\epsilon\in[0,1]$, we have
\begin{align}\label{p9k8}
\mathop {\sup }\limits_{\lambda  \in [\delta ,{\delta ^{ - 1}}],\epsilon\in [0,1]}
{\left\|\rho^{\alpha}(H-\frac{1}{4}-\lambda \pm i\epsilon)\rho^{\alpha}\right\|_{L^2\to L^2}}\le C(\delta).
\end{align}
\end{lemma}
\begin{proof}
The non-existence of positive eigenvalue of $-\Delta+W$ in $(\frac{1}{4},\infty)$ is standard by Mourre estimates, see [Prop. 5.2 \cite{TYWE5E65GVU}] for the electric potential case and \cite{Donnelly} for the original idea. Thus by Lemma \ref{re3}, for all $\lambda>0$
\begin{align}\label{yv8}
I+\rho^{-\alpha}\mathfrak{R}(\lambda\pm i0)\rho^{\alpha}\mbox{  } {\rm{is}}\mbox{  }{\rm{invertible}}\mbox{  }{\rm{in}}\mbox{  } L^2.
\end{align}
By the identity $R_H(\frac{1}{4}+\lambda \pm i\epsilon)=R_0(\frac{1}{4}+\lambda \pm i\epsilon)(I+WR_0(\frac{1}{4}+\lambda \pm i\epsilon))^{-1}$ and Lemma \ref{somerthing}, (\ref{p9k8}) follows by
\begin{align}\label{yvu78}
\mathop {\sup }\limits_{\lambda  \in [\delta ,{\delta ^{ - 1}}],\epsilon\in [0,1]}{\|I+\rho^{-\alpha}WR_0(\frac{1}{4}+\lambda  \pm i\epsilon)\rho^{\alpha}\|_{L^2\to L^2}}\le C(\delta).
\end{align}
Denote $\mathbf{T}(\lambda,\epsilon)=\rho^{-\alpha}W(-\Delta-\frac{1}{4}-\lambda \pm i\epsilon)\rho^{\alpha}$. Assume (\ref{yvu78}) fails, then there exists $f_n\in L^2$ with $\|f_n\|_{L^2}=1$ and $(\lambda_n,\epsilon_n)\in [\delta,\delta^{-1}]\times [0,1]$ such that
\begin{align}\label{epho}
\|(I+{\bf T}(\lambda_n,\epsilon_n))f_n\|_{L^2}\to0.
\end{align}
Up to subsequence, we assume $\lambda_n\to \lambda_*$ and $\epsilon_n\to \epsilon_*$.
And one may assume $f_n\rightharpoonup f_*$ weakly in $L^2$, then (\ref{flg1}) to (\ref{flg4}) give
\begin{align}\label{epho5}
{\bf T}(\lambda_n,\epsilon_n)f_n\to {\bf T}(\lambda_*,\epsilon_*)f_* \mbox{  }{\rm {strongly}}\mbox{ }{\rm {in}} \mbox{  }L^2.
\end{align}
Thus (\ref{epho}) shows for $f_*\in L^2$ it holds in the distribution sense that
\begin{align}\label{epho2}
f_*+{\bf T}(\lambda_*,\epsilon_*)f_*=0.
\end{align}
If $\epsilon_*>0$, it is obvious $f_*=0$ by (\ref{epho2}). If $\epsilon_*=0$, we also have $f_*=0$ by (\ref{yv8}). Then (\ref{epho}) and (\ref{epho5}) yield
\begin{align}
\mathop {\lim }\limits_{n \to \infty } \|f_n\|_{L_x^2} = 0,
\end{align}
which contradicts with $\|f_n\|_{L^2_x}=1$.
\end{proof}

\subsection{High Frequency Estimates}
Let $\chi(x)\in C^{\infty}(\Bbb H^2)$  be a cutoff function which equals one in $\{x:d(x,0)\ge2\delta\}$ and vanishes $d(x,0)<\delta$.
To separate the long range part of the magnetic potential away, we introduce the following decomposition:
\begin{align}\label{zqa1}
\left\{
  \begin{array}{ll}
   V^{far}&=\chi(x)V, \mbox{  }V^{near}=(1-\chi(x))V\hbox{ } \\
    X^{far}&=2\chi(x)h^{ij}\mathcal{A}_i\frac{\partial}{\partial x_j}, \mbox{  }X^{near}=2(1-\chi(x))h^{ij}\mathcal{A}_i\frac{\partial}{\partial x_j}. \hbox{ }
  \end{array}
\right.
\end{align}

Let us consider the operator $H_1:=-\Delta+X^{far}+V^{far}$ first. The high frequency resolvent estimates for $H_1$ in the weighted space is given by the following lemma. The proof of Lemma \ref{long09} is an energy argument based on \cite{CCV} where high frequency resolvent estimates for Schr\"odinger operators with large long-range magnetic potentials in $\Bbb R^n$ are considered.

Since the polar coordinate $(r,\theta)$ for $M=\Bbb H^2$ will be used below, we first rewrite $H_1$ in the following form
\begin{align}\label{Trdf}
H_1\varphi=-\Delta\varphi+\mathbf{V}_{r}\partial_r\varphi+\mathbf{V}_{\theta}\sinh^{-1}r\partial_{\theta}\varphi+U_r\varphi
+U_{\theta}\varphi
\end{align}
where we denote
\begin{align*}
&\mathbf{V}_r=-\chi(x)\mathcal{A}_r, \mbox{  }\mathbf{V}_{\theta}=-\chi(x)\sinh^{-1}r\mathcal{A}_{\theta}\\
&U_r\varphi=-\chi(x)\mathcal{A}_r\mathcal{A}_r\varphi+\chi(x)(\widehat{\phi}_r\wedge\varphi)\widehat{\phi}_r\\
&U_{\theta}\varphi=-\chi(x)\sinh^{-2}r\mathcal{A}_{\theta}\mathcal{A}_{\theta}\varphi+\chi(x)\sinh^{-2}r(\widehat{\phi}_{\theta}\wedge\varphi)\widehat{\phi}_{\theta}
\end{align*}

Noticing that due to the fact $H$ is independent of the coordinates chosen for $M$, the results obtained in the polar coordinates can be directly transformed to coordinates given by (2.1).

Before stating the following lemma, we remark that although $H_1$ is not self-adjoint due to the cutoff function $\chi$, the numerical range of $H_1$ is still contained in the real line. This can be verified \underline{by applying the Coulomb condition.} Thus, $\frac{1}{4}\pm\lambda^2+i\epsilon$ lies in the resolvent set of $H_1$ for any $\epsilon>0$.
\begin{lemma}\label{long09}
Let $0<\alpha\ll 1$ be fixed and let $\delta>0$.  If $\lambda_0>0$ is sufficiently large depending on $\alpha,\delta$, then for all $\lambda>\lambda_0$, $\epsilon\in [0,1]$, it holds that
\begin{align}\label{long27}
\left\|\rho^{\alpha}R_{H_1}(\frac{1}{4}+\lambda^2\pm i\epsilon)\rho^{\alpha}\right\|_{L^2\to L^2}\le C\lambda^{-1},
\end{align}
where $C$ is independent of $\delta,\epsilon$.
\end{lemma}
\begin{proof}
See \cite{CCV} for a $\Bbb R^n$ analogy. We will only prove (\ref{long27}) for $R_{H_1}(\frac{1}{4}+\lambda^2+ i\epsilon)$, the negative sign case is the same.
Introduce a weight function
\begin{align*}
\psi_{\alpha}(r)=(\tanh r)^{\alpha-1}\rho^{\alpha}.
\end{align*}
Let $(r,\theta)\in\Bbb R^+\times[0,2\pi]$ be the polar coordinates for $\Bbb H^2$, and $u=\sinh^{\frac{1}{2}}r f$. Let $X=(\Bbb R^+\times[0,2\pi],drd\theta)$.
Define the operator $P:g\longmapsto Pg$ by
\begin{align*}
Pg=\lambda^{-2}\sinh^{\frac{1}{2}}r(H_1-\frac{1}{4}-\lambda^2+i\epsilon)\left(\sinh^{-\frac{1}{2}}rg\right).
\end{align*}
Since $\rho^{\alpha}\le\psi_{\alpha}$, for (\ref{long27}) it suffices to prove
\begin{align}\label{long}
\|\psi_{\alpha}u\|_{L^2(X)}\le C\lambda \|\psi^{-1}_{\alpha}Pu\|_{L^2(X)}.
\end{align}
{\bf All the inner product $\langle \cdot,\cdot\rangle$ in this proof denotes  $\langle \cdot,\cdot\rangle_{L^2(X)}$.}
Denote
\begin{align*}
\mathcal{D}_r=i\lambda^{-1}\partial_r,\mbox{  }\mathcal{D}_{\theta}=i\lambda^{-1}\partial_{\theta}.
\end{align*}
It is easy to verify
\begin{align*}
Pu&=\mathcal{D}^2_ru+\sinh^{-2}r\mathcal{D}^2_{\theta}u-\lambda^{-2}\frac{1}{4}\frac{\cosh^2r}{\sinh^2r}u+\lambda^{-2}Lu
+(i\epsilon\lambda^{-2}-1)u\\
&+\lambda^{-2}\mathbf{V}_r\partial_ru
+\lambda^{-2}\mathbf{V}_{\theta}\sinh^{-1}r\partial_{\theta}u,
\end{align*}
where $L=\frac{1}{2}-\frac{1}{2}\mathbf{V}_r\frac{\cosh r}{\sinh r}+U_r+U_{\theta}$. (see (\ref{Trdf}) for $\mathbf{V}_r,U_{r,\theta}$)\\
\noindent Divide $L$ into the long range part $L_0$ and the short range part $L_1$ by
\begin{align}
L_0=\frac{1}{2}+U_r+U_{\theta},\mbox{  }L_1=-\frac{1}{2}\mathbf{V}_r\frac{\cosh r}{\sinh r}.
\end{align}
Define the energy functional $E(r)$ by
\begin{align}
E(r)&=\|\mathcal{D}_ru\|^2_2+\langle\lambda^{-2}\sinh^{-2} r\Lambda_{\theta}u+u,u\rangle-\lambda^{-2}\langle L_0u,u\rangle\label{xscz6v9}\\
&-\lambda^{-2}\Re\langle \mathbf{V}_{\theta}\sinh^{-1}r\partial_{\theta}u,u\rangle,\label{z6v9}
\end{align}
where we denote
\begin{align*}
\Lambda_{\theta}u=\partial^2_{\theta}u-\frac{1}{4}(\cosh^2 r) u.
\end{align*}
Then by direct calculations we have
\begin{align*}
\frac{dE}{dr}&=\underbrace{\lambda^{-2}\langle \partial_ru,\partial^2_ru\rangle+\lambda^{-2}\langle \partial^2_ru,\partial_ru\rangle}_{d}-2\langle \lambda^{-2}(\cosh r)\sinh ^{-3}r\Lambda_{\theta}u,u\rangle\\
&-\lambda^{-2}\langle (\partial_r L_0)u,u\rangle-\lambda^{-2}\Re\langle\partial_r(\mathbf{V}_{\theta}\sinh^{-1}r)\partial_{\theta}u,u\rangle\\
&+\underbrace{\langle(\lambda^{-2}\sinh^{-2}r\Lambda_{\theta}+1)\partial_ru,u\rangle}_{e}-\underbrace{\lambda^{-2}\langle L_0\partial_ru,u\rangle}_{c}-\underbrace{\lambda^{-2}\Re\langle \mathbf{V}_{\theta}\sinh^{-1}r\partial_{\theta}\partial_{r}u,u\rangle}_{a}\\
&+\underbrace{\langle\lambda^{-2}\sinh^{-2}r\Lambda_{\theta}u+u,\partial_ru\rangle}_{e}-\underbrace{\lambda^{-2}\langle L_0u,\partial_ru\rangle}_{c}-\underbrace{\lambda^{-2}\Re\langle \mathbf{V}_{\theta}\sinh^{-1}r\partial_{\theta}u,\partial_ru\rangle}_{a}\\
&-\frac{1}{2}(\sinh^{-1}r)\lambda^{-2}\cosh r\langle u,u\rangle.
\end{align*}
Meanwhile, we have for $\widetilde{P}=P-i\epsilon\lambda^{-2}-\lambda^{-2}L_1$,
\begin{align}
2\lambda \Im \langle\widetilde{P}u,\mathcal{D}_ru\rangle=K+\overline{{K}},
\end{align}
where $K$ denotes
\begin{align*}
K&=\underbrace{\lambda^{-2}\langle\partial^2_ru,\partial_ru\rangle}_{d}+
\underbrace{\langle(\lambda^{-2}\sinh^{-2}r\partial^2_{\theta}+1)u,\partial_ru\rangle}_{e}-\underbrace{\lambda^{-2}\langle L_0u,\partial_ru\rangle}_{c}\\
&-\underbrace{\langle\lambda^{-2} \mathbf{V}_{\theta}\sinh^{-1}r\partial_{\theta}u,\partial_ru\rangle}_{a}-\underbrace{\lambda^{-2}\langle \mathbf{V}_r\partial_ru,\partial_ru\rangle}_{b}.
\end{align*}
Integration by parts with respect to $\theta$ in $[0,2\pi]$ yields
\begin{align}
\langle \partial^2_{\theta}\partial_ru,u\rangle=\langle \partial_ru,\partial^2_{\theta}u\rangle.
\end{align}
By integration by parts and the skew-symmetry of $\mathbf{V}_{\theta}$, one obtains
\begin{align}
\langle \mathbf{V}_{\theta}\partial_{\theta}\partial_ru,u\rangle=\langle\partial_ru, \mathbf{V}_{\theta}\partial_{\theta}u\rangle
-\langle\partial_{\theta}(\mathbf{V}_{\theta})\partial_ru,u\rangle.
\end{align}
Hence we obtain
\begin{align}
&\Re\langle \sinh^{-1}r\mathbf{V}_{\theta}\partial_{\theta}\partial_ru,u\rangle
+\Re\langle \sinh^{-1}r\mathbf{V}_{\theta}\partial_{\theta}u,\partial_ru\rangle\nonumber\\
&=-\Re\langle \sinh^{-1}r(\partial_{\theta}\mathbf{V}_{\theta})\partial_ru,u\rangle+2\Re\langle \sinh^{-1}r\mathbf{V}_{\theta}\partial_{\theta}u,\partial_ru\rangle.\label{2w2}
\end{align}
This will lead to part cancellation of all the terms with low index $(a)$ in $\frac{dE}{dr}-2\lambda\langle \widetilde{P}u,\mathcal{D}_ru\rangle$.
Since $L_0$ is self-adjoint, we have
\begin{align}\label{2w1}
\langle L_0\partial_ru,u\rangle+\langle L_0u,\partial_ru\rangle-2\Re\langle L_0u,\partial_ru\rangle=0.
\end{align}
This will lead to part cancellation of all the terms with low index $(c)$. By (\ref{2w1}), (\ref{2w2}) it is easy to see
all the terms with low index $(d)$ appearing in $\frac{dE}{dr}-2\lambda\langle \widetilde{P}u,\mathcal{D}_ru\rangle$ cancel.
And by $\Re\langle \mathbf{V}_r\partial_ru,\partial_ru\rangle=0$ ($\mathbf{V}_r$ is skew-symmetric), the (b) term in $K+\bar{K}$ vanishes. Therefore, by (\ref{2w1}), (\ref{2w2}) and comparing the terms according to their low indexes we conclude
\begin{align}
\frac{dE}{dr}&=2\lambda\langle \widetilde{P}u,\mathcal{D}_ru\rangle-2\lambda^{-2}(\cosh r)(\sinh^{-3} r)\langle\Lambda_{\theta}u,u\rangle\label{kukui}\\
&-\lambda^{-2}\langle (\partial_rL_0)u,u\rangle-\lambda^{-2}\Re\langle\partial_r(\mathbf{V}_{\theta}\sinh^{-1}r)\partial_{\theta}u,u\rangle\nonumber\\
&-\lambda^{-2}\Re\langle(\partial_{\theta}\mathbf{V}_{\theta})\sinh^{-1}r\partial_{r}u,u\rangle-\frac{1}{2}\lambda^{-2}\coth r\langle u,u\rangle\label{upcs}
\end{align}
Since $\Lambda_{\theta}$ is positive, we define $\Lambda^{\frac{1}{2}}_{\theta}$ by $\langle \Lambda^{\frac{1}{2}}_{\theta}u,\Lambda^{\frac{1}{2}}_{\theta}u\rangle=\langle \partial_{\theta}u,\partial_{\theta}u\rangle+\cosh^2r\langle u,u\rangle $. Then by Lemma \ref{polarconnection} and the fact that the supports of $\mathbf{V}_{\theta,r},U_{r,\theta}$ are away from zero and noticing that the non-positive last term in (\ref{upcs}) can be absorbed to the second term on RHS of (\ref{kukui}), one obtains
\begin{align}
\frac{dE}{dr}&\ge \frac{1}{2}\lambda^{-2}(\cosh r)(\sinh^{-3} r)\langle\Lambda^{\frac{1}{2}}_{\theta}u,\Lambda^{\frac{1}{2}}_{\theta}u\rangle\nonumber\\
&-C(\delta)\lambda^{-2}\|\psi_{\alpha}u\|^2_{L^2 (X)}-C(\delta)\lambda^{-1}\|\psi_{\alpha}\mathcal{D}_ru\|^2_{L^2 (X)}-2\lambda N(r),\label{8n}
\end{align}
where $N(r)=\big|\langle \widetilde{P}u,\mathcal{D}_ru\rangle\big|$. Therefore,
\begin{align}
E(r)&=-\int^{\infty}_r\frac{d}{d\widetilde{r}}E(\widetilde{r})d\widetilde{r}\nonumber\\
&\le C(\delta)\lambda^{-1}\|\psi_{\alpha}\mathcal{D}_ru\|^2_{L^2(X)}+C(\delta)\lambda^{-2}\|\psi_{\alpha}u\|^2_{L^2(X)}+2\lambda\int^{\infty}_0N(r)dr.\label{tin6}
\end{align}
And by integration by parts on $[0,2\pi]$, one has the last term in (\ref{z6v9}) is equal to
\begin{align}\label{tin9}
\lambda^{-2}\Re\langle u\sinh^{-1} r\partial_{\theta}\mathbf{V}_{\theta},u\rangle,
\end{align}
which can be absorbed to the $\langle u,u\rangle$ term in (\ref{xscz6v9}) if $\lambda$ is sufficiently large depending on $\delta$.
And it is easy to see for $\lambda$ sufficiently large the term $-\lambda^{-2}\langle L_0u, u\rangle$ can also be absorbed into $\langle u,u\rangle$. Hence, for $\lambda$ sufficiently large $E(r)$ defined by (\ref{xscz6v9}) has the following  bound
\begin{align}\label{tin9}
-E(r)\le  \frac{1}{2}\lambda^{-2}(\sinh^{-2} r)\langle\Lambda^{\frac{1}{2}}_{\theta}u,\Lambda^{\frac{1}{2}}_{\theta}u\rangle.
\end{align}
Multiplying (\ref{8n}) with $\psi_{\alpha}\cosh^{-1}r\sinh r$, by integration by parts we deduce
\begin{align}\label{7bvc}
\int^{\infty}_0\psi_{\alpha}\cosh^{-1}r\sinh rE'(r)dr=-\int^{\infty}_0 \frac{d}{dr}(\psi_{\alpha}\cosh^{-1}r\sinh r)E(r)dr.
\end{align}
Since $\frac{d}{dr}(\psi_{\alpha}\cosh^{-1}r\sinh r)<c\alpha\psi_{\alpha}$ with a universal constant $c>0$, then (\ref{7bvc}) shows
\begin{align}\label{7bv}
\int^{\infty}_0\psi_{\alpha}\cosh^{-1}r\sinh rE'(r)dr\lesssim \alpha\int^{\infty}_0 \psi_{\alpha}|E(r)|dr.
\end{align}
Meanwhile (\ref{tin6}) and (\ref{tin9}) imply
\begin{align}
\int^{\infty}_0\psi_{\alpha}|E(r)|dr&\le C(\delta,\alpha)\lambda^{-1}\|\psi_{\alpha}\mathcal{D}_ru\|^2_{L^2(X)}
+C(\delta,\alpha)\lambda^{-2}\|\psi_{\alpha}u\|^2_{L^2(X)}\nonumber\\
&+2C(\alpha)\lambda\int^{\infty}_0N(r)dr+\frac{1}{2}\lambda^{-2}\|\psi^{\frac{1}{2}}_{\alpha}(\sinh^{-1}r)\partial_{\theta}u\|^2_{L^2(X)}.\label{7bm}
\end{align}
Therefore by (\ref{7bm}), (\ref{7bv}) and (\ref{8n}), we obtain
\begin{align*}
&\frac{1}{2}\lambda^{-2}\|\psi^{\frac{1}{2}}_{\alpha}(\sinh^{-1}r)\partial_{\theta}u\|^2_{L^2(X)}\nonumber\\
&\le c\alpha\lambda^{-2}\|\psi^{\frac{1}{2}}_{\alpha}(\sinh^{-1}r)\partial_{\theta}u\|^2_{L^2(X)}
+C(\delta,\alpha)\lambda^{-2}\|\psi_{\alpha}u\|^2_{L^2(X)}\nonumber\\
&+C(\delta,\alpha)\lambda^{-1}\|\psi_{\alpha}\mathcal{D}_ru\|^2_{L^2(X)}+2C(\alpha)\lambda\int^{\infty}_0N(r)dr.
\end{align*}
Let $0<\alpha\ll 1$  be fixed say $\alpha=1/100$. Then we conclude
\begin{align}
&\lambda^{-2}\|\psi^{\frac{1}{2}}_{\alpha}(\sinh^{-1}r)\partial_{\theta}u\|^2_{L^2(X)}\nonumber\\
&\le C(\delta,\alpha)\lambda^{-2}\|\psi_{\alpha}u\|^2_{L^2(X)}
+C(\delta,\alpha)\lambda^{-1}\|\psi_{\alpha}\mathcal{D}_ru\|^2_{L^2(X)}+2C(\alpha)\lambda\int^{\infty}_0N(r)dr.\label{k1}
\end{align}
Meanwhile, since $\|L_0\|_{L^{\infty}}$ is bounded ({\bf V} vanishes near zero) , one has by (\ref{z6v9}) that when $\lambda\gg 1$
\begin{align}
\int^{\infty}_0\psi_{\alpha}(r)E(r)dr&\ge \|\psi^{\frac{1}{2}}_{\alpha}\mathcal{D}_ru\|^2_{L^(X)}+\frac{1}{2}\|\psi^{\frac{1}{2}}_{\alpha}u\|^2_{L^(X)}\nonumber\\
&-2\lambda^{-2}\|\psi^{\frac{1}{2}}_{\alpha}\sinh^{-1}r\Lambda^{\frac{1}{2}}_{\theta}u\|^2_{L^2(X)}.\label{k2}
\end{align}
Combing (\ref{k1}), (\ref{7bm}) with (\ref{k2}), we arrive at
\begin{align}
&\|\psi^{\frac{1}{2}}_{\alpha}u\|^2_{L^2(X)}+\|\psi^{\frac{1}{2}}_{\alpha}\mathcal{D}_ru\|^2_{L^2(X)}
+\|\psi^{\frac{1}{2}}_{\alpha}\sinh^{-1}r\Lambda^{\frac{1}{2}}_{\theta}u\|^2_{L^2(X)}\nonumber\\
&\le C(\delta,\alpha)\lambda^{-1}\|\psi^{\frac{1}{2}}_{\alpha}\mathcal{D}_ru\|^2_{L^2(X)}+C(\delta,\alpha)\lambda^{-2}\|\psi^{\frac{1}{2}}_{\alpha}u\|^2_{L^2(X)}\nonumber\\
&+C(\alpha)\lambda\int^{\infty}_0N(r)dr.\label{7nmbv}
\end{align}
Define $\mathcal{P}=\widetilde{P}+i\epsilon\lambda^{-2}=P-\lambda^{-2}L_1$,  $M^*=|\langle \mathcal{P}u,\mathcal{D}_ru\rangle|$, and $M(r)=|\langle Pu,\mathcal{D}_ru\rangle|$. By Lemma  \ref{polarconnection} and the support of $\bf V$, we see
\begin{align}
\lambda \int^{\infty}_0M^*(r)dr&\le C(\delta)\lambda^{-1}\|\psi^{\frac{1}{2}}_{\alpha}u\|^2_{L^2(X)}+C(\delta)\lambda^{-2}\|\psi^{\frac{1}{2}}_{\alpha}\mathcal{D}_ru\|^2_{L^2(X)}\nonumber\\
&+ \mu^{-1}\lambda^2\|\psi^{-\frac{1}{2}}_{\alpha}Pu\|^2_{L^2(X)}+\mu\|\psi^{\frac{1}{2}}_{\alpha}\mathcal{D}_ru\|^2_{L^2(X)}.\label{7z3nmk}
\end{align}
Now we are ready to show (\ref{long}) is a corollary of the following {\bf Claim}: when $\lambda\gg1$
\begin{align}
&\epsilon\lambda^{-2}\|u\|^2_{L^2(X)}\le \int^{\infty}_0|\langle Pu,u\rangle|dr+C(\delta)\lambda^{-2}\|\psi^{\frac{1}{2}}_{\alpha}u\|^2_{L^2(X)}\label{k4} \\
&\|\mathcal{D}_ru\|^2_{L^2(X)}\le 2\int^{\infty}_0|\langle Pu,u\rangle|dr+4\|u\|^2_{L^2(X)}\label{k5}
\end{align}
In fact, inserting (\ref{k4}), (\ref{k5}) and (\ref{7z3nmk}) to the inequality
$$N(r)\le M^*(r)+\lambda^{-1}\epsilon\big(\|u\|^{2}_{L^2{(X)}}+\|\mathcal{D}_ru\|^2_{L^2(X)}\big),
$$
one immediately obtains that the RHS of (\ref{7nmbv}) is bounded by
\begin{align}
&C(\delta,\alpha)\lambda^{-1}\|\psi^{\frac{1}{2}}_{\alpha}\mathcal{D}_ru\|^2_{L^2(X)}
+C(\delta,\alpha)\lambda^{-2}\|\psi^{\frac{1}{2}}_{\alpha}u\|^2_{L^2(X)}\nonumber\\
&+C(\alpha)\mu^{-1}\lambda^2\|\psi^{\frac{1}{2}}_{\alpha}Pu\|^2_{L^2(X)}+\mu\|\psi^{\frac{1}{2}}_{\alpha}u\|^2_{L^2(X)}. \label{k51}
\end{align}
Let $0<\alpha\ll1$ first be determined, then take $0<\mu\ll1$, and finally let $\lambda\gg1$ depending on the size of $C(\delta,\alpha)$. Then (\ref{k51}) can be absorbed by the left of (\ref{7nmbv}) and thus giving (\ref{long}) with $C$ independent of $\delta$.

Hence it remains to verify (\ref{k4}) and (\ref{k5}).
Consider $\int^{\infty}_0\Re \langle Pu,u\rangle dr$, then (\ref{k5}) follows easily by integration by parts and the $L^{\infty}$ bounds of $\partial_r\mathbf{V}_r$ and $\mathbf{V}_{\theta}$ implied by Lemma \ref{polarconnection} (recall that $\bf V$ vanishes near zero).  Applying integration by parts and Lemma \ref{polarconnection}, we also have (\ref{k4}) by considering $\int^{\infty}_0\Im \langle Pu,u\rangle dr$,.
\end{proof}

We also need the gradient resolvent estimates for $H_1$.
\begin{lemma}\label{long90}
Let $0<\alpha\ll 1$ be fixed and let $\delta>0$.  If $\lambda_0>0$ is sufficiently large depending on $\alpha,\delta$, then for all $\lambda>\lambda_0$, $\epsilon\in [0,1]$, it holds that
\begin{align}\label{long2}
\|\psi^{\frac{1}{2}}_{\alpha}\nabla R_{H_1}(\frac{1}{4}+\lambda^2\pm i\epsilon)\psi^{\frac{1}{2}}_{\alpha}\|_{L^2\to L^2}\le C,
\end{align}
where $C$ is independent of $\delta,\epsilon$.
\end{lemma}
\begin{proof}
As before we only prove the positive sign $\nabla R_{H_1}(\frac{1}{4}+\lambda^2+ i\epsilon)$.
(\ref{k51}) and (\ref{7nmbv}) yield
\begin{align*}
&\|\psi^{\frac{1}{2}}_{\alpha}(\sinh^{-1}r)\partial_{\theta}u\|_{L^2(X)}
+\|\psi^{\frac{1}{2}}_{\alpha}(\coth r) u\|_{L^2(X)}\\
&+\|\psi^{\frac{1}{2}}_{\alpha}u\|_{L^2(X)}+\|\psi^{\frac{1}{2}}_{\alpha}\mathcal{D}_ru\|_{L^2(X)}\\
&\lesssim \lambda\|\psi^{\frac{1}{2}}_{\alpha}Pu\|_{L^2(X)}.
\end{align*}
Recall $u=\sinh^{\frac{1}{2}}rf$, $X=\Bbb R^+\times[0,2\pi]$.  $\mathcal{D}_r=i\lambda^{-1}\partial_r$, $\mathcal{D}_{\theta}=i\lambda^{-1}\partial_{\theta}$, $P=\lambda^{-2}\sinh^{\frac{1}{2}}r(H_1-\frac{1}{4}-\lambda^2+i\epsilon)\sinh^{-\frac{1}{2}}r$, then one has
\begin{align*}
\lambda\|\psi^{\frac{1}{2}}_{\alpha}Pu\|_{L^2(X)}&=\lambda^{-1}\|\psi^{\frac{1}{2}}_{\alpha}(H_1-\frac{1}{4}-\lambda^2+i\epsilon)f\|_{L^2(\Bbb H^2)}\\
\|\psi^{\frac{1}{2}}_{\alpha}\sinh^{-1}r\mathcal{D}_{\theta}u\|_{L^2(X)}&=\lambda^{-1}\|\psi^{\frac{1}{2}}_{\alpha}(r)\nabla_{\theta}f\|_{L^2(\Bbb H^2)},
\end{align*}
where with abuse of notation we denote $\nabla_{\theta}f=\partial_{\theta}f d\theta $, $\nabla_{r}f=\partial_{r}fdr$.
Thus we have
\begin{align}
\|\psi^{\frac{1}{2}}_{\alpha}(r)\nabla_{\theta}f\|_{L^2(\Bbb H^2)}+\|\psi^{\frac{1}{2}}_{\alpha}f\|_{L^2(\Bbb H^2)}\le \|\psi^{\frac{1}{2}}_{\alpha}(H_1-\frac{1}{4}-\lambda^2+i\epsilon)f\|_{L^2(\Bbb H^2)}.\label{po90h}
\end{align}
By direct calculations, one has
\begin{align}
\|\lambda^{-1}\psi^{\frac{1}{2}}_{\alpha}\partial_rf\|_{L^2(\Bbb H^2)}\lesssim\|\psi^{\frac{1}{2}}_{\alpha}{\mathcal D}_ru\|_{L^2(X)} +\lambda^{-1}\|\psi^{\frac{1}{2}}_{\alpha}(\sinh^{-\frac{1}{2}}r)\cosh rf\|_{L^2(X)}.\label{7cep0}
\end{align}
Split $X$ into $X_{I}=[1,\infty)\times[0,2\pi]$ and $X_{II}=(0,1]\times[0,2\pi]$. Then it is easily seen that
\begin{align*}
\|\psi^{\frac{1}{2}}_{\alpha}(\sinh^{-\frac{1}{2}}r)\cosh r\psi^{\frac{1}{2}}_{\alpha}f\|_{L^2(X_I)}\le C\|\psi^{\frac{1}{2}}_{\alpha}f\|_{L^2(\Bbb H^2)}.
\end{align*}
Meanwhile we have
\begin{align}
\|\psi^{\frac{1}{2}}_{\alpha}(\sinh^{-\frac{1}{2}}r)\cosh r\psi^{\frac{1}{2}}_{\alpha}f\|_{L^2(X_{II})}\lesssim \|\psi^{\frac{1}{2}}_{\alpha}(\sinh^{-1}r)\cosh r\psi^{\frac{1}{2}}_{\alpha}u\|_{L^2(X_{II})}.\label{9786bvf}
\end{align}
Thus by (\ref{po90h}), (\ref{9786bvf}) and (\ref{7cep0}), we conclude
\begin{align}\label{y9vfg}
 \|\psi^{\frac{1}{2}}_{\alpha}f\|_{L^2(\Bbb H^2)}+\|\psi^{\frac{1}{2}}_{\alpha}\nabla f\|_{L^2(\Bbb H^2)}\lesssim \|\psi^{-\frac{1}{2}}_{\alpha}(-\Delta-\frac{1}{4}-\lambda^2\pm i\epsilon)f\|_{L^2(\Bbb H^2)}.
\end{align}
\end{proof}

Now we transform the resolvent estimates for the long range part $H_1$ to the full Schr\"odinger operator $H$ by viewing the short range part as a $\psi^{\frac{1}{2}}_{\alpha}L^{\infty}$ perturbation for $H_1$.
\begin{lemma}\label{h}
Let $0<\alpha\ll 1$ be fixed and let $\delta>0$.  If $\lambda_0>0$ is sufficiently large depending on $\alpha,\delta$, then for all $\lambda>\lambda_0$, $\epsilon\in [0,1]$, it holds that
\begin{align}
\left\|\rho^{\alpha}R_{H}(\frac{1}{4}+\lambda^2\pm i\epsilon)\rho^{\alpha}\right\|_{L^2\to L^2}&\le C\lambda^{-1},\label{long2}\\
\left\|\rho^{\alpha}\nabla R_{H}(\frac{1}{4}+\lambda^2\pm i\epsilon)\rho^{\alpha}\right\|_{L^2\to L^2}&\le C,\label{long21}
\end{align}
where $C$ is independent of $\epsilon,\lambda$.
\end{lemma}
\begin{proof}
We only consider the positive sign. Recall the definitions of far and near electric/magnetic  potentials in (\ref{zqa1}).
Denote $z=\frac{1}{4}+\lambda^2+ i\epsilon$.
We aim to use the formal identity
\begin{align}\label{pg90n}
R_{H}(z)=R_{H_1}(z)(I+(V^{near}+X^{near}) R_{H_1})^{-1}.
\end{align}
Hence we first show for $\lambda>\lambda_0(\delta)$,
\begin{align}\label{pz1}
\|\rho^{-\alpha}(V^{near}+X^{near}) R_{H_1}\rho^{\alpha}\|_{L^2\to L^2}\le o(1).
\end{align}
where $o(1)$ denotes a quantity which tends to zero as $\delta\to0$.
By the identity,
\begin{align}\label{nb7lvc}
R_{H_1}(z)=R_0-R_0(V^{far}+X^{far})R_{H_1}(z),
\end{align}
Denote $\mathcal{A}^{near}=(1-\chi(x))\mathcal{A}$, then by Lemma \ref{somerthing},
\begin{align*}
&\|(\rho^{-\alpha}(V^{near}+X^{near})R_0\rho^{\alpha}\|_{L^2\to L^2}\\
&\le \|\rho^{-\alpha}\psi^{-\frac{1}{2}}_{\alpha}V^{near}\|_{L^{\infty}}\|\psi^{\frac{1}{2}}_{\alpha}R_0\rho^{\alpha}\|_{L^2\to L^2}\\
&+ \|\rho^{-\alpha}|\mathcal{A}^{near}|\psi^{-\frac{1}{2}}_{\alpha}\|_{L^{\infty}}\|\psi^{\frac{1}{2}}_{\alpha}\nabla R_0\rho^{\alpha}\|_{L^2\to L^2}\\
&\le C\delta^{\frac{1}{2}(1-\alpha)},
\end{align*}
where we used Lemma \ref{polarconnection} and the inequality $\psi^{-1}_{\alpha}(r)\le r^{1-\alpha}$ in the supports of $V^{near}$, $\mathcal{A}^{near}$ which are contained in $\{x:d(x,0)\le\delta\}$  in the last line.
The same arguments show
\begin{align*}
\|\rho^{-\alpha}(V^{near}+X^{near}) R_0\rho^{\alpha}\|_{L^2\to L^2}\le C\delta^{\frac{1}{2}(1-\alpha)}.
\end{align*}
Again due to Lemma \ref{somerthing}, one has
\begin{align}
&\big\|\rho^{-\alpha}X^{near} R_0(X^{far})R_{H^1}(z)\rho^{\alpha}\big\|_{L^2\to L^2}\nonumber\\
&\le \|\rho^{-\alpha}\psi^{-\frac{1}{2}}_{\alpha}\mathcal{A}^{near}\|_{L^{\infty}}\|\psi^{\frac{1}{2}}_{\alpha}\nabla R_0\rho^{\alpha}\|_{L^2\to L^2}\|\rho^{-\alpha}\mathcal{A}^{far}\|_{L^{\infty}} \|\rho^{\alpha}\nabla R_{H_1}(z)\rho^{\alpha}\|_{L^2\to L^2}\nonumber\\
&\le C\delta^{\frac{1}{2}},\label{oinb}
\end{align}
where we applied Lemma \ref{long90} in the last line. Hence (\ref{pz1}) follows by (\ref{nb7lvc}) and similar arguments as (\ref{oinb}).
Thus by first choosing $0<\delta\ll1$, then letting $\lambda\gg1$, we have
$$\|\rho^{-\alpha}(V^{near}+X^{near}) R_{H_1}\rho^{\alpha}\|_{L^2\to L^2}\le \frac{1}{2}.
$$
We remark that this is possible because the constant $C$ in Lemma \ref{long09} and Lemma \ref{long90} is independent of $\delta$.
Therefore (\ref{pg90n}) makes sense and we have
\begin{align*}
\|\rho^{\alpha}R_{H}(z)\rho^{\alpha}\|_{L^2\to L^2}&\le 2\|\rho^{\alpha}R_{H_1}(z)\rho^{\alpha}\|_{L^2\to L^2}\le C\lambda^{-1}\\
\|\rho^{\alpha}\nabla R_{H}(z)\rho^{\alpha}\|_{L^2\to L^2}&\le 2\|\rho^{\alpha}\nabla R_{H_1}(z)\rho^{\alpha}\|_{L^2\to L^2}\le C.
\end{align*}
\end{proof}

\subsection{Assemble Resolvent Estimates in All Frequencies}

Lemma \ref{resolvent1}, Lemma \ref{resolvent2} and  Lemma \ref{h} give the desired resolvent estimates for $H$.
\begin{lemma}\label{rhoalpha}
Let $H$ be defined in Lemma \ref{REWSDERF} with $\bf e$ being the Coulomb gauge. The limit $\mathop {\lim }\limits_{z  \to 0,{\mathop{\Im}\nolimits} z  > 0} {\left( {H -\frac{1}{4}- {\lambda ^2} \pm z } \right)^{ - 1}}$ exists strongly in $L^2_x$. And if we denote
\begin{align}\label{Lemma1}
{\mathfrak{R}_H}(\lambda^2  \pm i0) = \mathop {\lim }\limits_{z \to 0,{\mathop{\Im}\nolimits} z  > 0} {\left( {H - \frac{1}{4}-{\lambda ^2} \pm z } \right)^{ - 1}},
\end{align}
then we have
\begin{align}\label{Lemma}
\|\rho^{\alpha}{\mathfrak{R}_H}(\lambda^2  \pm i0)\rho^{\alpha}\|_{L^2\to L^2}\le C\min(1,\lambda^{-1}).
\end{align}
\end{lemma}
\begin{proof}
Lemma \ref{resolvent2} shows ${\mathfrak{R}_0}(\lambda^2  \pm i0)(I+W{\mathfrak{R}_0}(\lambda^2  \pm i0))^{-1}$ makes sense in $L^2_x$ for any $\lambda>0$. Thus for $\Im z>0$ with $\Im z<\delta_*$, we can pass the identity
\begin{align}\label{pert2}
(H-\frac{1}{4}- {\lambda ^2} \pm z)^{-1}=R_0(\lambda^2+\frac{1}{4}\pm z)(I+WR_0(\lambda^2+\frac{1}{4}\pm z))^{-1}
\end{align}
to the following by Lemma \ref{fswe34dre5d}
\begin{align*}
{\mathfrak{R}_H}(\lambda^2  \pm i0)=\mathfrak{R}_0(\lambda^2\pm i0)(I+W\mathfrak{R}_0(\lambda^2\pm i0))^{-1}.
\end{align*}
Moreover the estimates of (\ref{pert2}) can be transformed to ${\mathfrak{R}_H}(\lambda^2  \pm i0)$. Thus Lemma \ref{resolvent1}, Lemma \ref{resolvent2} and  Lemma \ref{h} yield (\ref{Lemma}).
\end{proof}

\begin{proposition}\label{k0vr567}
The operator $\sqrt{H}$ satisfies the resolvent estimates,
\begin{align}\label{fr47}
\|\rho^{\alpha}(\sqrt{H}-z)^{-1}\rho^{\alpha}\|_{L^2\to L^2}\le C,
\end{align}
where $C$ is independent of $z\in\Bbb C\backslash\Bbb R$.
As a corollary the Kato smoothing effect holds
\begin{align}\label{9iuhby7}
\|\rho^{\alpha }e^{\pm it\sqrt{H}}f\|_{L^2_{t,x}}\lesssim \|f\|_{L^2_x},
\end{align}
\end{proposition}
\begin{proof}
The original idea is due to \cite{DFDFDFDf2}.
Since we have the identity for $z\ge0$
\begin{align}\label{rp0}
(\sqrt{H}-z)^{-1}=(\sqrt{H}+z)^{-1}+2z(H-z^2)^{-1},
\end{align}
and $z\in(-\infty,0)$ belongs to the resolvent of $\sqrt{H}$,  by Lemma \ref{rhoalpha}, Lemma \ref{fswe34dre5d} and  Phragm$\acute{e}$n-Lindel\"of theorem, it suffices to prove (\ref{fr47}) for $z\in\Bbb R$.\\
\noindent{\bf Case 1.} Let $z\le 0$. Proposition \ref{CXXXXE} shows
\begin{align*}
\|(\sqrt{H}-z)f\|^2_{L^2_x}&=\|\sqrt{H}f\|^2_{L^2_x}-2\Re z\langle f,\sqrt{H}f\rangle\ge \|\sqrt{H}f\|^2_{L^2_x}\ge \langle f,Hf\rangle\\
&\ge C\|f\|_{L^2_x}.
\end{align*}
Thus we have
\begin{align*}
\|(\sqrt{H}+z)^{-1}f\|^2_{L^2_x}\lesssim \|f\|_{L^2_x},
\end{align*}
from which
\begin{align}\label{fr43}
\|\rho^{\alpha}(\sqrt{H}-z)^{-1}\rho^{\alpha}f\|_{L^2_x}\le C\|f\|_{L^2_x}
\end{align}
{\bf Case 2.} Let $z<0$.
Then (\ref{fr47}) follows by (\ref{rp0}), Lemma \ref{rhoalpha} and (\ref{fr43}). Then applying Kato's theorem (see Theorem \ref{TREWSD}) gives
(\ref{9iuhby7}).

\end{proof}

\section{Smoothing effects for Heat Semigroup and Strichartz Estimates}

\subsection{Heat Semigroup Generated by the Magnetic Schr\"odinger Operator}

Let $0<\delta<\delta_2<\delta_1<\frac{1}{4}$ be given constant to be determined.

In this section, we shall consider the $L^p$-$L^p$ estimates and smoothing effects for the heat semigroup $e^{-tH}$.
\begin{lemma}\label{qiu1}
Let $H$ be defined in Lemma \ref{REWSDERF}. Then for any $p\in(1,\infty)$ and some $0<\delta_1\ll 1$, we have
\begin{align}\label{qiu1PKL}
\|e^{-tH}f\|_{L^p_x}\le e^{-\delta_1 t}\|f\|_{L^p_x}.
\end{align}
\end{lemma}
\begin{proof}
Denote $e^{-tH}f(x)=u(t,x)$. Considering the tangent vector field $u\bf e$ defined by $u{\bf e}=u_1{\bf e}_1+u_2{\bf e}_2$, one deduces by the proof in {Proposition 3.1 }that
\begin{align*}
\partial_t|u{\bf e}|^2-\Delta|u{\bf e}|^2+2|\nabla(u{\bf e})|^2\le 0,
\end{align*}
which further yields
\begin{align*}
\partial_t|u{\bf e}|-\Delta|u{\bf e}|\le 0.
\end{align*}
Thus by maximum principle,
\begin{align*}
|u{\bf e}|(t,x)\le e^{t\Delta}|f(x)|.
\end{align*}
Then (\ref{qiu1PKL}) follows by (\ref{mm8}).
\end{proof}

\begin{lemma}\label{qiu2}
Let $H$ be defined in Lemma \ref{REWSDERF}. Then for $\lambda=-\delta_2+i\eta$ with $0<\delta_2<\delta_1 <\frac{1}{4}$, $\eta\in\Bbb R$, $p\in(1,\infty)$ we have
\begin{align}\label{j8bv}
\|(-H-\lambda)^{-1}f\|_{L^p_x}\le C\|f\|_{L^p_x},
\end{align}
where $C$ is independent of $\eta$.
\end{lemma}
\begin{proof}
Assume that $f\in L^2\cap L^p$, by standard theories of semigroups of linear operators, for $\lambda>0$
\begin{align}\label{oi43v}
(-H-\lambda)^{-1}f=\int^{\infty}_0e^{-\lambda t}e^{-tH}fdt.
\end{align}
Let $\mathcal{E}(\lambda)$ be the RHS of (\ref{oi43v}) with $\lambda$ slightly enlarged to $\mathcal{O}=\{\lambda\in \Bbb C:\Re \lambda>-\delta_2\}$, i.e.
\begin{align}\label{o43v}
\mathcal{E}(\lambda)=\int^{\infty}_0e^{-\lambda t}e^{-tH}dt,\mbox{  }\lambda\in \mathcal{O}.
\end{align}
First we verify (\ref{o43v}) converges in the uniform operator topology in $\mathcal{L}(L^p;L^p)$ for $\lambda\in\mathcal{O}$.  In fact, Lemma \ref{qiu1} shows
\begin{align}\label{0j8bv}
\int^{\infty}_0\left\|e^{-\lambda t}e^{-tH}\right\|_{L^p_x\to L^p_x}dt\le \int^{\infty}_0 e^{-\delta_1 t+\delta_2 t}dt.
\end{align}
Since $\mathcal{E}(\lambda)$ is analytic with respect to $\lambda\in\mathcal{O}$, (\ref{oi43v}) shows $(-H-\lambda)^{-1}$ coincides with  $\mathcal{E}(\lambda)$ when $\lambda\in\mathcal{O}$. Thus (\ref{0j8bv}) implies (\ref{j8bv}).
\end{proof}

\begin{lemma}\label{qiu3}
Let $H$ be defined in Lemma \ref{REWSDERF}. Then for any $p\in(1,\infty)$, $t>0$, we have
\begin{align}\label{kn167}
\|He^{-tH}f\|_{L^p_x}&\le C_\epsilon t^{-1-\epsilon}e^{-\delta_3 t}\|f\|_{L^p_x},
\end{align}
for some $\delta_3>0$ and any $\epsilon\in(0,1)$.
\end{lemma}
\begin{proof}
By standard theories of semigroups of linear operators (e.g. [Theorem 7.7,\cite{Pazy}]), for any $f\in L^p\cap L^2$, $t>0$,
\begin{align}\label{0p8bv}
e^{-tH}f=\frac{1}{2\pi i}\int_{\Gamma}e^{\lambda t}(-H-\lambda)^{-1}fd\lambda,
\end{align}
where $\Gamma$ is any curve lying in the resolvent set which connects $\infty e^{i\vartheta}$ and $-\infty e^{i\vartheta}$ with $\frac{\pi}{2}<\vartheta<\pi$. And the convergence of the integrand in (\ref{0p8bv}) is in the $L^2$ norm.

Let curve $\Gamma$ be made up of three components:
\begin{align*}
\left\{
  \begin{array}{ll}
    \Gamma_1:=\{-a+{\bf i}\aleph a:a\in(\delta_2,\infty ]\}  & \hbox{ } \\
   \Gamma_2 := \{-\delta_2+{\bf i}\aleph\eta:\eta\in[-\delta_2, \delta_2]\}   & \hbox{ } \\
   \Gamma_3 := \{-a-{\bf i}\aleph a:a\in[\delta_2,\infty]\}.  & \hbox{ }
  \end{array}
\right.
\end{align*}
The constant $\aleph>0$ will be chosen to be sufficiently large later. Let ${\rm arg}(z)$ denote the argument of a complex number $z\in \Bbb C$ in the range of $(-\pi,\pi)$.

Given $\lambda\in \Gamma_1$, by definition there exists $\lambda_a\in [\delta_2,\infty)$ such that $\lambda=-\lambda_a+{\bf i}\aleph\lambda_a$. If $\delta_2\le \lambda_a\le \frac{1}{4}$, then it is easy to check $0<{\rm arg}(\lambda+\frac{1}{4})\le \frac{\pi}{2}$ and
$0\le \cot \left({\rm arg}(\lambda+\frac{1}{4})\right)<\frac{2}{\aleph\delta_2}$. Hence by choosing $\aleph>100^{k}\delta^{-1}_2$ for sufficiently large $k\in \Bbb N^+$, one obtains $|\cos \left(\frac{1}{2}{\rm arg}(\lambda+\frac{1}{4})\right)- \frac{\sqrt{2}}{2}|<\frac{8}{\aleph\delta_2}$. Thus we get if $\aleph>100^{k}\delta^{-1}_2$, $\lambda\in\Gamma_1$ and $ -\frac{1}{4}\le \Re \lambda\le -\delta_2$, then the number $\sigma\in \Bbb C$ in the half regime $\{\sigma\in \Bbb C:\Re \sigma\ge 0\}$ which solves $\sigma^2-\frac{1}{4}=\lambda$, must satisfy
\begin{align}\label{plm1}
\Re \sigma&\ge \frac{1}{4}|\lambda+\frac{1}{4}|^{\frac{1}{2}}\ge \frac{1}{4}\sqrt{ {\Im{\lambda}}}= \frac{1}{4}\sqrt{\aleph \lambda_a}\ge \frac{1}{4}\sqrt{\aleph \delta_2}\ge 10^{k-1}.
\end{align}
Similarly, if $\aleph>100^{k}\delta^{-1}_2$, $\lambda\in\Gamma_1$  with $\lambda=-\lambda_a+{\bf i}\aleph\lambda_a$, and $ \lambda_a>\frac{1}{4}$, then
one has $\frac{\pi}{2}<{\rm arg}(\lambda+\frac{1}{4})<\pi$ and
$$0<\cot \left(\pi-{\rm arg}(\lambda+\frac{1}{4})\right)<\frac{2}{\aleph\delta_2}.
$$
Thus $|\cos \frac{1}{2}\left(\pi-{\rm arg}(\lambda+\frac{1}{4})\right)- \frac{\sqrt{2}}{2}|<\frac{8}{\aleph\delta_2}$. Since the number $\sigma $ in the half regime $\{\sigma\in \Bbb C:\Re \sigma\ge 0\}$ solving $\sigma^2-\frac{1}{4}=\lambda$ must satisfy
$${\rm arg}{\sigma}=\frac{\pi}{2}-\frac{1}{2}\left(\pi-{\rm arg}(\lambda+\frac{1}{4})\right),
$$
we deduce that if $\aleph>100^{k}\delta^{-1}_2$, $\lambda\in\Gamma_1$ and $\Re \lambda\le -\frac{1}{4}$, then $\Re \sigma \ge  10^{k-1}$.

Therefore, we conclude for $\lambda\in \Gamma_1$ the number $\sigma$ in $\{\sigma\in \Bbb C:\Re \sigma\ge 0\}$ solving $\sigma^2-\frac{1}{4}=\lambda$
satisfies
\begin{align}\label{plm2}
\Re \sigma&\ge  10^{k-1},\mbox{  }{\rm{and}}\mbox{  } {\rm {arg}}(\sigma)\approx \frac{\pi}{4}.
\end{align}
provided that $\aleph>100^{k}\delta^{-1}_2$ and $k\in \Bbb N^+$ is sufficiently large. So letting $k\gg 1$, $\sigma$ satisfies the conditions in Lemma \ref{LP}( the constant $c$ in Lemma \ref{LP} now can be chosen as $\frac{\sqrt{2}}{4}$.) And then we have for $\lambda\in\Gamma_1$ and $r=2,p$,
\begin{align*}
\|{( - H - \lambda )^{ - 1}}f\|_{L^r_x}\le C|\sigma|^{-2+\epsilon}\|f\|_{L^r_x},
\end{align*}
where $\lambda=\sigma^2-\frac{1}{4}$, provided $\aleph>100^{k}\delta^{-1}_2$, $k\gg1$. Since for $\lambda\in \Gamma_1$, $|\lambda|\thicksim_{\aleph}|\sigma|^{2}$, we conclude
\begin{align} \label{p0i1}
\|{( - H - \lambda )^{ - 1}}f\|_{L^r_x}\le  C\min(1,|\lambda|^{-1+\epsilon})\|f\|_{L^r_x},
\end{align}
for $\aleph\delta_2\gg1$ and any $\epsilon\in(0,1)$.

The same arguments show for $\lambda\in\Gamma_3$,  $r=2,p$ and any $\epsilon\in(0,1)$
\begin{align}\label{p0i2}
\|{( - H - \lambda )^{ - 1}}f\|_{L^r_x}\le  C\min(1,|\lambda|^{-1+\epsilon})\|f\|_{L^r_x}.
\end{align}
Lemma \ref{qiu2} shows for $\lambda\in\Gamma_2$ and $r=2,p$,
\begin{align}\label{p0xi2}
\|{( - H - \lambda )^{ - 1}}f\|_{L^r_x}\le  C \|f\|_{L^r_x}.
\end{align}
Thus (\ref{0p8bv}) holds in $L^2\cap L^p$ by using (\ref{p0i1}), (\ref{p0i2}), (\ref {p0xi2}) and noticing
\begin{align}\label{kn7}
|e^{\lambda t}|<e^{-t\mu|\lambda|}
\end{align}
for $\mu=\frac{1}{\sqrt{1+\aleph^2}}$ and $\lambda\in \Gamma_1\cup\Gamma_3$.

Moreover, by (\ref{0p8bv}) we have
\begin{align}
He^{-tH}f&=\frac{1}{2\pi i}\int_{\Gamma}e^{\lambda t}H(-H-\lambda)^{-1}fd\lambda\nonumber\\
&=-\frac{1}{2\pi i}\int_{\Gamma}e^{\lambda t}fd\lambda-\int_{\Gamma}e^{\lambda t}\lambda(-H-\lambda)^{-1}fd\lambda.\label{0bv}
\end{align}
Combining (\ref{p0i1}), (\ref{p0i2}) and (\ref {p0xi2}), we have by (\ref{0bv})
\begin{align}\label{kn9}
\|He^{-tH}f\|_{L^r_x}\lesssim \big(\int^{\infty}_{c(\aleph,\delta_2)}e^{-\mu t\tau}d\tau+\int^{\infty}_{c(\aleph,\delta_2)}e^{-\mu t\tau}\tau^{\epsilon}d\tau+e^{-ct}\aleph^2\big)\|f\|_{L^r_x}
\end{align}
Hence we arrive at
\begin{align*}
\|He^{-tH}f\|_{L^p_x}\lesssim t^{-1-\epsilon}e^{-t\delta_3}\|f\|_{L^p_x},
\end{align*}
for some $\delta_3>0$, thus giving (\ref{kn167}).
\end{proof}

\begin{proposition}\label{k8vr4}
For $s\in(0,2)$, $p\in(1,\infty)$, we have
\begin{align}
\|(-\Delta)^{\frac{s}{2}}f\|_{L^p_x}&\lesssim \|H^{\frac{s}{2}}f\|_{L^p_x}+\|f\|_{L^p_x}\label{p2ib}\\
\|H^{\frac{s}{2}}f\|_{L^p_x}&\lesssim \|(-\Delta)^{\frac{s}{2}}f\|_{L^p_x}+\|f\|_{L^p_x}.\label{p1ib}
\end{align}
\end{proposition}
\begin{proof}
By Lemma \ref{qiu1}, $e^{\delta_1 t}e^{-tH}$, whose infinitesimal generator is $\delta_1-H$, is a $C_0$ semigroup of contractions in $L^p$. Thus Lumer-Phillips theorem or [Corollary 3.6 of \cite{Pazy}] shows $\{\lambda:\Re \lambda>0\}\subset \rho(\delta_1-H)$. And for $\Re\lambda>-\frac{1}{2}\delta_1$,
\begin{align}\label{pknvoi}
\|(-H-\lambda)^{-1}\|_{L^p\to L^p}\le (\Re \lambda+\delta_1)^{-1}.
\end{align}
Due to Balakrishnan formula we deduce for $s\in(0,2)$ (see [Lemma 5.2, \cite{6HDFRG}] for the proof)
\begin{align}\label{p0x3ng}
H^{\frac{s}{2}}=(-\Delta)^{\frac{s}{2}}+c(s)\int^{\infty}_0\lambda^{\frac{s}{2}}(\lambda-\Delta+W)^{-1}W(\lambda-\Delta)^{-1}d\lambda.
\end{align}
By (\ref{p0x3ng}), it is obviously that for (\ref{p2ib}) and (\ref{p1ib}), it suffices to prove
\begin{align*}
\int^{\infty}_0\lambda^{\frac{s}{2}}\|(\lambda-\Delta+W)^{-1}W(\lambda-\Delta)^{-1}\|_{L^p\to L^p}d\lambda\le C.
\end{align*}

{\bf Case 1. $0\le s<1$. } In this case we consider $s\in [0,1)$.
Let $\lambda=-\frac{1}{4}+\sigma^2$, Lemma \ref{something} and (\ref{pknvoi}) give
\begin{align*}
\int^{\infty}_{0}\lambda^{\frac{s}{2}}\|(\lambda+H)^{-1}W(\lambda-\Delta)^{-1}\|_{L^p\to L^p}d\lambda\le C_{\delta_1}\int^{\infty}_{\frac{1}{2}}\sigma^{s-2} d\sigma\lesssim 1,
\end{align*}
provided $s\in [0,1)$, thus yielding (\ref{p2ib}), (\ref{p1ib}) in Case 1.

{\bf Case 2. $s=1$. } In this case we consider $s=1$.
Let $\lambda=-\frac{1}{4}+\sigma^2$, Lemma \ref{something} and (\ref{pknvoi}) show
\begin{align}
&\int^{\infty}_{0}\lambda^{\frac{1}{2}}\|(\lambda+H)^{-1}W(\lambda-\Delta)^{-1}\|_{L^p\to L^p}d\lambda\lesssim C_{\delta_1}\|V\|_{L^{\infty}_x}\int^{\infty}_{\frac{1}{2}}\sigma^{-2} d\sigma\label{8n7yu}\\
&+\int^{\infty}_{\frac{1}{2}} \|X(\sigma^2-\frac{1}{4}-\Delta)^{-1}\|_{L^p\to L^p} d\sigma.\label{n7yu}
\end{align}
by recalling $W:=V+X$ with $X$ denoting the magnetic part. Thus it suffices to deal with the magnetic part (\ref{n7yu}). By the boundedness of Riesz transform we see
\begin{align*}
&\|X(\sigma^2-\frac{1}{4}-\Delta)^{-1}f\|_{L^p}\lesssim \|(-\Delta)^{\frac{1}{2}}(\sigma^2-\frac{1}{4}-\Delta)^{-1}f\|_{L^p}\\
&=\|(\sigma^2-\frac{1}{4}-\Delta)^{-1}(-\Delta)^{\frac{1}{2}}f\|_{L^p}\\
&\lesssim \min (1,|\sigma|^{-2})\|(-\Delta)^{\frac{1}{2}}f\|_{L^p},
\end{align*}
where in the last line we used  Lemma \ref{something}.
Thus (\ref{n7yu}), (\ref{p0x3ng}) further give
\begin{align*}
\|H^{\frac{1}{2}}f\|_{L^p}\lesssim \|(-\Delta)^{\frac{1}{2}}f\|_{L^p}+\|f\|_{L^p},
\end{align*}
which proves (\ref{p1ib}) in Case 2. To prove (\ref{p2ib}), recalling that (\ref{8n7yu}), (\ref{n7yu}), (\ref{p0x3ng}) imply
\begin{align}
\|(-\Delta)^{\frac{1}{2}}f\|_{L^p}&\lesssim \|H^{\frac{1}{2}}f\|_{L^p}+\|f\|_{L^p}+\int^{\infty}_{\frac{1}{2}}\|X(\sigma^2-\frac{1}{4}-\Delta)^{-1}f\|_{L^p} d\sigma.\label{Yu76tg}
\end{align}
Divide (\ref{Yu76tg}) into $\sigma\in[\frac{1}{2},K_1)$ and  $\sigma\in[K_1,\infty)$, and apply   (\ref{L2qPQ}) of Lemma \ref{L1P}  to $\sigma\in[\frac{1}{2},K_1)$ and   (\ref{L2qPP}) of Lemma \ref{L1P} to $\sigma\in[K_1,\infty)$:
\begin{align*}
\left|(\ref{Yu76tg})\right|&\lesssim \int^{K_1}_{\frac{1}{2}}\sigma^{\frac{1}{2}}\|f\|_{L^p} d\sigma+\int^{\infty}_{K_1}\sigma^{-2}\|(-\Delta)^{\frac{1}{2}}f\|_{L^p} d\sigma\\
&\lesssim C(K_1)\|f\|_{L^p} d\sigma+(K_1)^{-1}\|(-\Delta)^{\frac{1}{2}}f\|_{L^p}.
\end{align*}
where we used again the boundedness of Riesz transform in the first line. Therefore, we conclude
\begin{align}\label{6jtg}
\|(-\Delta)^{\frac{1}{2}}f\|_{L^p}\lesssim \|H^{\frac{1}{2}}f\|_{L^p}+\|f\|_{L^p}+(K_1)^{-1}\|(-\Delta)^{\frac{1}{2}}f\|_{L^p},
\end{align}
which gives (\ref{p2ib}) by letting $K_1\gg1$ and absorbing the $(-\Delta)^{\frac{1}{2}}f$ on the RHS of (\ref{6jtg}) to the left.

{\bf Case 3. $s\in (1,2)$. }  As before, (\ref{p0x3ng}), (\ref{pknvoi}) and  (\ref{L2qPP}) of Lemma \ref{L1P} show
\begin{align}
\|(-\Delta)^{\frac{s}{2}}f\|_{L^p}&\lesssim \|H^{\frac{s}{2}}f\|_{L^p}+\|f\|_{L^p}+\int^{\infty}_{K_1}\sigma^{s-3}\|(-\Delta)^{\frac{1}{2}}f\|_{L^p}\nonumber\\
&\lesssim \|H^{\frac{s}{2}}f\|_{L^p}+\|f\|_{L^p}+(K_1)^{s-2}\|(-\Delta)^{\frac{1}{2}}f\|_{L^p},\label{uytfcg}
\end{align}
provided $s\in(1,2)$. Since on $\Bbb H^2$ there holds  $\|(-\Delta)^{\frac{1}{2}}f\|_{L^p}\lesssim \|(-\Delta)^{\frac{s}{2}}f\|_{L^p}$ for $s>1$, we further have by (\ref {uytfcg})
\begin{align}
\|(-\Delta)^{\frac{s}{2}}f\|_{L^p}
&\lesssim \|H^{\frac{s}{2}}f\|_{L^p}+\|f\|_{L^p}+(K_1)^{s-2}\|(-\Delta)^{\frac{s}{2}}f\|_{L^p},
\end{align}
which gives
\begin{align}
\|(-\Delta)^{\frac{s}{2}}f\|_{L^p}
&\lesssim \|H^{\frac{s}{2}}f\|_{L^p}+\|f\|_{L^p},
\end{align}
provided $K_1$ is sufficiently large. Thus (\ref {p2ib}) has been proved.
The same arguments give (\ref{p1ib}) which is the inverse direction of  (\ref {p2ib}).

\end{proof}

\begin{proposition}\label{pkg1}
Let $H$ be defined in Lemma \ref{REWSDERF}. Then for any $p\in(1,\infty)$, $s\in[0,2)$, and any $\epsilon\in (0,1)$, $t>0$, we have for some $\delta >0$
\begin{align}\label{kn16}
\|(-\Delta)^{\frac{s}{2}}e^{-tH}f\|_{L^p_x}&\le Ct^{-s(1+\epsilon)}e^{-\delta  t}\|f\|_{L^p_x}.
\end{align}
\end{proposition}
\begin{proof}
By Lemma \ref{qiu1}, $e^{-tH}$ is a $C_0$ semigroup of contractions in $L^p$. Thus again by Lumer-Phillips theorem or Corollary 3.6 of \cite{Pazy}, $\{\lambda:\Re \lambda>0\}\subset \rho(-H)$, and for such $\lambda$, (\ref{pknvoi}) holds.
And thus by Balakrishnan formula (see for instance in the proof of [Theorem 6.10\cite{Pazy}]), we have
\begin{align}\label{pknv}
\|H^{\frac{s}{2}}f\|_{L^p_x}\le \|f\|^{1-\frac{s}{2}}_{L^p_x}\|Hf\|^{\frac{s}{2}}_{L^p_x}.
\end{align}
Hence by Lemma \ref{qiu3} and Lemma \ref{qiu1},
\begin{align}
\|H^{\frac{s}{2}}e^{-tH}f\|_{L^p_x}&\lesssim e^{-\delta_3 t}t^{-\frac{s(1+\epsilon)}{2}}\|f\|_{L^p_x}\label{p3ib}\\
\|e^{-tH}f\|_{L^p_x}&\lesssim e^{-\delta_1 t}\|f\|_{L^p_x}.\label{p4ib}
\end{align}
Therefore  (\ref{p2ib}), (\ref{p4ib}), (\ref{p3ib}) give for $s\in[0,2)$
\begin{align*}
\|(-\Delta)^{\frac{s}{2}}e^{-tH}f\|_{L^p_x}\le Ce^{-\delta  t}t^{-\frac{s(1+\epsilon)}{2}}\|f\|_{L^p_x},
\end{align*}
for some $\delta>0$.
\end{proof}

When $p=2$, Proposition \ref{pkg1} can be refined to be the following.
\begin{proposition}\label{pkg2}
Let $H$ be defined in Lemma \ref{REWSDERF} with ${\bf e}$ being the Coulomb gauge. Then for any $s\in[0,2]$, $t>0$, we have
\begin{align}
\|H^{\frac{s}{2}}e^{-tH}f\|_{L^2_x}&\lesssim t^{-s}e^{-\delta t}\|f\|_{L^2_x}\label{kn15}\\
\|(-\Delta)^{\frac{s}{2}}f\|_{L^2_x}&\lesssim \|H^{\frac{s}{2}}f\|_{L^2_x}\label{2ib}\\
\|H^{\frac{s}{2}}f\|_{L^2_x}&\lesssim \|(-\Delta)^{\frac{s}{2}}f\|_{L^2_x}.\label{1ib}
\end{align}
\end{proposition}
\begin{proof}
We first prove (\ref{2ib}) and (\ref{1ib}). When $s=2$, By (\ref{0cdek}) and Sobolev embedding,
\begin{align*}
4\|Hu\|_{L^2_x}\|\nabla u\|_{L^2_x}\ge \|Hu\|_{L^2_x}\|u\|_{L^2_x}\ge \langle Hu,u\rangle \ge \langle \nabla u,\nabla u\rangle \ge \frac{1}{4}\|u\|^2_{L^2_x}.
\end{align*}
Then we have
\begin{align*}
\|u\|_{L^2_x}+\|\nabla u\|_{L^2_x}\lesssim \|Hu\|_{L^2_x}.
\end{align*}
Hence one obtains by $|\mathcal{A}|\le C$ and triangle inequality that
\begin{align*}
\|\Delta u\|_{L^2_x}&\lesssim \|H u\|_{L^2_x}+\|\nabla u\|_{L^2_x}\lesssim \|H u\|_{L^2_x}\\
\|H u\|_{L^2_x}&\lesssim \|\Delta u\|_{L^2_x}+\|\nabla u\|_{L^2_x}\lesssim \|\Delta u\|_{L^2_x},
\end{align*}
from which (\ref{2ib}) and (\ref{1ib}) follow.
(\ref{kn15}) follows by the same arguments as Proposition \ref{pkg1} with the following improved resolvent estimates in $L^2_x$:
\begin{align}\label{POIU}
\|(-H-z)^{-1}\|_{L^2\to L^2}\le {\rm {dist}} (z,Line),
\end{align}
where $Line:=\{x\in \Bbb R : x\le -\frac{1}{4}\}$.
Indeed (\ref{POIU}) implies the additional $\epsilon$ in (\ref{p0i1}), (\ref{p0i2}) can be removed. And thus the additional $\epsilon$ in (\ref{0bv}) is not needed and one gets
\begin{align*}
\|He^{-tH}f\|_{L^2_x}\lesssim t^{-1}e^{-\delta_3 t}\|f\|_{L^2_x}.
\end{align*}
Then (\ref{kn15}) follows by interpolation.
\end{proof}

In addition, we need an almost equivalence lemma for $H^{\frac{1}{2}}$ and $(-\Delta)^{\frac{1}{2}}$ in $\rho^{-\alpha}L^2$.
\begin{lemma}\label{py6mx}
Let $0<\alpha\ll\varrho$.
Let $H$ be defined in Lemma \ref{REWSDERF} with ${\bf e}$ being the Coulomb gauge. For $0<\delta_4\ll1$, $\lambda>\frac{1}{4}-\delta_4$, $(I+\rho^{-\alpha}W(-\Delta+\lambda)^{-1}\rho^{\alpha})$ is invertible in $L^2_x$ and analytic with respect to $\lambda$ in any compact set of $\{\lambda:\lambda>\frac{1}{4}-\delta_4\}$.
\end{lemma}
\begin{proof}
Assume $I+\rho^{-\alpha}W(-\Delta+\lambda)^{-1}\rho^{\alpha}$ is not invertible, then by Fredholm's alternative, there exists $f\in L^2_x$ such that
$$f+\rho^{-\alpha}W(-\Delta+\lambda)^{-1}\rho^{\alpha}f=0.$$
Let $(-\Delta+\lambda)^{-1}\rho^{\alpha}f=g$, then by Lemma \ref{somerthing} with (\ref{4der45P}) shows
\begin{align}\label{KHuy}
g\in L^{r} \mbox{ }{\rm for}\mbox{ } r\in(2,6), \mbox{ }\nabla g\in \rho^{-\alpha}L^{2}_x.
\end{align}
Moreover, we have
\begin{align}\label{yu7x6}
-\Delta g+\lambda g+Wg=0.
\end{align}
By (\ref{KHuy}), potentials in $W$ decay exponentially in $\Bbb H^2$, and H\"older, it is easy to check $Wg\in L^2_x$. Thus (\ref{yu7x6}) shows $-\Delta g+(\sigma^2-\frac{1}{4})g\in L^2$ and consequently we have $g\in L^2$ due to $\sigma(-\Delta)\subset[\frac{1}{4},\infty)$. Again by (\ref{yu7x6}), $g$ is an eigenfunction of $-\Delta+W$ with eigenvalue $\frac{1}{4}-\lambda$, which contradicts with Proposition \ref{CXXXXE}. The analyticity of $(I+\rho^{-\alpha}W(-\Delta+\lambda)^{-1}\rho^{\alpha})^{-1}$ claimed in our lemma follows by the fact $\frac{1}{4}-\lambda$ lies in the resolvent set of $-\Delta$ when $\lambda>\frac{1}{4}-\delta_4$.
\end{proof}

\begin{lemma}\label{pjuy}
Let $H$ be defined in Lemma \ref{REWSDERF} with ${\bf e}$ being the Coulomb gauge. We have
\begin{align}
\|\nabla f\|_{\rho^{-\alpha}L^2_x}&\lesssim \|H^{\frac{1}{2}} f\|_{\rho^{-\alpha}L^2_x}+\|f\|_{\rho^{-\alpha}L^2_x}\label{hern}\\
\|\nabla f\|_{\rho^{-\alpha}L^2_x}&\lesssim \|(-\Delta)^{\frac{1}{2}} f\|_{\rho^{-\alpha}L^2_x}\label{sher4}\\
\|(-\Delta)^{\frac{1}{2}} f\|_{\rho^{-\alpha}L^2_x}&\lesssim \|H^{\frac{1}{2}} f\|_{\rho^{-\alpha}L^2_x}+\|f\|_{\rho^{-\alpha}L^2_x}\label{ser4}.
\end{align}
\end{lemma}
\begin{proof}
(\ref{sher4}) has been proved in \cite{6HDFRG}. Since (\ref{hern}) is a corollary of (\ref{sher4}) and (\ref{ser4}), it suffices to prove (\ref{ser4}).
By (\ref{p0x3ng}), it suffices to estimate
\begin{align}\label{ity4}
\int^{\infty}_0\lambda^{\frac{1}{2}}\|(\lambda-\Delta+W)^{-1}W(\lambda-\Delta)^{-1}f\|_{\rho^{-\alpha} L^2}d\lambda.
\end{align}
For any fixed $K_2>0$, by the formal identity
\begin{align}\label{ity}
(\lambda+H)^{-1}=(-\Delta+\lambda)^{-1}(I+W(-\Delta+\lambda)^{-1})^{-1},
\end{align}
Lemma \ref{py6mx} implies for $\lambda\in (0,K_2)$
\begin{align*}
&\|\rho^{\alpha}(\lambda+H)^{-1}W(\lambda-\Delta)^{-1}f\|_{L^2}\\
&\le \|\rho^{\alpha}(-\Delta+\lambda)^{-1}\rho^{\alpha}\|_{L^2\to L^2}
\|(I+\rho^{-\alpha}W(-\Delta+\lambda)^{-1}\rho^{\alpha})^{-1}\|_{L^2\to L^2}\|\rho^{\alpha}W(\lambda-\Delta)^{-1}f\|_{L^2}\\
&\le C(K_2)\|\rho^{\alpha}(-\Delta+\lambda)^{-1}\rho^{\alpha}\|_{L^2\to L^2}\|\rho^{\alpha}W(\lambda-\Delta)^{-1}f\|_{L^2}.
\end{align*}
Then Lemma \ref{something} shows the low and mediate frequency part of the RHS of (\ref{ity4}) is bounded by
\begin{align}
&\int^{K_2}_0\lambda^{\frac{s}{2}}\|(\lambda-\Delta+W)^{-1}W(\lambda-\Delta)^{-1}f\|_{\rho^{-\alpha} L^2}d\lambda\nonumber\\
&\lesssim C(K_2)\int^{K_2}_0\lambda^{\frac{1}{2}}\min(1,\lambda^{-\frac{3}{2}})d\lambda \|f\|_{\rho^{-\alpha} L^2}.\label{ity1}
\end{align}
For the high frequency part of the RHS of (\ref{ity4}), i.e. $\lambda\in (K_2,\infty)$ with $K_2\gg1$, Lemma \ref{something} and Neumann series argument show (\ref{ity}) makes sense in $\rho^{-\alpha}L^2$ and
\begin{align*}
\|\rho^{\alpha}(\lambda+H)^{-1}\rho^{\alpha}f\|_{L^2}&\le 2\|\rho^{\alpha}(-\Delta+\lambda)^{-1}\rho^{\alpha}\|_{L^2\to L^2}\|f\|_{L^2}\nonumber\\
&\le \lambda^{-1}\|f\|_{L^2}.
\end{align*}
Thus we arrive at
\begin{align*}
&\|\rho^{\alpha}(\lambda+H)^{-1}W(\lambda-\Delta)^{-1}f\|_{L^2}\\
&\lesssim  \lambda^{-1}C(W)\left(\|\rho^{\alpha}\nabla(-\Delta+\lambda)^{-1}f\|_{L^2\to L^2}+\|\rho^{\alpha} (-\Delta+\lambda)^{-1}f\|_{ L^2}\right)\\
&\lesssim \lambda^{-2}C(W)\|f\|_{\rho^{-\alpha}L^2}+\lambda^{-1}C(W)\|\rho^{\alpha}\nabla(-\Delta+\lambda)^{-1}f\|_{L^2 },
\end{align*}
where we denoted
\begin{align*}
C(W):=\|\rho^{-2\alpha}\mathcal{A}\|_{L^{\infty}}+\|\rho^{-2\alpha}V\|_{L^{\infty}},
\end{align*}
and we applied  Lemma \ref{something}  in the last inequality.  By the boundedness of $\nabla (-\Delta)^{-\frac{1}{2}}$ in the weighted space
$\rho^{-\alpha}L^2$ (see Lemma 4.6 of \cite{6HDFRG} ) we obtain
\begin{align*}
\|\rho^{\alpha}\nabla(-\Delta+\lambda)^{-1}f\|_{L^2}&\lesssim
\|\rho^{\alpha}(-\Delta)^{\frac{1}{2}}(-\Delta+\lambda)^{-1}f\|_{L^2}\\
&=  \|\rho^{\alpha}(-\Delta+\lambda)^{-1}(-\Delta)^{\frac{1}{2}}f\|_{L^2}\\
&\lesssim\lambda^{-1}\|(-\Delta)^{\frac{1}{2}}f\|_{\rho^{-\alpha}L^2}
\end{align*}
where we applied Lemma \ref{something} in the last line. Hence the high frequency part of the RHS of
(\ref{ity4}) is bounded by
\begin{align} \label{ity2}
C(\int^{\infty}_{K_2}\lambda^{-\frac{3}{2}}d\lambda )\left(\|f\|_{\rho^{-\alpha}L^2}+\|(-\Delta )^{\frac{1}{2}}f\|_{\rho^{-\alpha}L^2}\right).
\end{align}
Then (\ref{ity1}), (\ref{ity2}) with (\ref{p0x3ng}) imply that
\begin{align*}
\|(-\Delta)^{\frac{1}{2}}f\|_{\rho^{-\alpha}L^2}&\lesssim\|H^{\frac{1}{2}}f\|_{\rho^{-\alpha}L^2}+C(K_2)\int^{K_2}_0\lambda^{\frac{1}{2}}\min(1,\lambda^{-\frac{3}{2}})d\lambda \|f\|_{\rho^{-\alpha} L^2}\\
&+(K_2)^{-\frac{1}{2}}  \|(-\Delta)^{\frac{1}{2}}f\|_{\rho^{-\alpha} L^2}.
\end{align*}
Letting $K_2\gg 1$ we obtain (\ref{ser4}) by absorbing the $\|(-\Delta)^{\frac{1}{2}}f\|_{\rho^{-\alpha} L^2}$ on the RHS to the left.

\end{proof}

\subsection{Non-endpoint and endpoint Strichartz estimates }

Theorem 5.2 and Remark 5.5 of Anker, Pierfelice \cite{45XDADER45P} obtained the Strichartz estimates for linear wave/Klein-Gordon equation.

The author's companion paper [Theorem 1.1, \cite{6HDFRG}] proved an abstract theorem which says the non-endpoint plus endpoint Strichartz estimates and the weighted Strichartz estimates are corollaries of the Kato smoothing effects, i.e.,

\begin{theorem}\label{dtfcer}
Let $r(x)=d(x,o)$ be the geodesic distance between $x$ and origin point $o$. Recall $\rho(x)=e^{-r(x)}$.
Let $H=-\Delta+W$ be defined above and  $0<\sigma\ll \alpha\ll 1$. If $H$ satisfies
\begin{align*}
\|H^{\frac{1}{2}}f\|_{L^2}&\lesssim \|(-\Delta )^{\frac{1}{2}}f\|_{L^2}+\|f\|_{L^2}\\
\|(-\Delta)^{\frac{1}{2}}f\|_{L^2}&\lesssim \|H^{\frac{1}{2}}f\|_{L^2}+\|f\|_{L^2}\\
\|\rho^{\alpha }\nabla f\|_{L^2}&\lesssim \|\rho^{\alpha } H^{\frac{1}{2}}f\|_{L^2}+\|\rho^{\alpha }f\|_{L^2},
\end{align*}
and the Kato smoothing effect
\begin{align*}
\|\rho^{\alpha }e^{\pm it\sqrt{H}}f\|_{L^2_{t,x}}\lesssim \|f\|_{L^2_x},
\end{align*}
then we have the weighted Strichartz estimates for the magnetic wave equation: If $u$ solves the equation
\begin{align*}
\left\{ \begin{array}{l}
\partial _t^2u - {\Delta}u + Wu = F \\
u(0,x) = {u_0},{\partial _t}u(0,x) = {u_1} \\
\end{array} \right.
\end{align*}
then it holds for any $p\in(2,6)$
\begin{align*}
&{\left\| {{{D}^{\frac{1}{2}}}u} \right\|_{L_t^2L_x^{p}}} +{\left\| {{\rho ^{ {\alpha } }}\nabla f} \right\|_{L_t^2L_x^2}}+
{\left\| {{\partial _t}u} \right\|_{L_t^\infty L_x^2}} + {\left\| {\nabla u} \right\|_{L_t^\infty L_x^2}}\\
&\lesssim {\left\| {\nabla {u_0}} \right\|_{{L^2}}} + {\left\| {{u_1}} \right\|_{{L^2}}} + {\left\| F \right\|_{L_t^1L_x^2}}.
\end{align*}
\end{theorem}

Then by Proposition \ref{pkg2}, Proposition \ref{k0vr567}, Lemma \ref{pjuy}, we have
\begin{proposition}\label{de486tfcer}
Let $H=-\Delta+W$ be defined above and $0<\sigma\le \alpha\ll 1$ with the frame ${\bf e}$ on $Q^*TN$ fixed by choosing Coulomb gauge. Then we have the weighted Strichartz estimates for the magnetic wave equation: If $u$ solves the equation
\begin{align*}
\left\{ \begin{array}{l}
\partial _t^2u - {\Delta}u + Wu = F \\
u(0,x) = {u_0},{\partial _t}u(0,x) = {u_1} \\
\end{array} \right.
\end{align*}
then it holds for any $p\in(2,6)$
\begin{align*}
&{\left\| {{{ D}^{\frac{1}{2}}}u} \right\|_{L_t^2L_x^{p}}} +{\left\| {{\rho ^{\sigma}}\nabla f} \right\|_{L_t^2L_x^2}}+
{\left\| {{\partial _t}u} \right\|_{L_t^\infty L_x^2}} + {\left\| {\nabla u} \right\|_{L_t^\infty L_x^2}}\\
&\lesssim {\left\| {\nabla {u_0}} \right\|_{{L^2}}} + {\left\| {{u_1}} \right\|_{{L^2}}} + {\left\| F \right\|_{L_t^1L_x^2}}.
\end{align*}
\end{proposition}

\section{The proof of Theorem 1.1.}

\subsection{Bootstrap of the heat tension filed}

Suppose that the initial data $(u_0,u_1)$ satisfies
\begin{align}\label{1988}
{\left\| {(du_0,u_1)} \right\|_{L_t^\infty L_x^2([0,T] \times {\Bbb H^2})}}+ {\left\| {(\nabla du_0,\nabla u_1)} \right\|_{L_t^\infty L_x^2([0,T] \times {\Bbb H^2})}} \le M_0.
\end{align}
Lemma \ref{nabla} shows that
\begin{align}\label{1989}
\|\phi_s(0,0,x)\|_{L^2_x}\le \mu_1C(R_0)C(M_0).
\end{align}

We will prove Theorem 1.1 by bootstrap. {\bf In this section we always fix the frame $\{\bf e_i\}^2_{i=1}$ in Proposition \ref{94obhxdg} to be the Coulomb gauge}.

We stop for a while to clarify the notations of connection coefficients used before.  According to Proposition 4.2, choosing the Coulomb gauge for $Q^*TN$, we denote the operator Schr\"odinger operator $H$ as
\begin{align*}
H:=-\Delta-2h^{ij}\mathcal{A}_i\partial_j-h^{ij}\mathcal{A}_i\mathcal{A}_j-h^{ij}( \widehat{\phi}_i  \wedge\cdot){\widehat{\phi}}_j,
\end{align*}
Moreover, the connection coefficient matrices and differential fields can be decomposed into
\begin{align}\label{awo}
\left\{
  \begin{array}{ll}
    A_i&=\mathcal{A}_i+A^{qua}_i,    \mbox{ }  A^{qua}_i:=\int^{\infty}_s{\phi}_i(s') \wedge \phi_s(s')ds'  \hbox{ } \\
\phi_i&=\widehat{\phi}_i(x)+\int^{\infty}_s\partial_s{\phi}_ids', \mbox{ } \phi^{qua}_i:=\int^{\infty}_s\partial_s{\phi}_ids'.
  \end{array}
\right.
\end{align}

And Proposition 4.2 implies
\begin{align}
\|\mathcal{A}\|_{L^{\infty}_x}+\|\nabla \mathcal{A}\|_{L^{\infty}_x}&\le C(M_0).\label{3.2}
\end{align}

We fix the constants $\mu_1,\varepsilon_1, \varrho, \sigma, M_1$ to be
\begin{align}\label{dozuoki}
0<\mu_1\ll \varepsilon_1\ll e^{-(C(M_1)+10)}, \mbox{  }0<\sigma\ll\varrho\ll 1, \mbox{ }M_0\ll C(M_1),
\end{align}
where the constant $C(M_1)$ is the one in Proposition 2.1 and grows at most polynomially as $M_1\to\infty$.

Let $L>0$ be sufficiently large say $L=10^{10}$.
For any  given constant  $a>0$, define $\omega:\Bbb R^+\to \Bbb R^+$ by
$$\omega_a(s) = \left\{ \begin{array}{l}
 {s^{a}}\mbox{  }{\rm{when}}\mbox{  }0 \le s \le 1 \\
 {s^L}\mbox{  }\mbox{  }\mbox{  }{\rm{when}}\mbox{  }s \ge 1 \\
 \end{array} \right.
$$

The main result of this section is the following proposition.
\begin{proposition}\label{PoiuyhU}
Assume that $\mathcal{T}$ is the maximal time of  $T\in[0,T_*)$ such that for any $2<p<6+2\gamma$ with $0<\gamma\ll1$,
\begin{align}
{\left\| {du} \right\|_{L_t^\infty L_x^2([0,T] \times {\Bbb H^2})}} + {\left\| {\nabla du} \right\|_{L_t^\infty L_x^2([0,T] \times {\Bbb H^2})}} &\le M_1\label{d7nn885658}\\
{\left\| {\nabla {\partial _t}u} \right\|_{L_t^\infty L_x^2([0,T] \times {\Bbb H^2})}} +
{\left\| {{\partial _t}u} \right\|_{L_t^\infty L_x^2([0,T] \times {\Bbb H^2})}}+{\left\| {{\partial _t}u} \right\|_{L_t^2L^p_x([0,T] \times {\Bbb H^2})}}&\nonumber\\
+{\left\| {D^{\frac{1}{4}}\phi_t(0,t,x)} \right\|_{L_t^2L^p_x([0,T] \times {\Bbb H^2})}}&\le {\varepsilon _1}\label{oootyiys}\\
{\left\| {\omega_{\frac{1}{2}}D^{-\frac{1}{2}}{\partial _t}{\phi _s}} \right\|_{L_s^\infty L_t^2L_x^p}} + {\left\| {\omega_{\frac{1}{2}}{\partial _t}{\phi _s}} \right\|_{L_s^\infty L_t^\infty L_x^2}} &\le {\varepsilon _1}\label{c5rtyf} \\
{\left\| {\omega_{\frac{1}{2}}\nabla {\phi _s}} \right\|_{L_s^\infty L_t^\infty L_x^2}} +{\left\| {{\phi _s}} \right\|_{L_s^\infty L_t^\infty L_x^2}}+ {\left\| {\omega_{\frac{1}{2}}{ D^{\frac{1}{2}}}{\phi _s}} \right\|_{L_s^\infty L_t^2L_x^p}} &\le {\varepsilon _1}.\label{bcde45rtyf}
\end{align}
Then all for $T<\mathcal{T}$ we have
\begin{align}
&{\left\| {\omega_{\frac{1}{2}}{D^{ - \frac{1}{2}}}{\partial _t}{\phi _s}} \right\|_{L_s^\infty L_t^2L_x^p([0,T] \times {\Bbb H^2})}} + {\left\| {\omega_{\frac{1}{2}}{D^{\frac{1}{2}}}{\phi _s}} \right\|_{L_t^2L_x^p([0,T] \times {\Bbb H^2})}} \nonumber\\
&+ {\left\| {\omega_{\frac{1}{2}}{\partial _t}{\phi _s}} \right\|_{L_s^\infty L_t^\infty L_x^2([0,T] \times {\Bbb H^2})}}
+ {\left\| {\omega_{\frac{1}{2}}\nabla {\phi _s}} \right\|_{L_s^\infty L_t^\infty L_x^2([0,T] \times {\Bbb H^2})}} \le \varepsilon _1^2C(M_1). \label{2kijhuoji}
\end{align}
and there holds
\begin{align}
{\left\| {du} \right\|_{L_t^\infty L_x^2([0,T] \times {\Bbb H^2})}} + {\left\| {\nabla du} \right\|_{L_t^\infty L_x^2([0,T] \times {\Bbb H^2})}}&\le CM_0\label{1boot10}\\
{\left\| {{\partial _t}u} \right\|_{L_t^\infty L_x^2([0,T] \times {\Bbb H^2})}} + {\left\| {\nabla {\partial _t}u} \right\|_{L_t^\infty L_x^2([0,T] \times {\Bbb H^2})}} &\lesssim C(M_1){\varepsilon^2 _1}\label{1boot8}\\
{\left\| D^{\frac{1}{4}}\phi_t \right\|_{L_t^2L_x^p([0,T] \times {\Bbb H^2})}} &\lesssim C(M_1){\varepsilon^2_1}.\label{1boot9}
\end{align}
\end{proposition}
The proof of Proposition \ref{PoiuyhU}  will be divided into several lemmas and propositions. (\ref{2kijhuoji}) will be proved in Proposition 7.2.
(\ref{1boot10})-(\ref{1boot9}) will be proved in Lemma 7.10.

Proposition \ref{fcuyder47te} has given us the long time and short time behaviors of $|\nabla^k d \widetilde{u}|$ along the heat flow. The bounds of $|\nabla^k\partial_s\widetilde{u}|$ in Proposition \ref{fcuyder47te} shall be improved by using (\ref{c5rtyf}) and (\ref{bcde45rtyf}).
\begin{lemma}\label{huguokitredcfr}
Assume (\ref{d7nn885658}) to (\ref{bcde45rtyf}) hold, then there exists $\delta>0$ such that it holds uniformly for $(s,t)\in\Bbb R^+\times[0,T]$ that
\begin{align}
&e^{\delta s}\|\phi_{t,s}\|_{L^2_x}+ s^{\frac{1}{2}}e^{\delta s}\|\phi_s\|_{L^{\infty}_x}+s^{\frac{1}{2}}e^{\delta s}\|\phi_t\|_{L^{\infty}_x}\lesssim \varepsilon_1\label{s45t}\\
&\|A\|_{L^{\infty}_x}\lesssim M_0+C(M_1)\varepsilon_1\label{kp32}\\
&\|A_t\|_{L^{\infty}_x}+{\left\| {{A_t}} \right\|_{L_t^\infty L_x^\infty }}\lesssim \varepsilon_1\label{gft56rtklio}\\
&s^{\frac{1}{2}}e^{\delta s}\|\nabla\partial_s\widetilde{u}\|_{L^2_x}+ se^{\delta s}\|\nabla\partial_s\widetilde{u}\|_{L^{\infty}_x}\lesssim \varepsilon_1\label{kp2}\\
&\|\sqrt{h^{ii}}\phi^{qua}_i\|_{L^{\infty}_x}\lesssim C(M_1)\varepsilon_1 e^{-\delta s}\log s   \label{vrts4rioik}\\
&\omega_{\frac{1}{2}}\|\nabla_t\partial_s\widetilde{u}\|_{L^2_x}+ \omega_1\|\nabla_t\partial_s\widetilde{u}\|_{L^{\infty}_x}\lesssim \varepsilon_1\label{kp3}\\
&\omega_{\frac{1}{2}}\|\nabla\partial_t\widetilde{u}\|_{L^2_x}+ \omega_1\|\nabla_t\partial_s\widetilde{u}\|_{L^{\infty}_x}\lesssim \varepsilon_1\label{kp4}\\
&s^{\beta_1}e^{\delta s}\left({\| {\sqrt {{h^{ii}}} {\partial _t}{A_i}(s)}\|_{L_t^\infty L_x^\infty }}+{\| { {{h^{ii}}} {\partial _i}{A^{qua}_i}(s)}\|_{L_t^\infty L_x^\infty }} \right) \lesssim C(M_1)\varepsilon_1 \label{kp5}\\
&1_{s\ge1}e^{\delta s}\left({\| {{{h^{ii}}} {\partial _i}{A^{qua}_i}(s)}\|_{L_t^\infty L_x^\infty }}+{\| {\sqrt {{h^{ii}}} {A^{qua}_i}(s)}\|_{L_t^\infty L_x^\infty }}\right)\lesssim C(M_1)\varepsilon_1\label{kp7}\\
&(s^{\frac{1}{2}}1_{s\le1}+1_{s\ge1})\|\nabla^2d\widetilde{u}\|_{L^{\infty}_{t}L^2_{x}}
+(s1_{s\le1}+1_{s\ge1})\|\nabla^2d\widetilde{u}\|_{L^{\infty}_{t,x}}\lesssim C(M_1)\label{kq11}
\end{align}
with $\beta_1$ being any constant in $(0,1)$.
\end{lemma}
\begin{proof}
Since $\partial_s\widetilde{u}$ satisfies $(\partial_s-\Delta)|\partial_s\widetilde{u}|\le0$, by maximum principle and (\ref{ytsr45}), we deduce
\begin{align*}
\|\partial_s\widetilde{u}\|_{L^2_x}\le e^{-\frac{s}{4}}\|\partial_s\widetilde{u}(0,t,x)\|_{L^2_x}.
\end{align*}
Meanwhile (\ref{tyihy6ys}) and maximum principle show
\begin{align*}
\|\partial_s\widetilde{u}\|_{L^{\infty}_x}\le s^{-\frac{1}{2}}e^{-\frac{s}{4}}\|\partial_s\widetilde{u}(0,t,x)\|_{L^2_x}.
\end{align*}
Similarly the same results hold for $|\partial_t \widetilde{u}|$ since we also have $(\partial_s-\Delta)|\partial_t \widetilde{u}|\le 0$.
Thus (\ref{s45t}) follows by (\ref{bcde45rtyf}), (\ref{oootyiys}). (\ref{kp32}) follows by  (\ref{s45t}) and (\ref{3.2}). And (\ref{gft56rtklio}) follows by Lemma 2.5 and (\ref{s45t}).

For $s\ge 1$, {Proposition \ref{fcuyder47te} } and (\ref{d7nn885658}) imply
\begin{align*}
\|\nabla d\widetilde{u}\|^2_{L^{\infty}_x}+\|d\widetilde{u}\|^2_{L^{\infty}_x}\le  C(M_1).
\end{align*}
Then (\ref{icionm}) and (\ref{s45t}) show
\begin{align*}
&\partial_s|\nabla\partial_s\widetilde{u}|^2-\Delta|\nabla\partial_s\widetilde{u}|^2\\
&\le |\nabla\partial_s\widetilde{u}|^2(C(M_1)+3)+C(M_1)^3e^{-2\delta s}\epsilon^2_1+C(M_1)^4e^{-2\delta}\epsilon^2_1.
\end{align*}
Fix any $s_1\ge 1$, let
$$f(s)=|\nabla\partial_s\widetilde{u}|^2e^{\int^s_{s_1}(C(M_1)+3)ds'}
+\big(C(M_1)^4+C(M_1)^3\big)\delta^{-1}e^{-2\delta s}\epsilon^2_1,
$$
then $f$ satisfies for $s\in(s_1,\infty)$
\begin{align*}
(\partial_s-\Delta)f(s)\le 0.
\end{align*}
Hence Lemma \ref{density} gives
\begin{align*}
f(x,s)\lesssim \int^{s+1}_s\int_{B(x,1)}f(\tau)d{\rm{vol_h}}d\tau\le C(M_1) \|\nabla\partial_s\widetilde{u}\|^2_{L^2_x}+C(M_1)\epsilon^2_1.
\end{align*}
Since $|\nabla\partial_s\widetilde{u}|\le |\nabla\phi_s|+\sqrt{h^{ii}}|A_i||\phi_s|$, (\ref{kp2}) follows by (\ref{c5rtyf}) and (\ref{gft56rtklio}).

By (\ref{Hwer45thiys}) and (\ref{y45aswe345t}), the same arguments yield (\ref{kp3}) and (\ref{kp4}) respectively.
(\ref{vrts4rioik}) follows by $|\sqrt{h^{ii}}\partial_i\phi_s|\le |\sqrt{h^{ii}}A_i\phi_s|+ |\nabla\partial_s\widetilde{u}|$ and (\ref{gft56rtklio}).

Finally, by (\ref{awo}),
\begin{align}
|{\sqrt {{h^{ii}}} {\partial _t}{A_i}(s)}|&\le \int^{\infty}_s \sqrt{h^{ii}}|\nabla_t\partial_s\widetilde{u}\wedge\partial_i\widetilde{u}|ds'+\int^{\infty}_s \sqrt{h^{ii}}|\nabla_t\partial_i\widetilde{u}\wedge\partial_s\widetilde{u}|ds'\nonumber\\
&+\sqrt{h^{ii}}|A_tA_i|\label{sw345t}\\
|{{h^{ii}} {\partial _i}{A^{qua}_i}(s)}|&\le \int^{\infty}_s h^{ii}|\nabla_i\partial_s\widetilde{u}\wedge\partial_i\widetilde{u}|ds'+\int^{\infty}_s h^{ii}|\nabla_i\partial_i\widetilde{u}\wedge\partial_s\widetilde{u}|ds'\nonumber\\
&+|h^{ii}A^{qua}_i||A_i|.\label{swsde4345t}
\end{align}
then (\ref{kp7}) follows by (\ref{sw345t}), (\ref{swsde4345t}), (\ref{s45t})-(\ref{kp4}).

Since $|d\widetilde{u}|$ satisfies $|d\widetilde{u}|\le e^{Cs}e^{s\Delta}|du|$, then by $\|d\widetilde{u}\|_{L^{2}_x}+\|\nabla d\widetilde{u}\|_{L^{2}_x}\lesssim C(M_1)$ (see Proposition \ref{fcuyder47te}), Lemma 8.4 and Sobolev embedding, we obtain that for $s\in[0,1]$
\begin{align*}
\|d\widetilde{u}\|_{L^{\infty}_x}\lesssim s^{-\beta_1}C(M_1)
\end{align*}
for any $\beta_1\in(0,1)$. Thus (\ref{kp5}) follows by (\ref{sw345t}), (\ref{swsde4345t}), (\ref{s45t})-(\ref{kp4}).

By integration by parts, one has
\begin{align*}
\|\nabla^2d\widetilde{u}\|^2_{L^2_{x}}\le \|\nabla\tau(\widetilde{u})\|^2_{L^2_{x}}+\|\nabla d\widetilde{u}\|^2_{L^2_{x}}\|d\widetilde{u}\|^2_{L^{\infty}_{x}}+\| d\widetilde{u}\|^6_{L^6_{x}}.
\end{align*}
Thus by $\tau(\widetilde{u})=\partial_s\widetilde{u}$, (\ref{kp2}) and Proposition \ref{fcuyder47te}, the $L^2_x$ part in (\ref{kq11}) follows. The short time $L^{\infty}_x$ part in (\ref{kq11}) follows by applying maximum principle and smoothing effect of $e^{t\Delta}$ to (\ref{tsreinhy45}) with the help of the previously obtained bounds for $\|\nabla^2d\widetilde{u}\|_{L^{\infty}_{t}L^2_{x}}$. The large time $L^{\infty}_x$ part in (\ref{kq11}) follows by applying Lemma \ref{density} and the bound $\|\nabla^2d\widetilde{u}\|_{L^{\infty}_{t}L^2_{x}}$ to (\ref{tsreinhy45}), see the proof of (\ref{kp2}) above.
\end{proof}

Direct calculations give
\begin{lemma}\label{poke}
Assume (\ref{d7nn885658}) to (\ref{bcde45rtyf}) hold. Then we have for the coordinates in {(2.10)}
\begin{align}
\| {\sqrt {{h^{pp}}} {\partial _p}\left( {{h^{ii}}{\partial _i}A_i^{qua} - {h^{ii}}\Gamma _{ii}^kA_k^{qua}} \right)} \|_{L^{\infty}_x}
&\lesssim s^{-\frac{1}{2}}e^{-\delta s}C(M_1)\varepsilon_1+\int^{\infty}_s|d\widetilde{u}||\nabla^2\partial_s\widetilde{u}|ds'.\label{2poke}\\
\sqrt {h^{pp}h^{ii}} |{\partial _p}\phi^{qua}_i|&\lesssim C(M_1)\int^{\infty}_s|\nabla^2\phi_s|+|\nabla\phi_s|ds'\label{1poke}
\end{align}
\end{lemma}

We turn to improve the bounds of high order differential fields.
\begin{lemma}\label{kq12}
Assume that (\ref{d7nn885658}) to (\ref{bcde45rtyf}) hold, then
\begin{align}
\omega_{1}(s)\|\nabla^2\phi_s\|_{L^{\infty}_{t}L^2_x}&\lesssim C(M_1)\varepsilon_1\label{kq12a}\\
\omega_{1}(s)\|\nabla^2\partial_s \widetilde{u}\|_{L^{\infty}_{t}L^2_x}&\lesssim C(M_1)\varepsilon_1\label{kq12b}\\
\omega_{\frac{3}{2}}(s)\|\nabla^2\partial_s\widetilde{u}\|_{L^{\infty}_{t,x}}&\lesssim C(M_1)\varepsilon_1\label{kq12c}\\
\omega_{\frac{3}{2}}(s)\|\nabla^2\phi_s\|_{L^{\infty}_x}&\lesssim C(M_1)\varepsilon_1\label{kq12e}\\
\|\sqrt{h^{ii}h^{pp}}\nabla_p\phi^{qua}_i\|_{L^{\infty}_x}&\lesssim C(M_1)\varepsilon_1\min(s^{-\frac{1}{2}},s^{-L+1}).\label{kq12d}
\end{align}
\end{lemma}
\begin{proof}
Integration by parts gives
\begin{align*}
\|\nabla^2\phi_s\|_{L^2_x}\lesssim \|\Delta\phi_s\|_{L^2_x}+\|\nabla\phi_s\|^2_{L^2_x}.
\end{align*}
For (\ref{kq12a}), due to (\ref{s45t})-(\ref{kp2}) and $|\nabla\phi_s|\le |\nabla\partial_s\widetilde{u}|+\sqrt{h^{ii}}|A_i\phi_s|$,  it suffices to prove
\begin{align*}
\omega_{1}(s)\|\Delta\phi_s\|_{L^{\infty}_{t}L^2_x}\le C(M_1)\varepsilon_1.
\end{align*}
By Duhamel principle one deduces from the heat equation of $\phi_s$ that
\begin{align}\label{povmb}
\phi_s=e^{-\frac{s}{2}H}\phi_s(\frac{s}{2})+\int^{s}_{\frac{s}{2}}e^{-H(s-\tau)}G(\tau)d\tau,
\end{align}
where $G$ denotes the inhomogeneous part, i.e.,
\begin{align*}
G&=2h^{ii}A^{qua}_i\partial_i\phi_s+h^{ii}(\partial_iA^{qua}_i)\phi_s-h^{ii}\Gamma^k_{ii}A^{qua}_k\phi_s+h^{ii}A^{qua}_iA^{qua}_i\phi_s\\
&+h^{ii}\mathcal{A}_iA^{qua}_i\phi_s+h^{ii}A^{qua}_i\mathcal{A}_i\phi_s+h^{ii}(\phi_s\wedge\phi^{qua}_i)\phi^{qua}_i\\
&+h^{ii}(\phi_s\wedge\widehat{\phi}_i)\phi^{qua}_i+h^{ii}(\phi_s\wedge\phi^{qua}_i)\widehat{\phi}_i.
\end{align*}
Applying $H$ to (\ref{povmb}) we get by Proposition \ref{pkg2}
\begin{align}\label{pro1}
\|H\phi_s\|_{L^2_x}\le s^{-1}e^{-\frac{s}{8}}\|\phi_s(\frac{s}{2})\|_{L^2_x}+\int^{s}_{\frac{s}{2}}(s-\tau)^{-\frac{1}{2}}e^{-\frac{s-\tau}{4}}\|H^{\frac{1}{2}}G(\tau)\|_{L^2_x}d\tau,
\end{align}
Then the homogeneous term in (\ref{pro1}) is acceptable by Lemma \ref{huguokitredcfr}. Again due to Proposition \ref{pkg2}, we deduce
\begin{align}\label{pro11}
\|H^{\frac{1}{2}}G(\tau)\|_{L^2_x}\lesssim \|\nabla G(\tau)\|_{L^2_x}.
\end{align}
Denote the one order derivative term in $G$ by $G_1\triangleq2h^{ii}A^{qua}_i\partial_i\phi_s$. And $G_2\triangleq h^{ii}(\partial_iA^{qua}_i)\phi_s-h^{ii}\Gamma^k_{ii}A^{qua}_k\phi_s$. The remainder terms in $G$ are denoted by $G_3\triangleq G-G_1-G_2$.
By Lemma \ref{poke} and (\ref{kp7}),
\begin{align}
\|\nabla G_1(\tau)\|_{L^2_x}&\lesssim \|\sqrt{h^{ii}}A^{qua}_i\|_{L^{\infty}_x}\|\nabla^2\phi_s\|_{L^2_x}+\big(\int^{\infty}_{\tau}\|\nabla \phi_s\|_{L^{\infty}_x}\|\sqrt{h^{ii}}\phi_i\|_{L^{\infty}_x}ds'\big)\|\nabla\phi_s\|_{L^2_x}\nonumber\\
+& \big(\int^{\infty}_{\tau}\|\phi_s\|_{L^{\infty}_x}\|\sqrt{h^{ii}h^{pp}}\partial_p\phi_i\|_{L^{\infty}_x}ds'\big)\|\nabla\phi_s\|_{L^2_x}
\label{pro3}\\
\|\nabla G_2(\tau)\|_{L^2_x}&\lesssim C(M_1)\varepsilon_1\tau^{-\frac{1}{2}}e^{-\delta \tau}\|\nabla \phi_s\|_{L^2_x}
+\big(\int^{\infty}_{\tau}\|\nabla^2 {\partial _s}\widetilde{u}\|_{L^2_x}\|d \widetilde{u}\|_{L^{\infty}_x}ds'\big)\|\phi_s\|_{L^{\infty}_x}\label{pro2}
\end{align}
We get from Lemma \ref{poke} and Lemma \ref{huguokitredcfr} that the $G_3$ term is bounded by
\begin{align}
\|\nabla G_3(\tau)\|_{L^2_x}&\lesssim \tau^{-\frac{1}{2}}e^{-\delta\tau}M^2_1\varepsilon^2_1\|\phi_s\|_{L^2_x}+e^{-\delta\tau}M^2_1\varepsilon^2_1\|\nabla\phi_s\|_{L^2_x}
+\varepsilon^2_1e^{-\delta \tau}(\log \tau)^2\|\nabla\phi_s\|_{L^2_x}\nonumber\\
&+\varepsilon^2_1\tau^{-\frac{1}{2}}e^{-\delta \tau}\log \tau\int^{\infty}_{\tau}\|\nabla^2\phi_s\|_{L^2_x}ds'\label{pro4}
\end{align}
Using the inequality
\begin{align}
\int^{\infty}_{\tau}\|\nabla^2\phi_s\|_{L^2_x}ds'&\lesssim \|\omega_{1}(s)\nabla^2\phi_s\|_{L^{\infty}_sL^2_x}\min(\log \tau,\tau^{-L+1})\label{p7csw34}\\
|\nabla^2\partial_s\widetilde{u}|&\lesssim |\nabla^2\phi_s|+h^{ii}|A_i||\partial_i\phi_s|+\sqrt{h^{pp}h^{ii}}|\partial_p A_i||\phi_s|\label{xskq12e}
\end{align}
we conclude by Lemma \ref{huguokitredcfr} and (\ref{pro1})-(\ref{pro4}) that
\begin{align*}
\|\omega_{1}(s)H\phi_s\|_{L^{\infty}_sL^2_x}\lesssim \varepsilon_1+C(M_1)\varepsilon_1\|\omega_{1}(s)H\phi_s\|_{L^{\infty}_sL^2_x},
\end{align*}
thus finishing the proof of (\ref{kq12a}).

 (\ref{kq12b}) follows from (\ref{kq12a}), Lemma \ref{huguokitredcfr} and (\ref{xskq12e}).
Then the large time part of (\ref{kq12c}) follows from (\ref{kq12b}) and Lemma \ref{huguokitredcfr} by applying Lemma \ref{density} to (\ref{tsrinhy45}). The short time part of (\ref{kq12c}) follows by (\ref{kq12b}) and applying smoothing effect of $e^{t\Delta}$ to (\ref{tsrinhy45}). And due to (\ref{xskq12e}), (\ref{kq12e}) follows by (\ref{kq12c}) and Lemma \ref{huguokitredcfr}.

Finally, (\ref{kq12d}) follows by (\ref{xskq12e}), (\ref{1poke}) and Lemma \ref{huguokitredcfr}.
\end{proof}

The following three lemmas give the bounds for $\phi_t$.
\begin{lemma}
Assume that (\ref{d7nn885658}) to (\ref{bcde45rtyf}) hold, then for $q\in(2,6+2\gamma]$
\begin{align}
{\left\| {{\phi _t}(s)} \right\|_{L_s^\infty L_t^2L_x^q}} &\lesssim C(M_1){\varepsilon _1}\label{time2}\\
{\left\| {D^{\frac{1}{4}}{\phi _t}(s)} \right\|_{L_s^\infty L_t^2L_x^q}} &\lesssim C(M_1){\varepsilon _1}\label{time30}\\
{\left\| s^{L}{D^{\frac{1}{4}}{\phi _t}(s)} \right\|_{L_s[1,\infty) L_t^2L_x^q}} &\lesssim C(M_1){\varepsilon _1}\label{time31}
\end{align}
\end{lemma}
\begin{proof}
(\ref{time2})  follows {by the same arguments of [Lemma 7.2\cite{EFE376}]}. It suffices to prove (\ref{time30}) and (\ref{time31}).
The differential filed $\phi_t$ satisfies
\begin{align*}
(\partial_s-\Delta)\phi_t=2h^{ii}A_i\partial_i\phi_t+h^{ii}A_iA_i\phi_t+h^{ii}\partial_iA_i\phi_t-h^{ii}\Gamma^k_{ii}A_k\phi_t+h^{ii}(\phi_t\wedge\phi_i)\phi_i.
\end{align*}
Separating the $H\phi_t$ part away from the nonlinearity terms yields
\begin{align}
(\partial_s+H)\phi_t&=2h^{ii}A^{qua}_i\partial_i\phi_t+h^{ii}A^{qua}_i\mathcal{A}_i\phi_t+h^{ii}\mathcal{{A}}_iA^{qua}_i\phi_t++h^{ii}A^{qua}_iA^{qua}_i\phi_t\nonumber\\
&+h^{ii}(\partial_iA^{qua}_i-\Gamma^k_{ii}A^{qua}_k)\phi_t+h^{ii}(\phi_t\wedge\widehat{\phi}_i)\phi^{qua}_i+h^{ii}(\phi_t\wedge\phi^{qua}_i)\widehat{\phi}_i\nonumber\\
&+h^{ii}(\phi_t\wedge\phi^{qua}_i)\phi^{qua}_i.\label{u7iy}
\end{align}
Denote the right hand side of (\ref{u7iy}) as $\mathcal{G}$. And denote the one order derivative term of $\phi_t$ by $\mathcal{G}_1$, i.e., $\mathcal{G}_1=2h^{ii}A^{qua}_i\partial_i\phi_t$. The other zero order terms are denoted by $\mathcal{G}_2$. Applying $H^{\frac{1}{8}}$ to (\ref{u7iy}), by Proposition \ref{pkg1} we have
\begin{align}
\|H^{\frac{1}{8}}\phi_t\|_{L^2_tL^q_x}&\lesssim \|H^{\frac{1}{8}}\phi_t(0,t,x)\|_{L^2_tL^q_x}+\int^{s}_0(s-\tau)^{-\frac{1}{8}-\epsilon}e^{-\delta(s-\tau)}\|\mathcal{G}_2(\tau)\|_{L^2_tL^q_x}d\tau
\nonumber\\
&+\int^{s}_0\|e^{-(s-\tau)H}H^{\frac{1}{8}}\mathcal{G}_1(\tau)\|_{L^2_tL^q_x}d\tau\label{poikl}
\end{align}
Lemma \ref{huguokitredcfr} and (\ref{oootyiys}) show
\begin{align}\label{p8v}
\|\mathcal{G}_2(\tau)\|_{L^2_tL^q_x}\lesssim  C(M_1)\varepsilon_1^2\min(\tau^{-\frac{1}{2}},\tau^{-L+1})
\end{align}
By Proposition \ref{pkg1}, the $\mathcal{G}_1$ term in (\ref{poikl}) is bounded by
\begin{align}
&\|e^{-(s-\tau)H}H^{\frac{1}{8}}\mathcal{G}_1(\tau)\|_{L^2_tL^q_x}\nonumber\\
&\lesssim\|e^{-(s-\tau)H}H^{\frac{1}{8}}h^{ii}\partial_i(A^{qua}_i\phi_t)\|_{L^2_tL^q_x}
+\|e^{-(s-\tau)H}H^{\frac{1}{8}}h^{ii}(\partial_iA^{qua}_i)\phi_t\|_{L^2_tL^q_x}.\label{p8mnbv}
\end{align}
The second term in (\ref{p8mnbv}) has appeared in $\mathcal{G}_2$.
By (\ref{p3ib}), (\ref{p4ib}) and Proposition \ref{pkg1}, one has for any $\epsilon\in(0,1)$
\begin{align}
&\|e^{-(s-\tau)H}H^{\frac{1}{8}}h^{ii}\partial_i(A^{qua}_i\phi_t)\|_{L^q_x}\nonumber\\
&\lesssim  \|e^{-(s-\tau)H}H^{\frac{1}{8}}H^{\frac{1}{2}}H^{-\frac{1}{2}}(-\Delta)^{\frac{1}{2}}(-\Delta)^{-\frac{1}{2}}h^{ii}\partial_i(A^{qua}_i\phi_t)\|_{L^q_x}\nonumber\\
&\lesssim (s-\tau)^{-\frac{5}{8}-\epsilon}e^{-\delta(s-\tau)}\|(-\Delta)^{-\frac{1}{2}}h^{ii}\partial_i(A^{qua}_i\phi_t)\|_{L^q_x}\nonumber\\
&\lesssim  (s-\tau)^{-\frac{5}{8}-\epsilon}e^{-\delta(s-\tau)}\|\sqrt{h^{ii}}(A^{qua}_i\phi_t)\|_{L^q_x}\label{p9cve4}
\end{align}
where we use the boundedness of Riesz transform in the last line.
Hence Lemma \ref{huguokitredcfr}, (\ref{d7nn885658})-(\ref{bcde45rtyf}), (\ref{poikl})-(\ref{p9cve4}) give
\begin{align}\label{hjyp67}
\|H^{\frac{1}{8}}\phi_t\|_{L^2_tL^q_x}\lesssim C(M_1)\varepsilon_1.
\end{align}
Then (\ref{time30}) follows from (\ref{hjyp67}) and (\ref{c5rtyf}). Similar arguments yield (\ref{time31}).
\end{proof}

\begin{lemma}
Assume that (\ref{d7nn885658}) to (\ref{bcde45rtyf}) hold, then for any $\epsilon\in(0,1)$,
$\phi_t$ satisfies for $q\in(2,6+2\gamma]$ with $0<\gamma\ll1$
\begin{align}
{\left\| {\omega_{\frac{3}{8}+\epsilon}(s)\nabla{\phi _t}(s)} \right\|_{L_s^\infty L_t^2L_x^q}} &\lesssim {\varepsilon _1}\label{ku3}
\end{align}
\end{lemma}
\begin{proof}
Applying $H^{\frac{1}{2}}$ to (\ref{u7iy}), Duhamel principle gives
\begin{align}\label{pwer}
\|H^{\frac{1}{2}}\phi_t\|_{L^q_x}\lesssim \|H^{\frac{1}{2}}e^{-\frac{s}{2}H}\phi_t(\frac{s}{2})\|_{L^q_x}+\int^{s}_{\frac{s}{2}}\sum^2_{j=1}\|H^{\frac{1}{2}}e^{-(s-\tau)H}\mathcal{G}_j(\tau)\|_{L^q_x}d\tau
\end{align}
Proposition \ref{pkg1} gives the bound for the first term in (\ref{pwer})
\begin{align*}
\omega_{\frac{3}{8}+\epsilon}(s)\|H^{\frac{1}{2}}e^{-\frac{s}{2}H}\phi_t(\frac{s}{2})\|_{L^q_x}\lesssim \|H^{\frac{1}{8}}\phi_t(\frac{s}{2})\|_{L^q_x},
\end{align*}
which combined with (\ref{bbd75658}), (\ref{p3ib}), (\ref{p4ib}) implies
\begin{align}\label{e3r}
\omega_{\frac{3}{8}+\epsilon}(s)\|H^{\frac{1}{2}}e^{-\frac{s}{2}H}\phi_t(\frac{s}{2})\|_{L^2_tL^q_x}\lesssim C(M_1)\varepsilon_1.
\end{align}
The second term in (\ref{pwer}) is bounded by
\begin{align}
&\int^{s}_{\frac{s}{2}}\sum^2_{j=1}\|H^{\frac{1}{2}}e^{-(s-\tau)H}\mathcal{G}_j(\tau)\|_{L^2_tL^q_x}d\tau\nonumber\\
&\lesssim C(M_1)\varepsilon_1\int^{s}_{\frac{s}{2}}e^{-\delta(s-\tau)}(s-\tau)^{-\frac{1}{2}-\epsilon}
\|\nabla\phi_t(\tau)\|_{L^2_tL^q_x}d\tau\nonumber\\
&+ C(M_1)\varepsilon_1\int^{s}_{\frac{s}{2}}e^{-\delta(s-\tau)}(s-\tau)^{-\frac{1}{2}-\epsilon}
(\tau^{-\frac{1}{2}}+1)\tau^{-L}d\tau.\label{e34r}
\end{align}
Thus (\ref{e34r}), (\ref{e3r}) and (\ref{pwer}) yield
\begin{align}\label{e4r}
\|\omega_{\frac{3}{8}+\epsilon}(s)H^{\frac{1}{2}}\phi_t\|_{L^{\infty}_sL^2_tL^q_x}\lesssim C(M_1)\varepsilon_1\|\omega_{\frac{3}{8}+\epsilon}(s)H^{\frac{1}{2}}\phi_t\|_{L^{\infty}_sL^2_tL^q_x}+C(M_1)\varepsilon_1.
\end{align}
Hence (\ref{ku3}) follows by (\ref{e4r}) and (\ref{p3ib}), (\ref{p4ib}).
\end{proof}

\begin{lemma}
{Assume (\ref{d7nn885658}) to (\ref{bcde45rtyf}) hold, then}
$\phi_t$ satisfies for $q\in(2,6+2\gamma]$ with $0<\gamma\ll1$ and any $\epsilon\in(0,1)$
\begin{align}
{\left\| {\omega_{\frac{7}{8}+\epsilon}(s)\nabla^2{\phi _t}(s)} \right\|_{L_s^\infty L_t^2L_x^q}} \lesssim {\varepsilon _1}.\label{ku4}
\end{align}
\end{lemma}
\begin{proof}
Applying $H$ to (\ref{u7iy}), Duhamel principle gives
\begin{align}\label{pwer1}
\|H\phi_t\|_{L^q_x}\lesssim \|He^{-\frac{s}{2}H}\phi_t(\frac{s}{2})\|_{L^q_x}+\int^{s}_{\frac{s}{2}}\sum^2_{j=1}\|He^{-(s-\tau)H}\mathcal{G}_j(\tau)\|_{L^q_x}d\tau.
\end{align}
(\ref{p3ib}), (\ref{p4ib}) and Proposition \ref{pkg1} yield
\begin{align}
&\|He^{-(s-\tau)H}\mathcal{G}_j(\tau)\|_{L^q_x}\lesssim \|H^{\frac{1}{2}}e^{-(s-\tau)H}H^{\frac{1}{2}}\mathcal{G}_j(\tau)\|_{L^q_x}\nonumber\\
&\lesssim (s-\tau)^{-\frac{1}{2}-\epsilon}e^{-\delta(s-\tau)}\|H^{\frac{1}{2}}\mathcal{G}_j(\tau)\|_{L^q_x}\nonumber\\
&\lesssim (s-\tau)^{-\frac{1}{2}-\epsilon}e^{-\delta(s-\tau)}(\|\nabla\mathcal{G}_j(\tau)\|_{L^q_x}+\|\mathcal{G}_j(\tau)\|_{L^q_x}).\label{wer1}
\end{align}
First we deal with the $\mathcal{G}_1$ term. By the explicit formula of $\Gamma^{i}_{kj}$ and $h^{ij}$ in (\ref{rex3456gvh}), we have
\begin{align*}
|\nabla(h^{ii}A^{qua}_i\partial_i\phi_t)|&\lesssim \sum^2_{p=1}\sqrt{h^{pp}}|\partial_p(h^{ii}A^{qua}_i\partial_i\phi_t)|\\
&\lesssim \sum^2_{p=1}\sqrt{h^{pp}}|h^{ii}A^{qua}_i\partial_p\partial_i\phi_t|+\sqrt{h^{pp}}|\partial_p(h^{ii}A^{qua}_i)\partial_i\phi_t|\\
&\lesssim \sqrt{h^{ii}}A^{qua}_i|{\nabla^2}(\phi_t)|+\sqrt{h^{ii}}\sqrt{h^{pp}}|\partial_pA^{qua}_i||\nabla\phi_t|+\sqrt{h^{ii}}|A^{qua}_i||\nabla\phi_t|
\end{align*}
Thus Lemma \ref{huguokitredcfr} shows that the $\mathcal{G}_1$ term in (\ref{wer1}) is bounded by
\begin{align}
&\int^{s}_{\frac{s}{2}}\|He^{-(s-\tau)H}\mathcal{G}_1(\tau)\|_{L^q_x}d\tau\nonumber\\
&\lesssim\int^{s}_{\frac{s}{2}} (s-\tau)^{-\frac{1}{2}-\epsilon}e^{-\delta(s-\tau)}\big(\|\sqrt{h^{ii}}A^{qua}_i{\nabla^2}(\phi_t)\|_{L^q_x}\big)d\tau\nonumber\\
&+\int^{s}_{\frac{s}{2}} \min(\tau^{-\beta_1},\tau^{L})\|\nabla\phi_t\|_{L^q_x}d\tau,\label{ret2}
\end{align}
where $\beta_1$ is any sufficiently small constant in $(0,1)$.
Thus by Lemma \ref{hessian} and Sobolev embedding, for $q<r$, $1+\frac{1}{r}-\frac{1}{q}> \frac{3}{4}$ and $\frac{1}{m}+\frac{1}{r}=\frac{1}{q}$ we have
\begin{align}
\|\sqrt{h^{ii}}A^{qua}_i{\nabla^2}(\phi_t)\|_{L^q_x}&\lesssim \|{\nabla^2}(\phi_t)\|_{L^r_x}\|\sqrt{h^{ii}}A^{qua}_i\|_{L^m_x}\nonumber\\
&\lesssim C(M_1)\varepsilon_1\|{\Delta}\phi_t\|_{L^q_x}.\label{ret}
\end{align}
Then by the trivial inequality
\begin{align}\label{ku5}
\|\Delta f\|_{L^q_x}\lesssim \|H f\|_{L^q_x}+ \|\nabla f\|_{L^q_x}+\|f\|_{L^q_x}
\end{align}
and Lemma \ref{huguokitredcfr}, (\ref{ret}) implies
\begin{align}
\|\sqrt{h^{ii}}A^{qua}_i{\nabla^2}(\phi_t)\|_{L^q_x}&\lesssim  C(M_1)\varepsilon_1\big(\|H\phi_t\|_{L^q_x}+\|\nabla\phi_t\|_{L^q_x}+\|\phi_t\|_{L^q_x}\big).\label{ret1}
\end{align}
Therefore, (\ref{ku3}), (\ref{ret1}) and (\ref{ret2}) give the acceptable bound for $\mathcal{G}_1$,
\begin{align}
&\int^{s}_{\frac{s}{2}}\|He^{-(s-\tau)H}\mathcal{G}_1(\tau)\|_{L^q_x}d\tau\nonumber\\
&\lesssim C(M_1)\varepsilon_1\int^{s}_{\frac{s}{2}}(s-\tau)^{-\frac{1}{2}
-\epsilon}e^{-\delta(s-\tau)}\big(\|H\phi_t\|_{L^q_x}+\|\phi_t\|_{L^q_x}+\|\nabla\phi_t\|_{L^q_x}\big)d\tau+C(M_1)\varepsilon_1.\label{ret3}
\end{align}
For the $\mathcal{G}_2$ term, we first consider the tougher term $\mathcal{G}_{21}\triangleq h^{ii}\partial_iA^{qua}_i\phi_t-h^{ii}\Gamma^{k}_{ii}A^{qua}_k\phi_t$.
Lemma \ref{poke}, Lemma \ref{huguokitredcfr} and the explicit formula of $\Gamma^{i}_{kj}$ and $h^{ij}$ in (\ref{rex3456gvh}) give for any $\beta_1\in (0,1)$
\begin{align}\label{ret5}
|\nabla\mathcal{G}_{21}|\lesssim C(M_1)\varepsilon_1\big(\min(\tau^{-\frac{1}{2}-\beta_1},\tau^{-L})|\phi_t|+\min(\log \tau,\tau^{-L})|\nabla\phi_t|\big).
\end{align}
We denote the remainder terms in $\mathcal{G}_2$ by $\mathcal{G}_{22}$, then Lemma \ref{huguokitredcfr} and Lemma \ref{kq12} show
\begin{align}\label{ret6}
|\nabla\mathcal{G}_{22}(\tau)|\lesssim C(M_1)\varepsilon_1\big(\min(\tau^{-\frac{1}{2}},\tau^{-L})|\phi_t|+\min(\log \tau,\tau^{-L})|\nabla\phi_t|\big).
\end{align}
Hence (\ref{ret5}), (\ref{ret6}) give
\begin{align}
\omega_{\frac{7}{8}+\epsilon}(s)\int^{s}_{\frac{s}{2}}\|He^{-(s-\tau)H}\mathcal{G}_2(\tau)\|_{L^2_tL^q_x}d\tau\lesssim C(M_1)\varepsilon_1\|\omega_{\frac{3}{8}+\epsilon}(s)\nabla\phi_t\|_{L^{\infty}_sL^2_tL^q_x}+\|\omega_{0}(s)\phi_t\|_{L^{\infty}_sL^2_tL^q_x}.\label{ret4}
\end{align}
Combining (\ref{ret3}), (\ref{ret4}), (\ref{wer1})  with (\ref{ku3}), we infer from (\ref{time30}) that
\begin{align*}
\|\omega_{\frac{7}{8}+\epsilon}(s)H \phi_t\|_{L^{\infty}_sL^2_tL^q_x}\lesssim C(M_1)\varepsilon_1\|\omega_{\frac{7}{8}+\epsilon}(s)H \phi_t\|_{L^{\infty}_sL^2_tL^q_x}+C(M_1)\varepsilon_1.
\end{align*}
Therefore, we get (\ref{ku4}) from (\ref{ku3}), (\ref{time30}) and (\ref{p3ib}), (\ref{p4ib}).
\end{proof}

We consider the high order derivatives of $\phi_s$.
\begin{lemma}\label{i97clv}
Assume that (\ref{d7nn885658}) to (\ref{bcde45rtyf}) hold, then for $p\in(2,6)$ and any $\epsilon\in(0,1)$,
\begin{align}\label{huojikn}
{\left\| {\omega_{\frac{3}{4}+\epsilon}(s){{\left\| {{\partial _t}{\phi _s}} \right\|}_{L_t^2L_x^p}}} \right\|_{L_s^\infty}} + {\left\| {\omega_{\frac{3}{4}+\epsilon}(s){{\left\| {\nabla {\phi _s}} \right\|}_{L_t^2L_x^p}}} \right\|_{L_s^\infty}} \lesssim {C(M_1)\varepsilon _1}.
\end{align}
Generally we have for $\theta\in[0,\frac{1}{2})$, $\theta_1\in[0,\frac{3}{4}]$
\begin{align}
&{\left\| {\omega_{\frac{1}{2}+\theta_1+\epsilon}(s){{\left( { - \Delta } \right)}^{\theta_1} }{D^{-\frac{1}{2}}}{\partial _t}{\phi _s}} \right\|_{L_s^\infty L_t^2L_x^p}} + {\left\| {\omega_{\frac{1}{2}+\theta+\epsilon}(s){{\left( { - \Delta } \right)}^\theta }{{D}^{\frac{1}{2}}}{\phi _s}} \right\|_{L_s^\infty L_t^2L_x^p}}\nonumber \\
&\le {C(M_1)\varepsilon _1}.\label{uojiknmpx3}
\end{align}
\end{lemma}
\begin{proof}
{\bf Step 1.} We prove the desired estimates for $\phi_s$ in Step 1. First by (\ref{p3ib}), (\ref{p4ib}) and (\ref{bcde45rtyf}), we note that (\ref{uojiknmpx3}) follows by
\begin{align}\label{koim6}
{\left\| {\omega_{\frac{1}{2}+\theta+\epsilon}{H}^{\theta}{{H}^{\frac{1}{4}}}{\phi _s}} \right\|_{L_s^\infty L_t^2L_x^p}} \lesssim C(M_1){\varepsilon _1}.
\end{align}
{\bf Step 1.1} It will be useful if we first obtain the following estimate
\begin{align}\label{koim9}
\left\| \omega_{\frac{3}{4}+\epsilon}{H^{\frac{1}{2}}}{\phi _s} \right\|_{L_s^\infty L_t^2L_x^p} \lesssim C(M_1){\varepsilon _1}.
\end{align}
Applying $H^{\frac{1}{2}}$ to (\ref{povmb}) by Proposition \ref{pkg1} we obtain for some $\delta>0$ and any $\epsilon\in(0,1)$
\begin{align}\label{kj5}
\|H^{\frac{1}{2}}\phi_s\|_{L^p_x}\lesssim s^{-\frac{1}{4}-\epsilon}e^{-s\delta}\|H^{\frac{1}{4}}\phi_s\|_{L^p_x}+\int^{s}_{\frac{s}{2}}\|e^{-H(t-\tau)}HG(\tau)\phi_s\|_{L^p_x}d\tau.
\end{align}
Split the $G$ into the one order term $G_1\triangleq2h^{ii}A^{qua}_i\partial_i\phi_s$ and the zero order terms $G_2\triangleq G-G_1$. Then for $G_2$, Proposition \ref{pkg1} gives
\begin{align*}
\int^{s}_{\frac{s}{2}}\|e^{-H(t-\tau)}H^{\frac{1}{2}}G(\tau)\phi_s\|_{L^p_x}d\tau\lesssim \int^{s}_{\frac{s}{2}}e^{-\delta(s-\tau)}(s-\tau)^{-\frac{1}{2}-\epsilon}\|G_2(\tau)\|_{L^p_x}d\tau.
\end{align*}
Thus by (\ref{kp5}), Proposition \ref{pkg1} and (\ref{p1ib}), we have
\begin{align}
\omega_{\frac{3}{4}+\epsilon}(s)\int^{s}_{\frac{s}{2}}\|e^{-H(t-\tau)}H^{\frac{1}{2}}G_2(\tau)\|_{L^2_tL^p_x}d\tau\lesssim C(M_1)\varepsilon_1\|\omega_{\frac{1}{2}}\phi_s\|_{L^{\infty}_sL^2_tL^p_x}.\label{kj1}
\end{align}
For the one order term $G_1$, Proposition \ref{pkg1}, Lemma \ref{huguokitredcfr} yield
\begin{align}\label{kj4}
\omega_{\frac{3}{4}+\epsilon}(s)\int^{s}_{\frac{s}{2}}\|e^{-H(t-\tau)}H^{\frac{1}{2}}G_1(\tau)\|_{L^2_tL^p_x}d\tau\lesssim C(M_1)\varepsilon_1\|\omega_{\frac{3}{4}+\epsilon}\nabla\phi_s\|_{L^{\infty}_sL^2_tL^p_x}.
\end{align}
Therefore, (\ref{kj5}) to (\ref{kj4}) give
\begin{align}\label{kj6}
\|\omega_{\frac{3}{4}+\epsilon}H^{\frac{1}{2}}\phi_s\|_{L^{\infty}_sL^2_tL^p_x}\lesssim C(M_1)\varepsilon_1 \|\omega_{\frac{3}{4}+\epsilon}\nabla\phi_s\|_{L^{\infty}_sL^2_tL^p_x}+C(M_1)\varepsilon_1.
\end{align}
Thus by (\ref{pknv}) and (\ref{bcde45rtyf}), (\ref{kj6}) we get (\ref{koim9}) and
\begin{align}\label{kj67}
\|\omega_{\frac{3}{4}+\epsilon}\nabla\phi_s\|_{L^{\infty}_sL^2_tL^p_x}\lesssim C(M_1)\varepsilon_1.
\end{align}
{\bf Step 1.2.} In this step, we prove (\ref{koim6}). Applying $H$ to  (\ref{povmb}), by Proposition \ref{pkg1} we obtain
\begin{align}\label{lk11}
\|H\phi_s\|_{L^p_x}&\lesssim s^{-\frac{1}{2}-\epsilon}e^{-s\delta}\|H^{\frac{1}{2}}
\phi_s(\frac{s}{2})\|_{L^p_x}+\int^{s}_{\frac{s}{2}}e^{-\delta(s-\tau)}(s-\tau)^{-\frac{1}{2}-\epsilon}\|H^{\frac{1}{2}}G(\tau)\|_{L^p_x}d\tau.
\end{align}
The same arguments in the proof of (\ref{ku4}) give
\begin{align}
&\omega_{\frac{5}{4}+\epsilon}(s)\int^{s}_{\frac{s}{2}}e^{-\delta (s-\tau)}(s-\tau)^{-\frac{1}{2}-\epsilon}\|\nabla G(\tau)\|_{L^2_tL^p_x}\nonumber\\
&\lesssim C(M_1)\varepsilon_1\|\omega_{\frac{5}{4}+\epsilon}(s)\Delta\phi_s(\tau)\|_{L^{\infty}_sL^2_tL^p_x}
+C(M_1)\varepsilon_1\|\omega_{\frac{3}{4}+\epsilon}(s)\nabla\phi_s(\tau)\|_{L^{\infty}_sL^2_tL^p_x}\nonumber\\
&+C(M_1)\varepsilon_1\|\omega_{\frac{1}{2}}\phi_s(\tau)\|_{L^{\infty}_sL^2_tL^p_x}+C(M_1)\varepsilon_1.\label{lk12}
\end{align}
Thus we arrive at (\ref{koim6}) from (\ref{lk11}), (\ref{lk12}), (\ref{kj67}) and (\ref{bcde45rtyf}).\\
\noindent{\bf Step 2.} In this step we prove the desired estimates in (\ref{uojiknmpx3}) for $\partial_t\phi_s$. The proof is almost the same as Step 1 with the help of (\ref{kp5}).
\end{proof}

The estimates for the wave map tension filed are given below.
\begin{lemma}
Assume that (\ref{d7nn885658}) to (\ref{bcde45rtyf}) hold, then
the wave map tension field $Z(s,t,x)$ satisfies
\begin{align}
{\left\|s^{-\frac{1}{2}} Z(s) \right\|_{L^\infty_s L_t^1L_x^2}}\lesssim C(M_1)\varepsilon^2_1\label{Zxser}\\
{\left\| \nabla Z(s) \right\|_{L^\infty_s L_t^1L_x^2}}\lesssim C(M_1)\varepsilon^2_1\label{fracsilon}\\
{\left\| s^{\frac{1}{2}}\Delta Z(s) \right\|_{L^\infty_s L_t^1L_x^2}}\lesssim C(M_1)\varepsilon^2_1\label{facsilon}\\
{\left\| \omega_{\frac{1}{2}}\partial_s Z(s) \right\|_{L^\infty_s L_t^1L_x^2}}\lesssim C(M_1)\varepsilon^2_1.\label{cf7er7834}
\end{align}
\end{lemma}
\begin{proof}
First we notice that $Z(0,t,x)=0$ for all $(t,x)$ in $[0,T^*)\times \Bbb H^2$. And
(\ref{ab1}) shows
\begin{align}
&\partial_s Z+H Z=2h^{ii}A^{qua}_i\partial_i Z+h^{ii}A^{qua}_i\mathcal{A}_i Z+h^{ii}\mathcal{A}_iA^{qua}_i Z+h^{ii}A^{qua}_iA^{qua}_i Z\nonumber\\
&+h^{ii}(\partial_iA^{qua}_i-\Gamma^{k}_{ii}A^{qua}_k) Z
+h^{ii}(Z\wedge\widehat{\phi}_i)\phi^{qua}_i+h^{ii}(Z\wedge \phi^{qua}_i)\widehat{\phi}_i\nonumber\\
&+h^{ii}(Z\wedge \phi^{qua}_i)\phi^{qua}_i+3h^{ii}(\partial_t\widetilde{u}\wedge\partial_i\widetilde{u})\nabla_t\partial_i\widetilde{u}.\label{ku1}
\end{align}
Then Duhamel Principle gives
\begin{align}\label{ku2}
\|Z(s)\|_{L^1_tL^2_x}\lesssim \sum^3_{j=1}\int^s_0\|e^{-H(s-\tau)}\widetilde{G}_j(\tau)\|_{L^1_tL^2_x}d\tau,
\end{align}
where $\widetilde{G}_1$ denotes the one derivative term of $Z$, i.e., $\widetilde{G}_1=2h^{ii}A^{qua}_i\partial_i Z$,  $\widetilde{G}_2$ denotes the zero order derivative terms
of $Z$, and $\widetilde{G}_3$ denotes $3h^{ii}(\partial_t\widetilde{u}\wedge\partial_i\widetilde{u})\nabla_t\partial_i\widetilde{u}$.
When $s\in[0,1]$, by Lemma \ref{huguokitredcfr}, the $\widetilde{G}_2$ term in (\ref{ku2}) is bounded by
\begin{align*}
C(M_1)\varepsilon_1\|Z(s)s^{-\frac{1}{2}}\|_{L^{\infty}_s[0,1]L^1_tL^2_x}(s^{\frac{3}{2}}+s+s^{\frac{1}{2}}).
\end{align*}
By (\ref{ku3}) and Lemma \ref{huguokitredcfr}, the $\widetilde{G}_3$ term in (\ref{ku2}) is bounded by
\begin{align}\label{poe7}
\int^{s}_0\|d\widetilde{u}\|_{L^{\infty}_tL^6_x}\|\nabla\phi_t\|_{L^2_tL^6_x}\|\partial_t\widetilde{u}\|_{L^2_tL^6_x}ds'
+\int^{s}_0\|d\widetilde{u}\|_{L^{\infty}_tL^6_x}\|\sqrt{h^{ii}}A_i\phi_t\|_{L^2_tL^6_x}\|\partial_t\widetilde{u}\|_{L^2_tL^6_x}ds'.
\end{align}
Thus (\ref{ku3}) and Lemma \ref{huguokitredcfr} show that
$\widetilde{G}_3$ in (\ref{ku2}) is bounded by
\begin{align}\label{poe6}
\int^s_0\|\widetilde{G}_3\|_{L^1_tL^2_x}d\tau\lesssim s^{\frac{5}{8}-\epsilon}C(M_1)\varepsilon_1.
\end{align}
For the $\widetilde{G}_1$  term in  (\ref{ku2}), direct calculations show
\begin{align}
&\|e^{-(s-\tau)H}h^{ii}A^{qua}_i\partial_i Z\|_{L^2_x}\nonumber\\
&\lesssim \|e^{-(s-\tau)H}h^{ii}\partial_i(A^{qua}_i Z)\|_{L^2_x}+\|e^{-(s-\tau)H}h^{ii}(\partial_iA^{qua}_i) Z\|_{L^2_x}\label{kj16v}
\end{align}
By Proposition \ref{pkg2} and the boundedness of Riesz transform, the first term in (\ref{kj16v}) is bounded by
\begin{align}
 &\|e^{-(s-\tau)H}h^{ii}\partial_i(A^{qua}_i Z)\|_{L^2_x}\nonumber\\
&\lesssim \|e^{-(s-\tau)H}H^{\frac{1}{2}}H^{-\frac{1}{2}}(-\Delta)^{\frac{1}{2}}(-\Delta)^{-\frac{1}{2}}h^{ii}\partial_i(A^{qua}_i Z)\|_{L^2_x}\nonumber\\
 &\lesssim (s-\tau)^{-\frac{1}{2}}e^{-\delta(s-\tau)}\|(-\Delta)^{-\frac{1}{2}}h^{ii}\partial_i(A^{qua}_i Z)\|_{L^2_x}\nonumber\\
 &\lesssim (s-\tau)^{-\frac{1}{2}}e^{-\delta(s-\tau)}\|\sqrt{h^{ii}}A^{qua}_i Z\|_{L^2_x}\label{poe4}
\end{align}
Thus by Lemma \ref{huguokitredcfr}, (\ref{poe4}) and the second term in (\ref{kj16v}) are bounded as
\begin{align}\label{poe5}
\|e^{-(s-\tau)H}h^{ii}A^{qua}_i\partial_i Z\|_{L^2_x}&\lesssim C(M_1)\varepsilon_1(s-\tau)^{-\frac{1}{2}}e^{-\delta(s-\tau)}\| Z\|_{L^2_x}\nonumber\\
&+C(M_1)\varepsilon_1\tau^{-\frac{1}{2}}e^{-\delta(s-\tau)}\| Z\|_{L^2_x}.
\end{align}
Therefore, (\ref{ku2}), (\ref{poe6}), (\ref{poe7}), (\ref{poe5}) give (\ref{Zxser}) for $s\in(0,1)$. Using the exponential decay to $s$ of $\{A^{qua}_i\}$  and their one order derivatives in Lemma \ref{huguokitredcfr}, one obtains (\ref{Zxser}) for $s\in[1,\infty)$ by the same arguments above.

For (\ref{fracsilon}), applying $H^{\frac{1}{2}}$ to  (\ref{ku1}) we have
\begin{align*}
\|H^{\frac{1}{2}}Z(s)\|_{L^1_tL^2_x}\lesssim \|H^{\frac{1}{2}}e^{-H(s-\tau)}Z(\frac{s}{2})\|_{L^1_tL^2_x}+\sum^3_{j=1}\int^s_{\frac{s}{2}}\|H^{\frac{1}{2}}e^{-H(s-\tau)}\widetilde{G}_j(\tau)\|_{L^1_tL^2_x}d\tau.
\end{align*}
And Proposition \ref{pkg2} and Lemma \ref{huguokitredcfr} give for $s\in(0,1)$
\begin{align*}
&\int^s_{\frac{s}{2}}\|H^{\frac{1}{2}}e^{-H(s-\tau)}\widetilde{G}_1(\tau)\|_{L^1_tL^2_x}d\tau\lesssim C(M_1)\varepsilon_1\int^s_{\frac{s}{2}}(s-\tau)^{-\frac{1}{2}}\|\nabla Z(\tau)\|_{L^1_tL^2_x}d\tau \\
&\int^s_{\frac{s}{2}}\|H^{\frac{1}{2}}e^{-H(s-\tau)}\widetilde{G}_2(\tau)\|_{L^1_tL^2_x}d\tau\lesssim \int^s_{\frac{s}{2}}(s-\tau)^{-\frac{1}{2}}\tau^{-\frac{1}{2}}\|Z(\tau)\|_{L^1_tL^2_x}d\tau \\
&\int^s_{\frac{s}{2}}\|H^{\frac{1}{2}}e^{-H(s-\tau)}\widetilde{G}_3(\tau)\|_{L^1_tL^2_x}d\tau\lesssim C(M_1)\varepsilon_1\int^s_{\frac{s}{2}}(s-\tau)^{-\frac{1}{2}}\tau^{-\frac{1}{2}}d\tau
\end{align*}
Thus (\ref{fracsilon}) follows from (\ref{Zxser}), (\ref{p3ib}) and (\ref{p4ib}) when $s\in(0,1)$. Similar arguments give (\ref{fracsilon}) for $s\in[0,\infty)$.

The rest is to prove (\ref{facsilon}). Applying $H$ to (\ref{ku1}), we have by Proposition \ref{pkg2} that
\begin{align}\label{1nonu}
\|HZ\|_{L^1_tL^2_x}&\lesssim  s^{-\frac{1}{2}}e^{-\delta\frac{s}{2}}\|\nabla Z(\frac{s}{2})\|_{L^1_tL^2_x}+\sum^3_{j=1}\int^s_{\frac{s}{2}}(s-\tau)^{-\frac{1}{2}}e^{-\delta(s-\tau)}\|\nabla\widetilde{G}_j(\tau)\|_{L^1_tL^2_x}d\tau.
\end{align}
The same arguments as the proof of (\ref{ku4}) with Proposition \ref{pkg2} give
\begin{align}
\sum^2_{j=1}\omega_{\frac{1}{2}}\int^s_{\frac{s}{2}}(s-\tau)^{-\frac{1}{2}}e^{-\delta(s-\tau)}\|\nabla\widetilde{G}_j(\tau)\|_{L^1_tL^2_x}d\tau
\lesssim C(M_1)\varepsilon_1\|\omega_{\frac{1}{2}}\Delta Z\|_{L^{\infty}_sL^1_tL^2_x}+C(M_1)\varepsilon_1.\label{nonu}
\end{align}
It remains to bound $\widetilde{G}_3$. Since $\nabla_t\partial_i\widetilde{u}=\nabla_i\partial_t\widetilde{u}$, we have
$\widetilde{G}_3=3h^{ii}(\phi_t\wedge\phi_i)(\partial_i\phi_t+A_i\phi_t)$. Then the explicit formula for $\Gamma^{k}_{ij}$ and $h^{jk}$ yields
\begin{align*}
|\nabla\widetilde{G}_3(\tau)|&\lesssim \sqrt{h^{pp}}\left|\big((\partial_ph^{ii})(\phi_t\wedge\phi_i)
+h^{ii}(\partial_p\phi_t)\wedge\phi_i+h^{ii}\phi_t\wedge\partial_p\phi_i\big)(\partial_i\phi_t+A_i\phi_t)\right|\\
&+\sqrt{h^{pp}}\left|h^{ii}(\phi_t\wedge\phi_i)\right|\left|\partial_p\partial_i\phi_t+(\partial_pA_i)\phi_t+A_i\partial_p\phi_t\right|\\
&\lesssim \big(|du||\nabla\phi_t|+|\nabla du||\phi_t|+|du||A||\phi_t|\big)\big(|\nabla\phi_t|+|A||\phi_t|\big)\\
&+|\phi_t||du|(|\nabla^2\phi_t|+|A||\nabla\phi_t|+|\nabla A||\phi_t|+|A||\phi_t|)
\end{align*}
By (\ref{ku4}), (\ref{ku3}) and Lemma \ref{huguokitredcfr}, the $\widetilde{G}_3$ term in (\ref{1nonu}) is bounded by
\begin{align*}
\int^s_{\frac{s}{2}}\|He^{-H(s-\tau)}\widetilde{G}_3(\tau)\|_{L^1_tL^2_x}&\lesssim C(M_1)\varepsilon_1\int^s_{\frac{s}{2}}(s-\tau)^{-\frac{1}{2}}
e^{-\delta(s-\tau)}\tau^{-\frac{7}{8}-\epsilon}d\tau,\\
 &\mbox{ }{\rm{when}}\mbox{  }s\in[0,1];\\
\int^s_{\frac{s}{2}}\|He^{-H(s-\tau)}\widetilde{G}_3(\tau)\|_{L^1_tL^2_x}&\lesssim C(M_1)\varepsilon_1\int^s_{\frac{s}{2}}(s-\tau)^{-\frac{1}{2}}
e^{-\delta(s-\tau)}\tau^{-L}d\tau, \\
&\mbox{ }{\rm{when}}\mbox{  }s\in[1,\infty).
\end{align*}
This combined with (\ref{nonu}), Proposition \ref{pkg2} yields (\ref{facsilon}). (\ref{cf7er7834}) follows by (\ref{ku1}), (\ref{facsilon}) and previously obtained bounds for
$\|\widetilde{G}_{i}\|_{L^1_tL^2_x}$, $i=1,2,3$.
\end{proof}

\begin{lemma}
Assume that (\ref{d7nn885658}) to (\ref{bcde45rtyf}) hold, then for $0<\gamma\ll1$
\begin{align}
\left\| s^{-\frac{1}{2}+\epsilon}Z(s)\right\|_{L_s^\infty L_t^2L_x^{3+\gamma}}+\left\| \omega_{\frac{1}{2}}Z(s)\right\|_{L_s^\infty L_t^2L_x^{3+\gamma}}& \lesssim {C(M_1)\varepsilon^2 _1}\label{time90}\\
{\left\|{\omega_{\frac{1}{2}}{\partial _t}{\phi _t}(s)} \right\|_{L_s^\infty L_t^2L_x^{3+\gamma}}} &\lesssim {C(M_1)\varepsilon^2 _1}\label{bbd75658}\\
{\left\|{{\partial _t}{A_t}(s)} \right\|_{L_s^\infty L_t^2L_x^{3+\gamma}}} &\lesssim  {C(M_1)\varepsilon^2 _1}\label{time37}\\
{\left\|{{A_t}(s)} \right\|_{L_s^\infty L_t^1L_x^{\infty}}} &\lesssim  {C(M_1)\varepsilon^2 _1}\label{AAA1}
\end{align}
\end{lemma}
\begin{proof}
Applying Duhamel principle to (\ref{ku1}), one obtains from \underline{Lemma 6.1} that
\begin{align}
\left\|Z(s)\right\|_{L^2_tL^{3+\gamma}_x}&\le \sum^3_{j=1}\int^{s}_0\|e^{-H(s-\tau)}\widetilde{G}_i\|_{L^2_tL^{3+\gamma}_x}d\tau\nonumber\\
&\lesssim \sum^3_{j=1}\int^{s}_0e^{-\delta(s-\tau)}\|\widetilde{G}_j\|_{L^2_tL^{3+\gamma}_x}d\tau.\label{pimmvb}
\end{align}
where $\{\widetilde{G}_i\}^3_{i=1}$ are defined below (\ref{ku2}).
Lemma \ref{huguokitredcfr} and Proposition \ref{pkg1} give
\begin{align}
&\int^{s}_0\|e^{-H(s-\tau)}\widetilde{G}_1\|_{L^2_tL^{3+\gamma}_x}d\tau\nonumber\\
&\lesssim\int^{s}_0\|e^{-H(s-\tau)}H^{\frac{1}{2}}H^{-\frac{1}{2}}(-\Delta)^{\frac{1}{2}}(-\Delta)^{-\frac{1}{2}}\widetilde{G}_1\|_{L^2_tL^{3+\gamma}_x}d\tau\nonumber\\
&\lesssim\int^{s}_0(s-\tau)^{-\frac{1}{2}-\epsilon}e^{-\delta(s-\tau)}\|(-\Delta)^{-\frac{1}{2}}\widetilde{G}_1\|_{L^2_tL^{3+\gamma}_x}d\tau.\label{swe3}
\end{align}
Due to the explicit expressions for $h^{ij}$, we can write $\widetilde{G}_1=2h^{ii}A^{qua}_i\partial_i Z$ in the form  $\widetilde{G}_1=2\sqrt{h^{ii}}\partial_i(\sqrt{h^{ii}}A^{qua}_i Z)-2(h^{ii}\partial_iA^{qua}_i)Z$. Then (\ref{swe3}), Lemma \ref{huguokitredcfr} and the boundedness of Riesz transform yield
\begin{align}\label{swe4}
\int^{s}_0\|e^{-H(s-\tau)}\widetilde{G}_1\|_{L^2_tL^{3+\gamma}_x}d\tau\lesssim\int^{s}_0C(M_1)\varepsilon_1(s-\tau)^{-\frac{1}{2}-\epsilon}\|Z\|_{L^2_tL^{3+\gamma}_x}d\tau.
\end{align}
The $ \widetilde{G}_i, i=2,3$ are bounded by Lemma \ref{huguokitredcfr}, (\ref{ku3}) and (\ref{bbd75658}):
\begin{align}
\sum_{j=2,3}\int^{s}_0\|e^{-H(s-\tau)}\widetilde{G}_{j}\|_{L^2_tL^{3+\gamma}_x}d\tau
&\lesssim C(M_1)\varepsilon_1\int^{s}_0(s-\tau)^{-\frac{1}{2}-\epsilon}\|Z\|_{L^2_tL^{3+\gamma}_x}d\tau\nonumber\\
&+C(M_1)\varepsilon_1\int^{s}_0(s-\tau)^{-\frac{3}{8}-\epsilon}d\tau.\label{swe5}
\end{align}
Thus we obtain (\ref{time90}) from (\ref{swe4}), (\ref{swe5}). By $Z=D_t\phi_t-\phi_s$, we get
\begin{align*}
\|{\partial _t}{\phi _t}(s)\|_{L^2_tL^{3+\gamma}_x}\lesssim \|{\phi _s}+A_t{\phi _t}+Z\|_{L^2_tL^{3+\gamma}_x}.
\end{align*}
Hence, (\ref{bbd75658}) when $s\in[0,1]$ follows from (\ref{ku3}), (\ref{bbd75658}) and (\ref{bcde45rtyf}).  (\ref{bbd75658}) when $s\ge1$ follows by the same arguments with (\ref{pimmvb}) replaced by
\begin{align*}
\left\|Z(s)\right\|_{L^2_tL^{3+\gamma}_x}&\lesssim e^{-\delta\frac{s}{2}}\left\|Z(\frac{s}{2})\right\|_{L^2_tL^{3+\gamma}_x}
+\sum^3_{j=1}\int^{s}_{\frac{s}{2}}\|e^{-H(s-\tau)}\widetilde{G}_i\|_{L^2_tL^{3+\gamma}_x}d\tau.
\end{align*}
(\ref{time37}) follows from (\ref{bbd75658}), (\ref{ku3}), (\ref{huojikn}) and
\begin{align*}
\partial_tA_t=\int^{\infty}_s\partial_t\phi_t\wedge\phi_sds'+\int^{\infty}_s\phi_t\wedge\partial_t\phi_sds'.
\end{align*}
(\ref{AAA1}) follows by Sobolev embedding, (\ref{ku3}), (\ref{uojiknmpx3}).
\end{proof}

The rest arguments are similar to \cite{EFE376} in the small energy case. We present most details here to keep the completeness.
\begin{proposition}\label{vcvfcdexser}
Assume that (\ref{d7nn885658})to (\ref{bcde45rtyf}) hold.
Then we have for $p\in(2,6)$
\begin{align}
&{\left\| {\omega_{\frac{1}{2}}{D^{ - \frac{1}{2}}}{{\partial}_t}{\phi _s}} \right\|_{L_s^\infty L_t^2L_x^p([{0,T}] \times {\Bbb H^2})}} + {\left\| {\omega_{\frac{1}{2}}{D^{\frac{1}{2}}}{\phi _s}} \right\|_{L_t^2L_x^p([0,T] \times {\Bbb H^2})}} \nonumber\\
&+ {\left\| {{\omega}_{\frac{1}{2}}{\partial _t}{{\phi}_s}} \right\|_{L_s^\infty L_t^\infty L_x^2([0,T] \times {\Bbb H^2})}}
+ {\left\| {\omega_{\frac{1}{2}}\nabla {{\phi _s}}} \right\|_{L_s^\infty L_t^\infty L_x^2([0,T] \times {\Bbb H^2})}} \le \varepsilon _1^2C(M_1). \label{huoji}
\end{align}
\end{proposition}
\begin{proof}
By Lemma \ref{REWSDERF} and Proposition \ref{de486tfcer},  for any $p\in(2,6)$ there holds
\begin{align}
&{\omega_{\frac{1}{2}}}{\left\| {{\partial _t}{{\phi _s}}} \right\|_{L_t^\infty L_x^2}} + {\omega_{\frac{1}{2}}}{\left\| {\nabla {\phi _s}} \right\|_{{L_t^\infty L_x^2}}} + {\omega_{\frac{1}{2}}}{\left\| D^{\frac{1}{2}}{{\phi _s}} \right\|_{L_t^2L_x^p}}\nonumber\\
&+ {\omega_{\frac{{1}}{2}}}{\left\| D^{-\frac{1}{2}}{\partial _t}{\phi _s} \right\|_{{L_t^2L_x^p}}} +{\omega_{\frac{1}{2}}}{\left\| {{\rho ^\sigma }\nabla {\phi _s}} \right\|_{L_t^2L_x^2}} \nonumber\\
&\lesssim {\omega_{{\frac{1}{2}}}}{\left\| {{\partial _t}{\phi _s}(0,s,x)} \right\|_{L_x^2}} + {\omega_{\frac{1}{2}}}{\left\| {\nabla {\phi _s}(0,s,x)} \right\|_{L_x^2}} + {\omega_{{\frac{1}{2}}}}{\left\|{\mathbf{G} }\right\|_{L_t^1L_x^2}}.\label{vcde4xser}
\end{align}
where $\mathbf{G}$ denotes the inhomogeneous term.
The $\phi_s(0,s,x)$ term is bounded by Proposition \ref{fcuyder47te} and (\ref{kp32}). In fact, we have
\begin{align*}
\|\omega_{\frac{1}{2}}\nabla_{t,x}\phi_s(0,s,x)\|_{L^2_x}&\lesssim \|\omega_{\frac{1}{2}}\nabla_{t,x}\partial_s\mathbf{U}\|_{L^2_x}+\|\omega_{\frac{1}{2}}\sqrt{h^{\gamma\gamma}}A_{\gamma}\partial_sU\|_{L^2_x}\\
&\lesssim \|(\nabla du_0,\nabla du_1)\|_{L^2_x\times L^2_x}+C(M_1)\|\nabla du_0\|_{L^2_x}\\
&\lesssim M_0+M^2_0+C(M_1)\varepsilon_1M_0.
\end{align*}
where $\mathbf{U}(s,x)$ denotes the heat flow with initial data $u_0$.
The terms in (\ref{vcde4xser}) involved with $A_t$ can be deal  with as follows
\begin{align*}
{\left\|{\omega_{\frac{1}{2}}} {{A_t}{\partial _t}{\phi _s}} \right\|_{L_t^1L_x^2}} &\lesssim{\left\| {{A_t}} \right\|_{L_t^1L_x^\infty }}{\left\| {\omega_{\frac{1}{2}}} {{\partial _t}{\phi _s}} \right\|_{L_t^\infty L_x^2}} \\
{\left\|{\omega_{\frac{1}{2}}} {{A_t}{A_t}{\phi _s}} \right\|_{L_t^1L_x^2}} &\lesssim {\left\| {{A_t}} \right\|_{L_t^1L_x^\infty }}{\left\| {{A_t}} \right\|_{L_t^\infty L_x^\infty }}{\left\| {\omega_{\frac{1}{2}}}{{\phi _s}} \right\|_{L_t^\infty L_x^2}} \\
{\left\| {\omega_{\frac{1}{2}}}{{\partial _t}{A_t}{\phi _s}} \right\|_{L_t^1L_x^2}} &\lesssim{\left\| {{\partial _t}{A_t}} \right\|_{L_t^2L_x^{3+\gamma}}}{\left\|{\omega_{\frac{1}{2}}} {{\phi _s}} \right\|_{L_t^2L_x^a}},
\end{align*}
where $\frac{1}{a}+\frac{1}{3+\gamma}=\frac{1}{2},$ and $a\in(2,6)$.
They are admissible by (\ref{d7nn885658})-(\ref{bcde45rtyf}), (\ref{AAA1}) and (\ref{time37}).
The $\partial_t\widetilde{u}$ term in (\ref{vcde4xser}) is bounded by
\begin{align*}
{\left\|{\omega_{\frac{1}{2}}} {({\partial _t}\widetilde{u}\wedge{\partial _s}\widetilde{u})({\partial _t}\widetilde{u})} \right\|_{L_t^1L_x^2}} \lesssim {\left\| {{\partial _t}\widetilde{u}} \right\|_{L_t^2L_x^{6+2\gamma}}}{\left\| {{\partial _t}\widetilde{u}} \right\|_{L_t^\infty L_x^{6+2\gamma}}}{\left\|{\omega_{\frac{1}{2}}} {{\phi _s}} \right\|_{L_t^2L_x^a}},
\end{align*}
where $\frac{1}{a}+\frac{1}{3+\gamma}=\frac{1}{2},$ and $a\in(2,6)$. This is acceptable due to (\ref{d7nn885658})-(\ref{bcde45rtyf}).
The $\partial_sZ$ term  in (\ref{vcde4xser}) is bounded by (\ref{cf7er7834}).
The $A^{qua}_i$ terms  in (\ref{vcde4xser}) need more efforts to bound. We state the estimates for $A^{qua}_i$ terms as a claim:\\
We claim that with the assumptions of Proposition \ref{vcvfcdexser}, there holds\\
\begin{align}
{\|{\omega_{\frac{1}{2}}} {{h^{ii}}A_i^{qua}{\partial _i}{\phi _s}} \|_{L_t^1L_x^2}} &\lesssim {\varepsilon _1}{\omega_{\frac{1}{2}}}{\left\| {{\rho ^\sigma }\nabla {\phi _s}} \right\|_{L_t^2L_x^2}} + C(M_1)\varepsilon _1^2\label{hiyura2}\\
\|{\omega_{\frac{1}{2}}} {\phi _s}h^{ii}A_i^{qua}\mathcal{A}_i \|_{{L_t^1L_x^2}}&\lesssim C(M_1)\varepsilon _1^2\label{p67yugbiha2}\\
\| {\omega_{\frac{1}{2}}}{\phi _s}  h^{ii}A_i^{qua}A_i^{qua}\|_{L_t^1L_x^2}&\lesssim C(M_1)\varepsilon _1^2\label{yugira2}\\
\| {\omega_{\frac{1}{2}}}{\phi _s}h^{ii}\partial_iA_i^{qua} \|_{{L_t^1L_x^2}}&\lesssim C(M_1)\varepsilon _1^2\label{goplira2}\\
\|{\omega_{\frac{1}{2}}}{\phi _s} h^{ii}\Gamma _{ii}^kA_k^{qua}\|_{L_t^1L_x^2}&\lesssim C(M_1)\varepsilon _1^2\label{hlira2}.
\end{align}
\noindent {\bf Step 1. Proof of (\ref{hiyura2})}
Recall $\phi_i=\widehat{\phi}_i+\int^{\infty}_s\partial_s \phi_ids'$, then
\begin{align*}
 A^{qua} _i:= \int^\infty_s (\phi _i\wedge \phi _s) d\kappa  = \int^\infty _s {\left( {\int_{\kappa}^\infty  {{\partial _s}} {\phi _i}(\tau )d\tau  + \widehat{\phi} _i } \right) \wedge {\phi _s}} (\kappa)d\kappa.
\end{align*}
Inserting this expansion to the LHS of (\ref{hiyura2}) we get
\begin{align*}
{\left\| {\omega_{\frac{1}{2}}}{{h^{ii}}A_i^{qua}{\partial _i}{\phi _s}} \right\|_{L_t^1L_x^2}} &\lesssim {\left\|{\omega_{\frac{1}{2}}}\big( {{h^{ii}}\widehat{\phi }_i\wedge\int_s^\infty  {{\phi _s}(\kappa)d\kappa} }\big) {\partial _i}{\phi _s}\right\|_{L_t^1L_x^2}}\\
&+ {\left\|{\omega_{\frac{1}{2}}}\left( {\int_s^\infty  {{\phi _s}(\kappa) \wedge \left( {\int_{\kappa}^\infty  {{\partial _s}} {\phi _i}(\tau )d\tau } \right)d\kappa} } \right) {{h^{ii}} {\partial _i}{\phi _s}} \right\|_{L_t^1L_x^2}} \\
&:= F_1 + F_2
\end{align*}
The $F_1$ term is bounded by
\begin{align}
 {F_1} &\lesssim {\omega_{\frac{1}{2}}}{\left\| {{\rho ^\sigma }\nabla {\phi _s}} \right\|_{L_t^2L_x^2}}{\left\| {\int_s^\infty  {{\rho ^{ - \sigma }}\sqrt{h^{ii}}\widehat{\phi }_i\wedge{\phi _s}(\kappa)d\kappa} } \right\|_{L_t^2L_x^{\infty}}}\nonumber \\
 &\lesssim {\omega_{\frac{1}{2}}}{\left\| {{\rho ^\sigma }\nabla {\phi _s}} \right\|_{L_t^2L_x^2}}{\left\| {{\rho ^{ - \sigma }}\widehat{\phi } } \right\|_{L_x^{\infty}}}\int_s^\infty  {{{\left\| {{\phi _s}(\kappa)} \right\|}_{L_t^2L_x^{\infty}}}d\kappa}\nonumber  \\
 &\lesssim {\omega_{\frac{1}{2}}}{\left\| {{\rho ^\sigma }\nabla {\phi _s}} \right\|_{L_t^2L_x^2}}{\left\| {{\rho ^{ - \sigma }}\widehat{\phi} _i} \right\|_{L_x^{\infty}}}{\left\| {\omega_{\frac{3}{4}+\epsilon}(s){{\left\| {\nabla{\phi _s}(s)} \right\|}_{L_t^2L_x^{4}}}} \right\|_{L_s^\infty }},\label{xs390puygjexser}
\end{align}
where $\widehat{\phi }$, which denotes $\widehat{\phi }_idx^i$, decays exponentially in spatial space and we used the Sobolev embedding in the last line. Thus, by Lemma \ref{i97clv} we have an acceptable bound:
\begin{align*}
{F_1} \lesssim C(M_1)\varepsilon_1{\omega_{\frac{1}{2}}}{\left\| {{\rho ^\sigma }\nabla {\phi _s}} \right\|_{L_t^2L_x^2}}.
\end{align*}
The $F_2$ term is bounded by
\begin{align*}
 {F_2} \lesssim {\left\| {\omega_{\frac{1}{2}}}{\nabla {\phi _s}} \right\|_{L_t^\infty L_x^2}}\int_s^\infty  {{{\left\| {{\phi _s}(\kappa)} \right\|}_{L_t^2L_x^\infty }}\left( {\int_{\kappa}^\infty  {{{\left\| {\nabla {\phi _s}(\tau )} \right\|}_{L_t^2L_x^\infty }}} d\tau } \right)d\kappa}.
\end{align*}
Moreover, by Sobolev embedding and Lemma \ref{i97clv}, we obtain for $\vartheta\in (\frac{7}{5},\frac{3}{2})$
\begin{align}
&{\left\| {\nabla {\phi _s}(\tau )} \right\|_{L_t^2L_x^\infty }} \nonumber\\
&\lesssim {\left( {{\omega_{\frac{3}{4}+\epsilon}}{{\left\| {\nabla {\phi _s}(\tau )} \right\|}_{L_t^2L_x^5}}} \right)^{1 - \varsigma }}{\left( {{\omega_{\vartheta +\epsilon}}{{\left\| {{D^\vartheta }{\phi _s}(\tau )} \right\|}_{L_t^2L_x^5}}} \right)^\varsigma }{\left( {{\omega _{\frac{3}{4}+\epsilon}}} \right)^{ - \varsigma + 1}}{\left( {{\omega_{\vartheta+\epsilon}}} \right)^{ - \varsigma }},\label{v1213dexser}
\end{align}
where $\varsigma  = \frac{2}{{5(\vartheta  - 1)}}$.
Similarly, the Sobolev embedding $\|g\|_{L^{\infty}}\lesssim \|D^{\frac{1}{2}}g\|_{L^5_x}$ shows
\begin{align}\label{w7k2cd}
{\left\|\omega_{\frac{1}{2}} {{\phi _s}(\tau )} \right\|_{L_t^2L_x^\infty }} \lesssim {\varepsilon _1}.
\end{align}
Therefore choosing $\vartheta$ slightly above $\frac{7}{5}$, we conclude from (\ref{v1213dexser}) and (\ref{w7k2cd}) that
\begin{align}
{F_2}\lesssim \varepsilon _1^2{\left\|{\omega_{\frac{1}{2}}} {\nabla {\phi _s}} \right\|_{L^{\infty}_sL_t^\infty L_x^2}}\label{vx3pdexser}.
\end{align}
Lemma \ref{i97clv} together with (\ref{xs390puygjexser}), (\ref{vx3pdexser}) gives (\ref{hiyura2}).\\
\noindent {\bf Step 2. Proof of (\ref{p67yugbiha2}), (\ref{yugira2}) and (\ref{goplira2})} These three terms can be bounded as \cite{EFE376}, we do not repeat the arguments here to avoid overlap.

\end{proof}

Lemma \ref{i97clv} and Proposition \ref{vcvfcdexser} yield
\begin{proposition}\label{zz2ccuvc344fyeo}
Assume that the solution to (\ref{z32ccc344fyeo}) satisfies (\ref{d7nn885658}) to (\ref{oootyiys}), then for any  $p\in(2,6)$, $\theta\in[0,\frac{1}{2})$, $\theta_1\in[0,\frac{3}{4}]$
\begin{align*}
{\left\| \omega_{\frac{1}{2}}{\nabla\phi_s} \right\|_{L^{\infty}_sL_t^\infty L_x^2}} + {\left\|\omega_{\frac{1}{2}} D^{\frac{1}{2}}{\phi_s} \right\|_{L^{\infty}_sL_t^2 L_x^p}} &\lesssim {C(M_1)\varepsilon^2 _1}\\
{\left\| \omega_{\frac{1}{2}+\theta_1+\epsilon}{(-\Delta)^{\theta_1}D^{-\frac{1}{2}}{\partial _t}\phi_s} \right\|_{L^{\infty}_sL_t^2L_x^p}}
+\left\|\omega_{\frac{1}{2}+\theta+\epsilon}(-\Delta)^{\theta}D^{\frac{1}{2}}{\phi_s} \right\|_{L^{\infty}_sL_t^2L_x^p}&\lesssim{C(M_1)\varepsilon^2 _1}.
\end{align*}
\end{proposition}

\subsection{Close all the bootstrap}

\begin{lemma}\label{oi89}
Assume that the solution to (\ref{z32ccc344fyeo}) satisfies (\ref{d7nn885658}) and (\ref{oootyiys}), then for any $p\in(2,6+2\gamma]$
\begin{align}
{\left\| {du} \right\|_{L_t^\infty L_x^2([0,T] \times {\Bbb H^2})}} + {\left\| {\nabla du} \right\|_{L_t^\infty L_x^2([0,T] \times {\Bbb H^2})}}&\le CM_0\label{boot10}\\
{\left\| {{\partial _t}u} \right\|_{L_t^\infty L_x^2([0,T] \times {\Bbb H^2})}} + {\left\| {\nabla {\partial _t}u} \right\|_{L_t^\infty L_x^2([0,T] \times {H^2})}} &\lesssim C(M_1){\varepsilon^2 _1}\label{fengyucbedcfr}\\
{\left\| D^{\frac{1}{4}}\phi_t \right\|_{L_t^2L_x^p([0,T] \times {\Bbb H^2})}} &\lesssim C(M_1){\varepsilon^2_1}.\label{hccfdgjvcbedcfr}
\end{align}
\end{lemma}
\begin{proof}
{\bf Step 1.} We prove (\ref{hccfdgjvcbedcfr}) first. By  $D_s\phi_t=D_t\phi_s$, $A_s=0$, one has
\begin{align}
{\left\| {D^{\frac{1}{4}}{\phi _t}(0,t,x)} \right\|_{L_t^2L_x^p}}& \lesssim{\left\| {\int_0^\infty  {\left| D^{\frac{1}{4}}{{\partial _s}{\phi _t}} \right|} ds} \right\|_{L_t^2L_x^p}} \nonumber\\
&\lesssim{\left\|D^{\frac{1}{4}} ({\partial _t}{\phi _s}) \right\|_{L_s^1L_t^2L_x^p}}+{\left\|D^{\frac{1}{4}} ({A _t}{\phi _s}) \right\|_{L_s^1L_t^2L_x^p}}.\label{upx4}
\end{align}
{\bf Step 2.} We verify (\ref{boot10}) in this step.
(\ref{uojiknmpx3}) shows for $\vartheta\in(\frac{1}{4},\frac{1}{2})$, $q\in (2,6)$
\begin{align}\label{upx3}
{\left\|D^{\vartheta} ({\partial _t}{\phi _s}) \right\|_{L_s^1L_t^2L_x^q}}\lesssim C(M_1)\varepsilon_1.
\end{align}
And by Sobolev embedding one has
$${\left\|D^{\frac{1}{4}} {{\partial _t}{\phi _s}} \right\|_{L_x^{6 + 2\gamma }}} \lesssim  {\left\| D^\vartheta \partial_t\phi _s \right\|_{L_x^{6 - \eta }}},
$$
where $\frac{\vartheta }{2} -\frac{1}{8}= \frac{1}{{6 - \beta }} - \frac{1}{{6 + 2\gamma }}$, $0<\beta\ll1,0<\gamma\ll1$.
Thus the first term in (\ref{upx4}) is acceptable by (\ref{upx3})  and (\ref{ku3}).
For the second term in (\ref{upx4}), by Sobolev embedding
\begin{align*}
&{\left\|D^{\frac{1}{4}} ({A _t}{\phi _s}) \right\|_{L_s^1L_t^2L_x^p}}\lesssim {\left\|\nabla({A _t}{\phi _s}) \right\|_{L_s^1L_t^2L_x^p}}\\
&\lesssim{\left\|\nabla{\phi _s} \right\|_{L_s^1L_t^{\infty}L_x^{\infty}}\left\|A_t\right\|_{L_s^{\infty}L_t^{2}L_x^{p}}}+
{\left\|\nabla{A _t} \right\|_{L_s^1L_t^2L_x^p}\left\|\phi_s\right\|_{L_s^{\infty}L_t^{\infty}L_x^{\infty}}}.
\end{align*}
Meanwhile, we have
\begin{align*}
\|{A_t}\|_{L_t^2L_x^p}&\le \int^{\infty}_s\|{\phi_t}\|_{L_t^2L_x^p}\|{\phi_s}\|_{L_t^{\infty}L_x^{\infty}}d\tau\\
&\lesssim C(M_1)\varepsilon^2_1\min(1,s^{-L})\\
\|\nabla{A_t}\|_{L_t^2L_x^p}&\lesssim \int^{\infty}_s\|\nabla{\phi_t}\|_{L_t^2L_x^p}\|{\phi_s}\|_{L_t^{\infty}L_x^{\infty}}d\tau\\
&\lesssim C(M_1)\varepsilon^2_1\min(1,s^{-L})
\end{align*}
Thus (\ref{uojiknmpx3}) and Lemma \ref{huguokitredcfr} imply that the second term in (\ref{upx4}) is also acceptable, thus proving (\ref{hccfdgjvcbedcfr}).\\
\noindent{\bf Step 3.1.} We prove (\ref{fengyucbedcfr}) in this step. Recalling ${\phi _i}(0,t,x) =\widehat{ \phi }_i  + \int_0^\infty  {{\partial _s}{\phi _i}ds'}$, and  $|d\widetilde{u}|\le \sqrt{h^{ii}}|\phi_i|$,  $\|\sqrt{h^{ii}}\widehat{\phi }_i \|_{L^2}\le \|dQ\|_{L^2}\le M_0$, it suffices to show for all $t \in[0,T] $
$$\int_0^\infty  \|\sqrt{h^{ii}}{\partial _s}{\phi _i}\|_{L^2_x}d\kappa\lesssim C(M_1)\varepsilon_1.
$$
which is acceptable by applying Proposition \ref{zz2ccuvc344fyeo}, Lemma \ref{huguokitredcfr} and $|\sqrt{h^{ii}}\partial_s\phi_i|\lesssim |\nabla \phi_s|+\sqrt{h^{ii}}|A_i||\phi_s|$. Hence, we get
\begin{align*}
\|du\|_{L^2_x}\le M_0.
\end{align*}
{\bf Step 3.2.}
Recalling the evolution equation of $\phi_s$  in the heat flow direction,
\begin{align*}
\phi_s&=h^{kl}D_k\phi_l-h^{kl}\Gamma^p_{kl}\phi_p\\
\partial_s\phi_s&=h^{kl}D_kD_l\phi_s-h^{kl}\Gamma^p_{kl}D_p\phi_s+h^{kl}(\phi_s\wedge\phi_k)\phi_l,
\end{align*}
one obtains by integration by parts,
\begin{align*}
 \frac{d}{{ds}}\left\| {\tau (\widetilde{u})} \right\|^2_{L^2_x}=  - 2{h^{kl}}\left\langle {{D_k}{\phi _s},{D_l}{\phi _s}} \right\rangle+\left\langle h^{lk}(\phi_s\wedge\phi_l)\phi_k,\phi_s\right\rangle.
\end{align*}
Thus by $\|\partial_s\widetilde{u}\|_{L^2}\lesssim e^{-\delta s}$, we have
\begin{align}
 &\left\| {\tau (\widetilde{u}(0,t,x))} \right\|_{L_x^2}^2 \lesssim \int_0^\infty  {{h^{ii}}\left\langle {{D_i}{\phi _s},{D_i}{\phi _s}} \right\rangle } ds \le \int_0^\infty  {\left\langle {\nabla {\phi _s},\nabla {\phi _s}} \right\rangle } ds\nonumber\\
 &+ \int_0^\infty  {{h^{ii}}\left\langle {{A_i}{\phi _s},{A_i}{\phi _s}} \right\rangle } ds+4\int^{\infty}_0|d\widetilde{u}|^2|\phi_s|^2ds.\label{z32ccuvc344fyeo}
\end{align}
Recall the inequality
$$\left\| {\nabla du} \right\|_{L_x^2}^2 \lesssim \left\| {\tau (u)} \right\|_{L_x^2}^2 + \left\| {du} \right\|_{L_x^2}^2.
$$
due to the nonnegative sectional curvature property of $N=\Bbb H^2$ and integration by parts.
Then (\ref{z32ccuvc344fyeo}) gives
\begin{align}
\left\| {\nabla d\widetilde{u}(0,t,x)} \right\|_{L_x^2}^2 &\lesssim \int_0^\infty  {\left\langle {\nabla {\phi _s},\nabla {\phi _s}} \right\rangle } ds +\int^{\infty}_0|d\widetilde{u}|^2|\phi_s|^2ds\nonumber\\
&+ \int_0^\infty  {{h^{ii}}\left\langle {{A_i}{\phi _s},{A_i}{\phi _s}} \right\rangle } ds+\|d\widetilde{u}(0,t,x)\|^2_{L^2_x}.\label{hjbn7v89}
\end{align}
Notice that the $|d\widetilde{u}|$ term has been estimated before, then by Proposition \ref{zz2ccuvc344fyeo} , Lemma \ref{huguokitredcfr} and (\ref{hjbn7v89}),
$$ {\left\| {\nabla d\widetilde{u}} \right\|_{L_t^\infty L_x^2([0,T] \times {\Bbb H^2})}}\le M_0+\varepsilon^2_1C(M_1).
$$
{\bf Step 4. Estimates for $|\nabla\partial_t\widetilde{u}|$ in (\ref{fengyucbedcfr}) }
By integration by parts,
\begin{align*}
 \frac{d}{{ds}}\left\| {\nabla {\partial _t}\widetilde{u}} \right\|_{{L^2}}^2 =  - 2\left\langle {{D_t}{\phi _s},{h^{jj}}{D_j}{D_j}{\phi _t} - {h^{jj}}\Gamma _{jj}^l{D_l}{\phi _t}} \right\rangle  + 2{h^{jj}}\left\langle {({\phi _s} \wedge {\phi _j}){\phi _t},{D_j}{\phi _t}} \right\rangle.
\end{align*}
By the parabolic equation of $\phi_t$ along the heat flow direction, one has
$$\frac{1}{2}\frac{d}{{ds}}\left\| {\nabla {\partial _t}\widetilde{u}} \right\|_{{L^2}}^2 =  - \left\langle {{D_t}{\phi _s},{D_s}{\phi _t}} \right\rangle  +{h^{jj}}\left\langle {({\phi _s} \wedge {\phi _j}){\phi _t},{D_j}{\phi _t}} \right\rangle  + {h^{jj}}\left\langle {{D_t}{\phi _s},({\phi _t} \wedge {\phi _j}){\phi _j}} \right\rangle.
$$
Consequently, we obtain
\begin{align*}
 &\left\| {\nabla {\partial _t}\widetilde{u}(0,t,x)} \right\|_{{L^2}}^2 \le 4\int_0^\infty  {\left\langle {{\partial _t}{\phi _s},{\partial _t}{\phi _s}} \right\rangle ds'}  + 4\int_0^\infty  {\left\langle {{A_t}{\phi _s},{A_t}{\phi _s}} \right\rangle ds'}  \\
 &+ 2\int_0^\infty  {{{\left\| {{\phi _s}} \right\|}_{L_x^2}}{{\left\| {d\widetilde{u}} \right\|}_{L_x^\infty }}{{\left\| {{\partial _t}\widetilde{u}} \right\|}_{L_x^\infty }}{{\left\| {\nabla {\partial _t}\widetilde{u}} \right\|}_{L_x^2}}ds'}  + 2{\int_0^\infty  {\left\| {{\partial _t}\widetilde{u}} \right\|} _{L_x^2}}\left\| {d\widetilde{u}} \right\|_{L_x^\infty }^2{\left\| {{D_t}{\phi _s}} \right\|_{L_x^2}}ds'.
\end{align*}
Hence by Proposition \ref{fcuyder47te}, Proposition \ref{zz2ccuvc344fyeo} and Lemma \ref{huguokitredcfr}, we deduce that
$$\left\| {\nabla {\partial _t}\widetilde{u}(0,t,x)} \right\|_{{L^2}}^2 \lesssim \varepsilon _1^4C(M_1).
$$
So, we have obtained all estimates in (\ref{fengyucbedcfr}) and (\ref{hccfdgjvcbedcfr}).
\end{proof}

\subsection{Proof of Theorem 1.1}

By Proposition \ref{vskfheccc344fyeo} and Lemma \ref{oi89}, (\ref{1988}), (\ref{1989}), (\ref{dozuoki}),  (1.1) has a global solution and $\phi_s$ satisfies
\begin{align}
\|\phi_s\|_{L^2_tL^4_x}+\|\partial_t\phi_s\|_{L^2_tL^4_x}\le C(s).
\end{align}
Then Theorem 1.1. follows by the same arguments in [Section 8, \cite{EFE376}].

\section{Appendix A}

\subsection{Intrinsic v.s. extrinsic formulations}

The intrinsic and extrinsic formulations are equivalent in the following sense.
\begin{lemma}[\cite{LZ}]\label{swe345t}
Assume  that $Q$ is an admissible harmonic map.
If $u\in \mathcal{H}_{Q}^k$ then for $k=2,3$, then there exist continuous functions ${\Upsilon}_1,\Upsilon_2$ such that
\begin{align}
\|u\|_{{\mathcal H}^k_Q}&\le C(R_0,\|u\|_{\mathfrak{H}^2}){\Upsilon}_1(\|u\|_{\mathfrak{H}^k})\label{jia8}\\
\|u\|_{\mathfrak{H}^k}&\le C(R_0,\|u\|_{{\mathcal{H}}_{Q}^2})\Upsilon_2(\|u\|_{\mathcal{H}^k_Q})\label{jia9}.
\end{align}
\end{lemma}

The $\mathcal{H}_{Q}^k$ space implies the map $u$ has a compact image in the target $N=\Bbb H^2$.
\begin{corollary}[\cite{LZ}]\label{8908}
Suppose that $Q$ is an admissible harmonic map.
If $u\in \mathcal{H}_{Q}^k$ then for $k=2,3$, then $\overline{u(\Bbb H^2)}$ is compact in $N=\Bbb H^2$.
\end{corollary}

The same arguments of proving Lemma \ref{swe345t} (see \cite{LZ}) give the following lemma which shows the heat tension field associated with the initial data of wave map is small under the assumptions of Theorem 1.1.
\begin{lemma}\label{nabla}
Let $  M=\Bbb H^2,  N=\Bbb H^2$. If $(u_0,u_1)$ with $u_0:$, $u_1(x)\in T_{u_0(x)}N$ for any $x\in M$
is the initial data to (1.2) satisfying (\ref{as3}) and
\begin{align*}
\|\nabla du_0\|_{L^2_x}+\|du_0\|_{L^2_x}\le M_0,
\end{align*}
then we have
\begin{align*}
\|\tau(u_0)\|_{L^2_x}\le C(M_0)\mu_1,
\end{align*}
where $\tau(u_0)$ denotes the tension filed.
\end{lemma}

The conditional global well-posedness theory is recalled below:
\begin{proposition}[\cite{EFE376}]\label{vskfheccc344fyeo}
For any initial data $(u_0,u_1)\in \mathcal{H}^3\times \mathcal{H}^2$, there exists $T>0$ depending only on $\|(u_0,u_1)\|_{\mathcal{H}^3\times\mathcal{H}^2}$ such that (\ref{z32ccc344fyeo}) has a unique local solution $(u,\partial_tu)\in C([0,T];\mathcal{H}^3\times\mathcal{H}^2)$.
Furthermore assume that for all $t\in[0,T_*)$ and some $C>0$ independent of $t\in[0,T_*)$ the solution $(u,\partial_tu)$ satisfies
\begin{align}\label{vskfhefyeo}
\|\nabla\partial_t u\|_{L^2_x}+\|\partial_t u\|_{L^2_x}+\|\nabla du\|_{L^2_x}+\|du\|_{L^2_x}\le C,
\end{align}
then $T_*=\infty$.
\end{proposition}

\subsection{Sobolev inequality}

\begin{lemma}[\cite{KINTY,45XDADER45P,6H8UR}]\label{KINsder45tyy}
Let $f\in C^{\infty}_c(\mathbb{H}^2;\Bbb R)$. Then for $1\le p\le q\le \infty$, $0<\theta<1$, $\frac{1}{p} - \frac{\theta }{2} = \frac{1}{q}$, the Gagliardo-Nirenberg inequality is
\begin{align}
 {\left\| f \right\|_{{L^q}}} \lesssim \left\| {\nabla f} \right\|_{{L^2}}^\theta \left\| f \right\|_{{L^p}}^{1 - \theta }. \label{yy676gvfrt}
\end{align}
The spectrum  gap inequality is known as
\begin{align}
{\left\| f \right\|_{{L^2}}} \lesssim {\left\| {\nabla f} \right\|_{{L^2}}}. \label{yy676frt}
\end{align}
For $\alpha>1$ the following inequality holds
\begin{align}
 {\left\| f \right\|_{{L^\infty }}} \lesssim {\left\| {{{\left( { - \Delta } \right)}^{\frac{\alpha }{2}}}f} \right\|_{{L^2}}}.\label{yy6frt}
\end{align}
The Riesz transform is bounded in $L^p$ for $1<p<\infty,$ i.e.,
\begin{align}
{\left\| {\nabla f} \right\|_{{L^p}}} \sim{\left\| {{{\left( { - \Delta } \right)}^{\frac{1}{2}}}f} \right\|_{{L^p}}} \label{y6frtsder45}.
\end{align}
And we recall the Sobolev inequality: Let $1<p,q<\infty$ and $\sigma_1,\sigma_2\in\Bbb R$  such that $\sigma_1-\sigma_2\ge n/p-n/q\ge0$. Then for all $f\in C^{\infty}_c(\Bbb H^n;\Bbb R)$
\begin{align*}
\|(-\Delta)^{\sigma_2}f\|_{L^q}\lesssim \|(-\Delta)^{\sigma_1}f\|_{L^p}.
\end{align*}
The diamagnetic inequality known also as Kato's inequality is as follows (e.g. [\cite{6H8UR}]): If $T$ is a tension filed defined on $\Bbb H^2$, then in the distribution sense it holds that
\begin{align}\label{y6frtsr45}
|\nabla|T||\le |\nabla T|.
\end{align}
\end{lemma}

The estimate of the heat semigroup in $\mathbb{H}^2$ is as follows.
\begin{lemma}[\cite{DoooooM,EFE376,6H8UR}]\label{8.5}
For $1\le r\le p\le\infty$, $\alpha\in[0,1]$, $1<q<\infty$, $0\delta_1>0$, the heat semigroup on $\Bbb H^2$ denoted by $e^{s\Delta}$ satisfies
\begin{align}
\|e^{s\Delta}f\|_{L^{2}_x}&\lesssim e^{-\frac{s}{4}}\|f\|_{L^{2}_x}\label{ytsr45}\\
\|e^{s\Delta}f\|_{L^{\infty}_x}&\lesssim e^{-\frac{s}{4}}s^{-1}\|f\|_{L^{1}_x}\label{tyihy6ys}\\
\|e^{s\Delta}(-\Delta)^{\alpha} f\|_{L^{q}_x}&\lesssim s^{-\alpha}e^{-s\delta_1}\|f\|_{L^{q}_x},\label{mm8}\\
\|e^{s\Delta}f\|_{L^p_x}&\lesssim s^{\frac{1}{p}-\frac{1}{r}}\|f\|_{L^{r}_x}.\label{huhu89}
\end{align}
And for $f\in L^2$ it holds that
\begin{align}\label{xcss434fer}
\int^{\infty}_0\|e^{s\Delta}f\|^2_{L^{\infty}_x}ds\lesssim \|f\|^2_{L^2}.
\end{align}
\end{lemma}

Estimates of the kernel of resolvent are recalled as follows:
\begin{lemma}[\cite{TYWE5E65GVU}]\label{Tomnagdfe4}
Denote the kernel of $(-\Delta_{\Bbb H^{n}}+\sigma^2-\frac{(n-1)^2}{4})^{-1}$ by $[{ }^n\widetilde{R}]_0(\frac{n-1}{2}+\sigma,x,y)$.
Then for $\Re \sigma\ge0$, $|\sigma|\ge 1$, $r\in (0,\infty)$, we have
\begin{align}\label{hyt543wsedx}
\left| {{[{ }^n\widetilde{R}]_0}(\frac{n-1}{2} + \sigma ,x,y)} \right| \le \left\{ \begin{array}{l}
 C\left| {\log r} \right|,\mbox{  }\mbox{  }\left| {r\sigma } \right| \le 1, n=2 \\
 C_nr^{2-n},\mbox{  }\left| {r\sigma } \right| \le 1, n\ge 3 \\
 C_n{\left| \sigma  \right|^{\frac{n-1}{2}-1}}{e^{ - (\frac{n-1}{2} + {\Re} \sigma )r}},\left| {r\sigma } \right| \ge 1 \\
 \end{array} \right.
\end{align}
and for $\Re \sigma\ge0$, $|\sigma|\ge 1$, $r\in (0,\infty)$, any $\epsilon\in(0,1)$, we have
\begin{align}\label{Hypwer45thiys}
\left| {{\partial_{\sigma}[{ }^n\widetilde{R}]_0}(\frac{n-1}{2} + \sigma ,x,y)} \right| \le \left\{ \begin{array}{l}
 C\left| {\log r} \right|,\mbox{  }\mbox{  }\left| {r\sigma } \right| \le 1, n=2 \\
 C_nr^{2-n},\mbox{  }\left| {r\sigma } \right| \le 1, n\ge 3 \\
 C_{n,\epsilon}{\left| \sigma  \right|^{\frac{n-1}{2}-1}}{e^{ - (\frac{n-1}{2} + {\Re} \sigma-\epsilon )r}},\left| {r\sigma } \right| \ge 1 \\
 \end{array} \right.
\end{align}
Moreover, for $\Re \sigma\ge0$, $|\sigma|\le 1$, $r\in (0,\infty)$ we have
\begin{align}\label{jiujinfa2}
\left| {{[{ }^n\widetilde{R}]_0}(\frac{n-1}{2} + \sigma ,x,y)} \right| \le \left\{ \begin{array}{l}
C\left| {\log r} \right|,\mbox{  }\mbox{  }\left| r \right| \le 1, \mbox{  }n=2 \\
C_n|r|^{2-n},\mbox{  }\mbox{  }\left| r \right| \le 1, \mbox{  }n\ge 3 \\
C_n{\left| \sigma  \right|^{\frac{n-1}{2}-1}}{e^{ - (\frac{n-1}{2} + {\Re}\sigma )r}},\left| r \right| \ge 1 \\
\end{array} \right.
\end{align}
and for $\Re \sigma\ge0$, $|\sigma|\le 1$, $r\in (0,\infty)$
\begin{align}\label{jiujinbfa2}
\left| {{\partial_{\sigma}[{ }^n\widetilde{R}]_0}(\frac{n-1}{2} + \sigma ,x,y)} \right| \le \left\{ \begin{array}{l}
C\left| {\log r} \right|,\mbox{  }\mbox{  }\mbox{  }\mbox{  }\mbox{  }\left| r \right| \le 1, \mbox{  }n=2 \\
C_n|r|^{2-n},\mbox{  }\mbox{  }\mbox{  }\mbox{  }\left| r \right| \le 1, \mbox{  }n\ge 3\\
C_{n,\epsilon}{\left| \sigma  \right|^{\frac{n-1}{2}-1}}{e^{ - (\frac{n-1}{2} + {\Re}\sigma )r}},\left| r \right| \ge 1 \\
\end{array} \right.
\end{align}
\end{lemma}

\begin{lemma}[\cite{6HDFRG}]\label{xuannv}
In the $\Bbb H^2$ case, for $\Re \sigma\ge0$, $|\sigma|\ge 1$, $r\in (0,\infty)$, we have
\begin{align}\label{jiujinfa3}
\left| {{\nabla _x}{[{ }^2\widetilde{R}]_0}(\frac{1}{2} + \sigma ,x,y)} \right| \lesssim \left\{ \begin{array}{l}
 {r^{ - 2}}{\left( {\sinh r} \right)^2}{\left( {{{\cosh }^2}r - 1} \right)^{ - \frac{1}{2}}},\mbox{  }\mbox{  }\left| {r\sigma } \right| \le 1 \\
 {\left| \sigma  \right|^{\frac{1}{2}}}{e^{ - (\frac{3}{2} + {\Re}\sigma )r}}{\left( {\sinh r} \right)^2}{\left( {{{\cosh }^2}r - 1} \right)^{ - \frac{1}{2}}},\left| {r\sigma } \right| \ge 1 \\
 \end{array} \right.
\end{align}
and for $\Re \sigma\ge0$, $|\sigma|\le 1$, $r\in (0,\infty)$ we have
\begin{align}\label{jiujinfa4}
\left| {{\nabla _x}{[{ }^2\widetilde{R}]_0}(\frac{1}{2} + \sigma ,x,y)} \right| \lesssim \left\{ \begin{array}{l}
 {r^{ - 2}}{\left( {\sinh r} \right)^2}{\left( {{{\cosh }^2}r - 1} \right)^{ - \frac{1}{2}}},\mbox{  }\left| r \right| \le 1 \\
 {\left| \sigma  \right|^{\frac{1}{2}}}{e^{ - (\frac{3}{2} + {\Re}\sigma )r}}{\left( {\sinh r} \right)^2}{\left( {{{\cosh }^2}r - 1} \right)^{ - \frac{1}{2}}},\left| r \right| \ge 1 \\
\end{array} \right.
\end{align}
\end{lemma}

\section{Appendix B. Two resolvent estimates in $L^p$}

\begin{lemma}\label{L1P}
We have the following high frequency $L^p $ estimates for resolvent of $-\Delta$ on $\Bbb H^2$.
\begin{itemize}
  \item (a) For $p\in (1,\infty)$, $\Re \sigma\ge 3$,
if $\Re \sigma\ge c|\sigma| $, then for any $\epsilon>0$
\begin{align}
\|(-\Delta+\sigma^2-\frac{1}{4})^{-1}f\|_{L^p_x}&\lesssim |\sigma|^{-2+\epsilon}\|f\|_{L^p_x}\label{LPP}\\
\|\nabla(-\Delta+\sigma^2-\frac{1}{4})^{-1}f\|_{L^p_x}&\lesssim |\sigma|^{-1}\|f\|_{L^p_x}£¬\label{LPQ}
\end{align}
where the implicit constant depends only on $c$.
  \item (b) For $p\in (1,\infty)$, $\Re \sigma\ge 0$,
if $\Re \left(\sigma^2-\frac{1}{4}\right)\ge 0$, then
\begin{align}
\|(-\Delta+\sigma^2-\frac{1}{4})^{-1}f\|_{L^p_x}&\lesssim \min\left(1,\frac{1}{\Re(\sigma ^{2})}\right)\|f\|_{L^p_x}\label{L2qPP}\\
\|\nabla(-\Delta+\sigma^2-\frac{1}{4})^{-1}f\|_{L^p_x}&\lesssim |\sigma|^{\frac{1}{2}}\|f\|_{L^p_x}.\label{L2qPQ}
\end{align}
\end{itemize}
\end{lemma}
\begin{proof}
We prove (\ref{LPP}) first. By (\ref{jiujinfa3}) and Young's inequality, it suffices to prove
\begin{align}
\int^{\frac{1}{|\sigma|}}_0\log r|rdr&\lesssim |\sigma|^{-2+\epsilon}\label{NM1}\\
\int^{\infty}_{{|\sigma|}^{-1}}\sigma^{-\frac{1}{2}}e^{-(\frac{1}{2}+\Re \sigma)r}\sinh {r} dr&\lesssim |\sigma|^{-2+\epsilon}.\label{NM2}
\end{align}
(\ref{NM1}) is direct. To verify (\ref{NM2}), we consider two cases:\\

When $r\in[\sigma^{-1},1]$ we have
\begin{align}
&\int^{1}_{{|\sigma|}^{-1}}|\sigma|^{-\frac{1}{2}}e^{-(\frac{1}{2}+\Re \sigma)r}\sinh r dr\lesssim\int^{1}_{{|\sigma|}^{-1}}|\sigma|^{-\frac{1}{2}}e^{-(\frac{1}{2}+\Re \sigma)r} r dr\\
&\lesssim\int^{|\sigma|}_{1}|\sigma|^{-\frac{5}{2}}e^{-\frac{\Re \sigma}{|\sigma| } \widetilde{r}}\widetilde{r} d\widetilde{r}\lesssim |\sigma|^{-\frac{5}{2}},
\end{align}
where in the last inequality we used $\Re \sigma\ge c|\sigma| $.

When $r\in[1,\infty)$ we have
\begin{align}
&\int^{\infty}_{1}|\sigma|^{-\frac{1}{2}}e^{-(\frac{1}{2}+\Re \sigma)r}\sinh r dr\lesssim e^{-\Re \sigma+1}\int^{\infty}_{1}|\sigma|^{-\frac{1}{2}}e^{-\frac{3}{2}r} \sinh r dr\\
&\lesssim e^{-\frac{1}{2}\Re \sigma}\lesssim | \sigma|^{-n}
\end{align}
for any $n>0$, where in the last inequality we used $\Re \sigma\ge c|\sigma| $.

Second, we prove (\ref{LPQ}). By (\ref{jiujinfa4}) and Young's inequality it suffices to prove
\begin{align}
|\sigma|^{\frac{1}{2}}\int^{\frac{1}{|\sigma|}}_0 {r^{ - 2}}{\left( {\sinh r} \right)^2}{\left( {{{\cosh }^2}r - 1} \right)^{ - \frac{1}{2}}}rdr&\lesssim |\sigma|^{-1}\label{NM12}\\
|\sigma|^{\frac{1}{2}}\int^{1}_{{|\sigma|}^{-1}}r^{-2}(\sinh  {r})^2(\cosh r ^2r-1)^{-\frac{1}{2}}e^{-(\frac{3}{2}+\Re\sigma)r}rdr&\lesssim |\sigma|^{-1}\label{NM22}\\
|\sigma|^{\frac{1}{2}}\int^{\infty}_{1}r^{-2}(\sinh  {r})^2(\cosh  ^2r-1)^{-\frac{1}{2}}e^{-(\frac{3}{2}+\Re\sigma)r}e^{r}dr&\lesssim |\sigma|^{-1}\label{NM23}.
\end{align}
(\ref{NM12})-(\ref{NM23}) follow by direct calculations as above. (\ref{L2qPQ}) is much easier and follows directly from estimating the LHS of (\ref{NM12})-(\ref{NM23}).

Now we prove (\ref{L2qPP}). By (\ref{mm8}), $e^{\delta_1 t}e^{t\Delta}$ is a contraction $C^0$ semigroup in $L^p_x$. Since the infinitesimal generator of $e^{\delta_1 t}e^{t\Delta}$ is $\delta_1+\Delta$. Then by Lumer-Phillips theorem or [ Corollary 3.6 , \cite{Pazy}], $\{z:\Re z>0\}\subset \rho(\Delta+\delta_1)$, and for such $z$, there  holds
\begin{align*}
\|(\Delta+\delta_1-z)^{-1}f\|_{L^p_x}&\lesssim \frac{1}{\Re(z)}\|f\|_{L^p_x}.
\end{align*}
Thus for $\Re \left(\sigma^2-\frac{1}{4}\right)\ge 0$ we have
\begin{align*}
\|(-\Delta+\sigma^2-\frac{1}{4})^{-1}f\|_{L^p_x}\lesssim \frac{1}{\Re(\delta_1+\sigma^2-\frac{1}{4})}\|f\|_{L^p_x},
\end{align*}
which gives (\ref{L2qPP}).
\end{proof}

\begin{lemma}\label{LP}
We have the following high frequency $L^p $ estimates for resolvent of $H$. For $p\in (1,\infty)$, $\Re \sigma\ge 3$,
if $\Re \sigma\ge c|\sigma| $, then for any $\epsilon>0$ there exists $K_0$ sufficiently large such that for all $|\sigma|\ge K_0$
\begin{align}
\|(H+\sigma^2-\frac{1}{4})^{-1}f\|_{L^p_x}&\lesssim |\sigma|^{-2+\epsilon}\|f\|_{L^p_x}\label{LPY}.
\end{align}
where the implicit constant depends only on $c$.
\end{lemma}
\begin{proof}
The proof is an easy application of resolvent identity and Lemma \ref{L1P}. In fact formally one has
\begin{align*}
(H+\sigma^2-\frac{1}{4})^{-1}=(-\Delta+\sigma^2-\frac{1}{4})^{-1}\left(I+W(-\Delta+\sigma^2-\frac{1}{4})^{-1}\right)^{-1}.
\end{align*}
By Lemma \ref{L1P} it suffices to prove
\begin{align*}
\|W(-\Delta+\sigma^2-\frac{1}{4})^{-1}\|_{ (L^p\to L^p)}\le \frac{1}{2}.
\end{align*}
By Lemma \ref{L1P}, this is easy to obtain by letting $|\sigma|\ge K_0\gg 1$:
\begin{align*}
&\|W(-\Delta+\sigma^2-\frac{1}{4})^{-1}\|_{ (L^p\to L^p)}\\
&\lesssim (\|\mathcal{A}\|_{L^{\infty}}+\|V\|_{L^{\infty}})\|(-\Delta+\sigma^2-\frac{1}{4})^{-1}\|_{ (L^p\to L^p)}\\
&\lesssim  CK^{-1+\epsilon}_0,
\end{align*}
where we applied Proposition {4.1} to bound the potentials in $W$.
\end{proof}

\section{Appendix C. Harnack inequality for linear heat equations and Bochner inequality for heat flows}

The following version of Harnack inequality for linear heat equations on complete manifolds was widely used in heat flow literature.
\begin{lemma}\label{density}
Suppose that $f$ is a nonnegative function which satisfies
\begin{align*}
\partial_t f-\Delta f\le 0,
\end{align*}
then it is known in the heat flow literature that for $t\ge1$
\begin{align*}
f(x,t)\le \int^t_{t-1}\int_{B(x,1)}f(y,s){\rm{dvol_y}}ds.
\end{align*}
\end{lemma}

We collect the Bochner inequalities for heat flows in the following two lemmas.
\begin{lemma}
If $(u,\partial_tu)$ solves (\ref{z32ccc344fyeo}) in $\mathcal{X}_T$, then
\begin{align}
&{\partial _s}{\left| {{\nabla _t}{\partial _s}\widetilde{u}} \right|^2} - \Delta {\left| {{\nabla _t}{\partial _s}\widetilde{u}} \right|^2} + 2{\left| {\nabla {\nabla _t}{\partial _s}\widetilde{u}} \right|^2} \lesssim \left| {{\nabla _t}{\partial _s}\widetilde{u}} \right|\left| {\nabla {\partial _s}\widetilde{u}} \right|\left| {{\partial _t}\widetilde{u}} \right|\left| {d\widetilde{u}} \right| \nonumber\\
&+ \left| {\nabla {\partial _t}\widetilde{u}} \right|\left| {{\partial _s}\widetilde{u}} \right|\left|{\nabla _t}{\partial _s}\widetilde{u}\right|\left| {d\widetilde{u}} \right| + \left| {\nabla d\widetilde{u}} \right|\left| {{\partial _t}\widetilde{u}} \right|\left|{\nabla _t}{\partial _s}\widetilde{u}\right|\left| {{\partial _s}\widetilde{u}} \right| +{\left| {{\nabla _t}{\partial _s}\widetilde{u}} \right|^2}{\left| {d\widetilde{u}} \right|^2}  ,\label{Hwer45thiys}
\end{align}
and it holds
\begin{align}
&{\partial _s}{\left| {{\nabla}{\partial _s}\widetilde{u}} \right|^2} - \Delta {| {{\nabla}{\partial _s}\widetilde{u}} |^2} + 2{\left| { {\nabla^2 }{\partial _s}\widetilde{u}} \right|^2} \lesssim {| {{\nabla}{\partial _s}\widetilde{u}} |^2}{| {d\widetilde{u}} |^2} + |\nabla\partial_s\widetilde{u}|^2\nonumber\\
&+ | {{\nabla}{\partial _s}\widetilde{u}} || { {\partial _s}\widetilde{u}} || {{\nabla d}\widetilde{u}}|| {d\widetilde{u}} |+| {\nabla \partial_s\widetilde{u}}|| {{\partial _s}\widetilde{u}} ||{d}\widetilde{u}|^3.\label{icionm}
\end{align}
Moreover we have
\begin{align}
&{\partial _s}{| {{\nabla}{\partial _t}\widetilde{u}} |^2} - \Delta {| {{\nabla}{\partial _t}\widetilde{u}} |^2} + 2{| { {\nabla^2 }{\partial _t}\widetilde{u}} |^2} \lesssim |\nabla\partial_t\widetilde{u}|^2|d\widetilde{u}|^2+ |\nabla\partial_t\widetilde{u}|^2\nonumber\\
&+|\partial_s\widetilde{u}||d\widetilde{u}|^2|\nabla\partial_t\widetilde{u}|+ | {{\nabla}{\partial _t}\widetilde{u}} || { {\partial _t}\widetilde{u}} || {{\nabla d}\widetilde{u}}|| {d\widetilde{u}}|+| {\nabla \partial_t\widetilde{u}}|| {{\partial _t}\widetilde{u}} ||{d}\widetilde{u}|^3.\label{y45aswe345t}
\end{align}
\end{lemma}

\begin{lemma}
If $(u,\partial_tu)$ solves (\ref{z32ccc344fyeo}) in $\mathcal{X}_T$, then we have
\begin{align}
&\partial _s| {{\nabla}^2d\widetilde{u}}|^2 - \Delta | {{\nabla}^2d\widetilde{u}}|^2 + 2|  \nabla^3 d\widetilde{u}|^2\nonumber\\
&\lesssim |\nabla^2d\widetilde{u}|^2|d\widetilde{u}|^2+|\nabla^2d\widetilde{u}|^2+|\nabla d\widetilde{u}||\nabla^2 d\widetilde{u}|^2+|\nabla d\widetilde{u}|^2|\nabla^2d\widetilde{u}||d\widetilde{u}|\nonumber\\
&+|d\widetilde{u}|^3|\nabla d\widetilde{u}||\nabla^2d\widetilde{u}|\label{tsreinhy45}
\end{align}
and
\begin{align}
&\partial _s| \nabla^2{\partial _s}\widetilde{u}|^2 - \Delta | \nabla^2{\partial _s}\widetilde{u}|^2 + 2| \nabla^3 \partial _s\widetilde{u} |^2\nonumber\\
&\lesssim |\nabla^2\partial_s\widetilde{u}|^2(|d\widetilde{u}|^2+1)+|\partial_s\widetilde{u}|^2|\nabla^2\partial_s\widetilde{u}||\nabla d\widetilde{u}|+|\partial_s\widetilde{u}||d\widetilde{u}||\nabla \partial_s\widetilde{u}||\nabla^2\partial_s\widetilde{u}|\nonumber\\
&+|\nabla^2 \partial_s\widetilde{u}|^2|\nabla d\widetilde{u}|+|d\widetilde{u}||\nabla d\widetilde{u}||\nabla^2\partial_s\widetilde{u}||\nabla\partial_s\widetilde{u}|+|\nabla^2\partial_s\widetilde{u}|^2|d\widetilde{u}||\partial_s\widetilde{u}|
+\nonumber\\
&+
|\nabla^2d\widetilde{u}||d\widetilde{u}||\partial_s\widetilde{u}||\nabla^2\partial_s\widetilde{u}|
+|d\widetilde{u}||\nabla d\widetilde{u}||\nabla\partial_s\widetilde{u}||\nabla^2\partial_s\widetilde{u}|
.\label{tsrinhy45}
\end{align}
\end{lemma}

\textit{rikudosennin@163.com}\\
\noindent{\small{Academy of Mathematics and Systems Science (AMSS)}}\\
\noindent{\small{Chinese Academy of Sciences (CAS)}}\\
\noindent{\small{Beijing 100190, P. R. China}}\\

\end{document}